\documentclass[11pt]{amsart}

\usepackage{amsmath, amsthm, amssymb, amscd, amsfonts}
\usepackage{fontsmpl}
\usepackage{fontsmpl,layout}
\usepackage[all]{xy}             
\usepackage[latin1]{inputenc}

\newtheorem{mtheorem}{Theorem}
\newtheorem{lema}{Lemma}[section]
\newtheorem{theorem}[lema]{Theorem}
\newtheorem{cor}[lema]{Corollary}

\newtheorem{prop}[lema]{Proposition}

\theoremstyle{definition}
\newtheorem{definition}[lema]{Definition}

\theoremstyle{remark}
\newtheorem{obs}[lema]{Remark}

\newcommand{\quot}{\mathcal{QUOT}}

\newcommand\id{\operatorname{id}}

\newcommand\co{\operatorname{co}}

\newcommand\cop{\operatorname{cop}}

\newcommand\Alg{\operatorname{Alg}}

\newcommand\Ker{\operatorname{Ker}}
\newcommand\Img{\operatorname{Im}}

\newcommand{\eps}{\varepsilon}
\newcommand{\ot}{\otimes}
\newcommand{\com}{\Delta}

\newcommand\Res{\operatorname{Res}}
\newcommand\res{\operatorname{res}}
\newcommand{\g}{\mathfrak g}

\newcommand{\Tte}{\textbf{T}}

\newcommand{\D}{{\mathcal D}}

\newcommand{\Ee}{{\mathcal E}}
\newcommand{\Ff}{{\mathcal F}}

\newcommand{\Oc}{{\mathcal O}}

\newcommand\Lie{\operatorname{Lie}}

\def\NN{\mathbb{N}}
\def\ZZ {\mathbb{Z}}

\def\CC{\mathbb{C}}

\def\GG{\mathbb{G}}

\def\JJ{\mathcal{J}}
\def\II{\mathcal{I}}
\def\OO{\mathcal{O}}
\def\SS{\mathcal{S}}
\def\S{\mathbb{S}}

\def\T{\mathbf{T}}
\def\Tt{\mathbb{T}}
\def\z{\mathbf{Z}}
\def\w{\mathbf{W}}
\def\k{\Bbbk}
\def\x{\mathbf{X}}
\def\y{\mathbf{Y}}

\def\Z{\mathcal{Z}}

\def\D{\mathcal{D}}

\def\J{\mathfrak{J}}

\def\lieg{\mathfrak{g}}
\def\liegl{\mathfrak{gl}}
\def\lieb{\mathfrak{b}}

\def\liesl{\mathfrak{sl}}

\def\QEabg{U_{\alpha,\beta}(\liegl_{n})}
\def\QEabl{U_{\alpha,\beta}(\lgot)}

\def\Ga{\Gamma}

\def\Oe{\OO_{\epsilon}(G)}
\def\Oea{\Oc_{\alpha,\beta}(GL_{n})}
\def\Oeal{\Oc_{\alpha,\beta}(L)}

\def\qe{{\bf u}_{\epsilon}(\lieg)}
\def\qeo{{\bf u}_{\epsilon}(\tilde{\liel}_{0})}
\def\qab{{\bf u}_{\alpha,\beta}(\liegl_{n})}
\def\qabl{{\bf u}_{\alpha,\beta}(\liel)}
\def\qablhat{\hat{{\bf u}}_{\alpha,\beta}(\liel)}
\def\qablohat{\hat{{\bf u}}_{\alpha,\beta}(\liel_{0})}

\def\qabhat{\hat{{\bf u}}_{\alpha,\beta}(\liegl_{n})}

\def\liel{\mathfrak{l}}
\def\lgot{\mathfrak{l}}

\def\Ol{\Oc_{\epsilon}(L)}

\def\ql{{\bf u}_{\epsilon}(\lgot)}

\def\pf{\begin{proof}}
\def\epf{\end{proof}}
\def\HZe{\mathcal{HZ}}

\theoremstyle{plain}

\begin{document}




\title[quantum subgroups of $GL_{\alpha,\beta}(n)$]
{Quantum subgroups of $GL_{\alpha,\beta}(n)$}

\author[gast\'on andr\'es garc\'\i a]
{Gast\' on Andr\' es Garc\'\i a}

\address{FaMAF-CIEM, Universidad Nacional de C\'ordoba
Medina Allende s/n, Ciudad Universitaria, 5000 C\' ordoba, Rep\'
ublica Argentina.} \email{ggarcia@mate.uncor.edu}

\thanks{
The work was
partially supported by DAAD, CONICET, ANPCyT, Secyt (UNC) and TWAS.}

\date{\today.}

\begin{abstract}
Let $\alpha, \beta \in \CC \smallsetminus \{0\}$ and $\ell \in \NN$, odd with 
$\ell \geq 3$. 
We determine all
Hopf algebra quotients of 
the quantized coordinate algebra $\Oc_{\alpha, \beta}(GL_{n})$
when $\alpha^{-1}\beta$ is a primitive $\ell$-th root of unity 
and $\alpha, \beta$ satisfy
certain mild conditions, and we caracterize 
all finite-dimensional quotients
when $\alpha^{-1}\beta$ is not a root of unity. As a byproduct
we give a new family of non-semisimple 
and non-pointed Hopf algebras with non-pointed duals
which are
quotients of $\Oc_{\alpha, \beta}(GL_{n})$.
\end{abstract}

\maketitle

\centerline{\emph{Dedicated to Nicol\'as Andruskiewitsch in his
50th birthday}.}

\tableofcontents


\section{Introduction}
One-parameter quantizations of the general linear
group $GL_{n}$ are well-known objects that arises 
as dual Hopf algebras of the Drinfeld-Jimbo quantized enveloping
algebras $U_{q}(\liegl_{n})$ and were studied by many authors, 
for example \cite{FRT}, \cite{HH}, \cite{Ma}, \cite{PW}, \cite{TT} 
and \cite{Tk2}. The problem of constructing families of multiparameter
quantum groups was first raised in \cite{Ma2} and subsequently
treated in \cite{Tk}, \cite{Re}, \cite{OY}, \cite{H}, \cite{LS}, \cite{AB}
among others. Two-parameter deformations of $\Oc(GL_{n})$ 
were introduced by M. Takeuchi in \cite{Tk}, see also \cite{Ko}, and
they were studied in detail in \cite{DPW}.
In \cite{Tk} a two-parameter deformation of the enveloping algebra
$U(\liegl_{n})$ is also defined and it is proven 
that there exists a Hopf pairing between them.
In the case of the one-parameter deformation, one
can prove that this pairing is perfect. This is
way the Hopf algebras $\Oc_{q}(GL_{n})$ are dual to
$U_{q}(\liegl_{n})$. 
A family depending on more parameters has been constructed 
independently by Sudbery \cite{S} and Reshetikhin \cite{Re}. It is shown in 
\cite{AST}
that this family can be obtained and characterized by a 
construction of Manin and they explain how the algebras in the family are
twists of $\Oc_{q}(GL_{n})$ by 2-cocycles. We have chossen
to study the two-parameter deformation given by Takeuchi
simply because they consist of a family of objects which cannot be obtained
by 2-cocycle deformations of $\Oc_{q}(GL_{n})$, see Remark 
\ref{rmk:valores-de-a-b} $(c)$.

\par Several authors,
among them 
Benkart and Witherspoon \cite{BW, BW2}, Jing
\cite{Ji} and Kulish \cite{Ku}, defined also two-parameter deformations 
of $U(\liegl_{n})$ 
which are particular cases of the multiparameter
deformations, see for example \cite{CM2} and 
in particular \cite{S}, where they are
defined using a pairing with a multiparameter deformation
of the coordinate ring of $GL_{n}$.  
These two-parameter defomormations of $U(\liegl_{n})$ 
are closely related to each other,
for example Takeuchi's deformation $U_{\alpha,
\beta}(\liegl_{n})$ is isomorphic as algebra 
to the deformation $U_{r,s^{-1}}(\liegl_{n})$  given by Benkart and 
Witherspoon but as coalgebra
they have the co-opposite coproduct. 
Recently, N. Hu and Y. Pei \cite{HP} defined two-parameter
deformations of $U(\lieg)$ for any semisimple Lie algebra and showed
that they can be realized as Drinfeld doubles. This generalizes 
previous results of Benkart and Witherspoon for type $A$, 
Bergeron, Gao and Hu for type $B$, $C$ and $D$ \cite{BGH}, \cite{BGH2}, 
Hu and Shi for type $G$ \cite{HS} and Bai and Hu for type $E$ \cite{BH}. 

\par In this paper we determine the Hopf algebra quotients
of the two parameter quantization $\Oea$ of the coordinate
algebra on $GL_{n}$ 
when $\alpha^{-1}\beta$ is a primitive $\ell$-th root of unity,
$\ell \in \NN$,  odd  with $l\geq 3$
 and $\alpha, \beta$ satisfy
certain mild conditions, and we caracterize all finite-dimensional quotients
when $\alpha^{-1}\beta$ is not a root of unity. 
As a consequence, 
we give in Thm. \ref{thm:pulenta} a new family of non-semisimple
and non-pointed Hopf algebras with non-pointed duals,
which are
quotients of $\Oc_{\alpha, \beta}(GL_{n})$ and cannot
be obtained as quotients of $\Oe$, 
with $G$ a connected,
simply connected,  simple complex Lie group, $\epsilon$
a primitive $s$-th root of unity, see \cite{AG}. 

\par It is crucial
for the determination of the quotients the relation between
$U_{\alpha,\beta}(\liegl_{n})$ and $\Oea$, as well as
some known facts about the pairing between them and
the center of $U_{\alpha,\beta}(\liegl_{n})$. For this reason,
we relay on results of \cite{BKL, BW, BW2, BW3} and \cite{DPW}.

\par In order to study quantum subgroups of
more general quantum groups, we believe that 
it would be necessary to define first, as in the
case of one-parameter deformations \cite{L}, the rational form of
multiparameter
deformations of coordinate rings of reductive or more general 
algebraic groups. This involves the study of the representation 
theory of these quantum groups at roots of $1$, since
the quantized coordinate rings would be generated as algebras by
the matrix coefficients of representations of type 1. 
One may also use for the definition 
of these type of quantum groups 
the pairing between the quantized coordinate rings
and the quantized enveloping algebras, but in the case where
the parameters are roots of unity, the pairing is degenerate, 
which makes the representation theory more complicated.
For multiparameter deformations of
other simple Lie algebras,
see \cite{BGH}, \cite{AB}. 

\par The problem of determining 
the quantum subgroups of a quantum group 
was first considered by P. Podle\'s
\cite{podles} for quantum $SU(2)$ and $SO(3)$. Then, the
characterization of all finite-dimensional Hopf algebra
quotients of the quantized coordinate algebra $\Oc_q(SL_N)$ was
obtained by Eric M\"uller \cite{Mu}. For more general simple
groups, all possible quotients of $\Oe$ were determined in
\cite{AG} in the case where the parameter $\epsilon$
is a root of unity, generalizing the results of M\"uller.
They are parameterized by data $\D= (I_{+}, I_{-}, N,
\Ga, \sigma, \delta)$ where $I_{+}$ and $-I_{-}$ are 
subsets of the basis of
a fixed root system of $\Lie(G)= \g$, 
$N$ is a finite abelian group related
to $I = I_{+} \cup -I_{-}$,
$\Ga$ is an algebraic group, 
$\sigma: \Ga \to L$ is an injective
morphism of algebraic groups, 
where $L\subseteq G$ is a connected
algebraic subgroup associated to $I$ and
$\delta: N \to \widehat{\Ga}$ is a
group map into the character group of $\Ga$.
The corresponding quotient $A_{\D}$ fits into a commutative diagram 
with exact rows:
\begin{equation}\label{diag:const-def-datum}
\xymatrix{1 \ar[r]^{} & \Oc(G) \ar[r]^{\iota} \ar[d]_{^{t}\sigma}
& \Oe \ar[r]^{\pi}
\ar[d]^{q_{\D}} & \qe^{*} \ar[r]^{}\ar[d]^{} & 1\\
1 \ar[r]^{} &   \Oc(\Ga)  \ar[r]^{\hat{\iota}} &
A_{\D}\ar[r]^{\hat{\pi}}& H\ar[r]^{}& 1,}
\end{equation}
where $\qe$ is the Frobenius-Lusztig kernel of $\g$ and
$H^{*}$ is a Hopf subalgebra of $\qe$ determined by the triple
$(I_{+}, I_{-}, N)$. In particular, the quotients $A_{\D}$ fits into
a central exact sequence of Hopf algebras. 

\par We prove that the quotients of $\Oea$ follow the
same pattern when $\alpha^{-1}\beta$ is a 
primitive $\ell$-th root of unity
and $\alpha, \beta$ satisfy
certain mild conditions, see Thm. \ref{teo:constrfinal}.
If $\alpha^{-1}\beta$
is not a roof of unity, the finite-dimensional
quotients are
the function algebras of finite subgroups of the diagonal torus
in $GL_{n}(\k)$. The definition
of a subgroup datum for $\Oea$ is the following:

\begin{definition}\label{def:subgroupdatum}
A \emph{subgroup datum} of the quantum group $\Oea$ 
is a collection $\D= (I_{+}, I_{-}, N, \Ga,
\sigma, \delta)$ where
\begin{itemize}
\item[$\bullet$] $I_{+}, I_{-}\subseteq \{1,\ldots , n-1\}$.
These subsets determine an algebraic subgroup $ L $ of
$GL_{n}$ consisting in block matrices whose nonzero
blocks are in the diagonal, see Remark \ref{rmk:definicion-L}.

\medbreak \item[$\bullet$] $N$ is a subgroup of
$\widehat{\mathbb{T}}$, see Remark \ref{obs:hat-t}.

\medbreak \item[$\bullet$] $\Ga$ is an algebraic group.

\medbreak \item[$\bullet$] $\sigma: \Ga \to L$ is an injective
homomorphism of algebraic groups.

\medbreak \item[$\bullet$] $\delta: N \to \widehat{\Ga}$ is a
group homomorphism.
\end{itemize}
\end{definition}

\noindent If $\Ga$ is finite, we call $\D$ a \emph{finite subgroup
datum}.  The following theorem is the main result of
the paper. Part $(a)$ is completely analogous to \cite[Thm. 4.1]{Mu}
with almost the same proof and part $(b)$ is also analogous to
\cite[Thm. 2.17]{AG}, but its proof is different since the quantum
group $\Oea$ is given by generators and relations. In particular,
several technical lemmata will be needed to prove this part
of the theorem.

\begin{mtheorem}\label{thm:main}
Let $q: \Oea \to A$ be a surjective Hopf algebra map.
\begin{enumerate}
 \item[$(a)$] If $\alpha^{-1}\beta$ is not a root of unity
  and $\dim A$ is finite, then $A$ is a function algebra of 
  a finite subgroup of the diagonal torus in $GL_{n}(\k)$.
 \item[$(b)$] If $\alpha^{-1}\beta$ is a primitive $\ell$-th
  root of unity with $\alpha^{\ell} = 1 = \beta^{\ell}$ then
  there is a bijection between
    \begin{enumerate}
      \item[$(i)$] Hopf algebra quotients $q: \Oea \to A$
      \item[$(ii)$] Subgroup data of $\Oea$ up to equivalence.
    \end{enumerate}
\end{enumerate}
\end{mtheorem}

In Section \ref{sec:a-b-not-root} we give the proof of part $(a)$ and
in Section \ref{sec:a-b-root} we give the proof of part $(b)$. 
Specifically, in Sec. \ref{subsec:a-b-root-constr}
we carry out the construction
of a quotient $A_{\D}$ of $\Oea$ starting from a subgroup datum $\D$,
see Thm. \ref{teo:constrfinal}.
In Section \ref{subsec:a-b-root-charac}, we attach a subgroup datum $\D$ to
an arbitrary Hopf algebra quotient $A$ and prove that $A_{\D}\simeq
A$ as quotients of $\Oea$. 
Finally, in Sec.
\ref{sec:equivalence}, we study the lattice of quotients $A_{\D}$.
This concludes the proof of the theorem.

As a consequence, any quotient $A_{\D}$ also fits into a 
commutative diagram with exact rows:
\begin{equation}\label{diag:const-def-datum-2}
\xymatrix{1 \ar[r]^{} & \Oc(GL_{n}) \ar[r]^{\iota} \ar[d]_{^{t}\sigma}
& \Oea \ar[r]^{\pi}
\ar[d]^{q_{\D}} & \qabhat^{*} \ar[r]^{}\ar[d]^{} & 1\\
1 \ar[r]^{} &   \Oc(\Ga)  \ar[r]^{\hat{\iota}} &
A_{\D}\ar[r]^{\hat{\pi}}& H\ar[r]^{}& 1,}
\end{equation}
where $\qabhat$ is a quotient  of the 
restricted quantum group of $\liegl_{n}$ defined in
\cite{BW}, and
$H^{*}$ is a Hopf subalgebra of $\qabhat$ determined by the triple
$(I_{+}, I_{-}, N)$. In particular, the quotient $A_{\D}$ fits into
a central exact sequence of Hopf algebras. 
It is not known if a kind of converse is true, that is, 
if a Hopf algebra is a central extension and it satisfies some
additional but specific properties, then it is a quotient of a quantum
group. An example of a specific property, 
for instance, is being finite-dimensional and
generated
by a simple subcoalgebra of dimension 4 stable by the antipode.
This fact was proved by \c Stefan \cite{stefan} and it is used
with profit in the classification of Hopf algebras
of small dimension, see for example \cite{GV}, \cite{N}.

\par The paper is organized as follows. We recall in 
Section \ref{sec:prelim} some
known facts about Hopf algebras, central extensions of Hopf algebras
and PI-Hopf triples. In Section \ref{sec:two-par-def} we recall the
definition of the two-parameter deformation of
the coordinate ring of $GL_{n}$, the universal enveloping
algebra of $\liegl_{n}$, the pairing between them and some
results due to Kharchenko \cite{K} and Benkart and 
Witherspoon \cite{BW} 
on a PBW-basis of $U_{\alpha,
\beta}(\liegl_{n})$. As already mentioned, we prove Thm. \ref{thm:main}
in Sec. \ref{sec:a-b-not-root} and Sec. \ref{sec:a-b-root}, and
we end the paper by giving some properties and 
relations between distinct 
quantum subgroups in the case where the parameters are
roots of unity. As a byproduct we obtain 
a new family of finite-dimensional 
non-semisimple and non-pointed Hopf
algebras with non-pointed duals, see Thm. \ref{thm:pulenta}.
\subsection*{Acknowledgements}
Research of this paper was begun when the author was visiting the
Mathematisches Institut der Ludwig-Maximilians Universit\" at
M\"unchen under the support of the DAAD and CONICET. 
He thanks H.-J. Schneider and the people of the institute
for their warm hospitality. The author also wishes to thank I.
Heckenberger and N. Andruskiewitsch for fruitful discussions. 

\section{Preliminaries}\label{sec:prelim}
\subsection{Conventions}
We work over an algebraically closed field $\k$ of characteristic
zero and by $\k^{\times}$ we denote the group of units of $\k$.
We write $\GG_{\ell}$ for the group
of $\ell$-th roots of unity.
Our references for the theory of Hopf algebras are \cite{Mo} and
\cite{Sw}, for Lie algebras \cite{Hu} and for quantum groups
\cite{J} and \cite{BG}. If $\Ga$ is a group, we denote by
$\widehat{\Ga}$ the character group. The
antipode of a Hopf algebra $H$ is denoted by $\mathcal S$. The
Sweedler notation is used for the comultiplication of $H$ but
dropping the summation symbol. The set of group-like elements of a
coalgebra $C$ is denoted by $G(C)$. We also denote by $C^{+} = \Ker
\eps$ the augmentation ideal of $C$, where $\eps: C\to \k$ is the
counit of $C$. Let $A\xrightarrow{\pi} H$ be a Hopf algebra map,
then $ A^{\co H}= A^{\co \pi} =\{a \in A|\ (\id\otimes \pi)\Delta
(a) = a\otimes 1\}$ denotes the subalgebra of right coinvariants and
$ ^{\co H}A= \ ^{\co \pi}A =\{a \in A|\ (\pi\otimes \id)\Delta (a) =
1\otimes a\}$ denotes the subalgebra of left coinvariants.

\par Let $H$ be a Hopf algebra,
$A$ a right $H$-comodule algebra
with structure map $ \delta: A\to A \ot H$,
$a \mapsto a_{(0)}\ot a_{(1)}$ and
$B = A^{\co H} $. The extension $B\subseteq A$
is called a Hopf Galois extension or $H$-Galois
if the canonical map $\lambda: A\ot_{B} A\to A\ot H$,
$a\ot b \mapsto  ab_{(0)}\ot b_{(1)}$
is bijective. 
 
\begin{definition}\label{def:hopf-pairing}
A {\it Hopf pairing} between two Hopf algebras $U$ and $H$ 
is a bilinear form $(-,-): H \times U \to
\mathcal{R}$ such that, for all $u,\ v \in U$ and $f,\ h \in H$,

\begin{align*}
& (i)\qquad  (h,uv) = (h_{(1)},u)(h_{(2)},v);\qquad & (iii)&\qquad
(1,u) =
\eps(u);\\
&(ii)\qquad (fh,u) = (f,u_{(1)})(h,u_{(2)}); \qquad & (iv)& \qquad
(h,1) = \eps(h).
\end{align*}
\end{definition}
Hence $(h,\SS(u)) = (\SS(h),u)$, for all $u \in
U$, $h \in H$. Given a Hopf pairing, one has Hopf algebra maps $U
\to H^{\circ}$ and $H \to U^{\circ}$, where $H^{\circ}$ and
$U^{\circ}$ are the Sweedler duals. The pairing is called {\it
perfect} if these maps are injections.

\subsection{Central extensions of Hopf
algebras}\label{subsec:centralextensions}
We recall some results on quotients and extensions of Hopf algebras.

\begin{definition}\cite{AD}\label{def:sucex}
A sequence of Hopf algebras maps $1 \to B \xrightarrow{\iota} A
\xrightarrow{\pi} H \to 1$, where $1$ denotes the Hopf algebra $\Bbbk$,
is {\it exact} if $\iota$ is injective, $\pi$ is surjective, $\Ker
\pi = AB^{+}$ and $B =\ ^{\co \pi} A$.
\end{definition}

If the image of $B$ is central in $A$, then $A$ is called a {\it
central} extension of $B$. We shall use the following result.

\begin{prop}\cite[Prop. 2.10]{AG}\label{prop:cociente1}
Let $A$ and $K$ be Hopf algebras, $B$ a central Hopf subalgebra of
$A$ such that $A$ is left or right faithfully flat over $B$ and
$p: B \to K$ a Hopf algebra epimorphism. Then $H = A/AB^{+}$ is a
Hopf algebra and $A$ fits into the exact sequence $1\to B
\xrightarrow{\iota} A \xrightarrow{\pi} H \to 1$. If we set $\JJ =
\Ker p \subseteq B$, then $(\JJ) = A\JJ $ is a Hopf ideal of $A$
and $A_{p} := A / (\JJ)$ is the pushout given by the following
diagram:
$$\xymatrix{B
\ar[r]^{\iota} \ar[d]_{p} & A \ar[d]^{q}\\
K\ar[r]_(.4){j} & A_{p} .}$$
Moreover, $K$ can be identified with a central Hopf
subalgebra of $A_{p}$ and $A_{p}$ fits into the exact sequence $
1\to K \to A_{p} \to H \to 1$. \qed
\end{prop}

\begin{obs}\label{rmk:decocaext} Let $A$ and $B$ be as in Prop.
\ref{prop:cociente1},
then the following diagram of central exact sequences is
commutative.
 \begin{equation}\label{ext2} \xymatrix{1\ar[r]^{} &
  B\ar[r]^{\iota} \ar[d]_{p} &
  A\ar[r]^{\pi}\ar[d]^{q} & H\ar[r]^{}\ar@{=}[d]_{} & 1\\
  1\ar[r]^{}& K\ar[r]^{j} &A_{p} \ar[r]^{\pi_{p}} & H \ar[r]^{}&
  1.}
 \end{equation}
\end{obs}

The following general fact is due to Masuoka, see
\cite[Lemma 1.14]{AG}.

\begin{lema}
Let $H$  be a bialgebra over an arbitrary commutative ring, and let
$A$,  $A'$  be right $H$-Galois extensions over a common algebra $B$
of $H$-coinvariants. If $A'$  is right $B$-faithfully flat, then
any $H$-comodule algebra map $\theta : A \to A'$ that is
identical on $B$  is an isomorphism. \qed 
\end{lema}

Recall that a $\k$-algebra $A$ is called \textit{affine} if it
is finitely generated as an algebra and a ring $R$ is called
a \textit{polynomial identity ring} or
\textit{PI-ring} for short, if there exists a monic polynomial $f$ in
the free algebra $\ZZ\langle X \rangle$ on a set
$X = \{x_{1},\ldots, x_{m} \}$ such that
$f(r_{1},\ldots, r_{m}) = 0$ for all $r \in R$.

\begin{definition}\label{def:pi-hopf-triple}\cite[Def. III.4.1]{BG}
A \textit{PI-Hopf triple} $(B,H,\overline{H})$ over $\k$
consists of three Hopf algebras such that
 \begin{itemize}
  \item[$(i)$] $H$ is a $\k$-affine $\k$-Hopf algebra.
  \item[$(ii)$] $B$ is a central Hopf subalgebra of $H$
  which is a domain and
    such that $H$ is a finitely-generated $B$-module.
  \item[$(iii)$] $\overline{H}:= H/B^{+}H$ is the finite-dimensional
    Hopf algebra quotient.
 \end{itemize}
\end{definition}

We end this section with the following results.

\begin{lema}\label{lema:triple-affine}
Let $H$ be a $\k$-affine $\k$-Hopf algebra, $B$ a
central Hopf subalgebra of $H$ which is a domain and
such that $H$ is a finitely-generated $B$-module and denote
by $\overline{H}:= H/B^{+}H$ the finite-dimensional
Hopf algebra given by the quotient. Then
 \begin{itemize}
   \item[$(i)$] \cite[Lemma III.4.2]{BG} $B$ is
    an affine $\k$-algebra. Thus
    $H$ is a noetherian PI-algebra and $\Z(H)$ is an affine algebra.
   \item[$(ii)$] \cite[Thm. III.4.5]{BG} $H$ is a
   finitely generated projective $B$-module.
   \item[$(iii)$]\cite[Lemma III.4.6]{BG} $B\subseteq H$ is a
    faithfully flat
     $\overline{H}$-Galois extension. In particular,
     $B=\ ^{\co \overline{H}}H = H^{\co \overline{H}}$. \qed
\end{itemize}
\end{lema}

\begin{obs}\label{rmk:PI-hopf-triple-sus-ex}
By Lemma \ref{lema:triple-affine} $(iii)$ and \cite[Prop.
3.4.3]{Mo}, any PI-Hopf triple $(B,H,\overline{H})$ gives rise to a
central extension of Hopf algebras -- see Def. \ref{def:sucex}:
 $$1 \to B \xrightarrow{\iota} H
\xrightarrow{\pi} \overline{H} \to 1.$$
\end{obs}


\section{Two-parameter deformations of classical objects}
\label{sec:two-par-def}
In this section we recall the definition and some basic
properties of the two-parameter quantization of the coordinate
algebra of $GL_{n}$ as well as the two-parameter quantization of
$U(\liegl_{n})$ given in \cite{Tk}.  

\subsection{The quantum group $GL_{\alpha,\beta}(n)$}
\begin{definition}\label{def:2-par-quant-ogln}\cite[Sec. 2]{Tk}
Let $\alpha,\ \beta \in \k^{\times}$ and $n\in \NN$. The algebra
$\Oc_{\alpha,\beta}(M_{n})$ is the $\k$-algebra generated by the elements
$\{x_{ij}:\ 1\leq i,j\leq n\}$ satisfying the following relations:
  \begin{eqnarray}
    \nonumber x_{ik}x_{ij} &=&\alpha^{-1} x_{ij}x_{ik}\qquad\mbox{ if }j<k,  \\
    \nonumber x_{jk}x_{ik}&= &\beta x_{ik}x_{jk}\qquad\mbox{ if }i<j,  \\
    \nonumber x_{jk}x_{il}&=&\beta\alpha x_{il}x_{jk}\qquad\mbox{ and } \\
    \nonumber x_{jl}x_{ik} - x_{ik}x_{jl}& = &(\beta-\alpha)x_{il}x_{jk}
    \qquad\mbox{ if }i<j \mbox{ and }k<l.
  \end{eqnarray}
\end{definition}
This algebra is a non-commutative
polynomial algebra in the variables $x_{ij}$ and has no non-zero
divisor. It has a basis
$\{\prod_{i,j}^{} x_{ij}^{e_{ij}}\vert\ e_{ij} \in \NN_{0}\}$, where the
products are formed with respect to a
fixed ordering of $\{x_{ij}:\ 1\leq i,j\leq n\}$.
It has a bialgebra structure determined by
  \begin{align*}
   \com(x_{ij}) = \sum_{s=1}^{n} x_{is}\ot x_{sj}
   \quad\mbox{and}\quad \eps(x_{ij}) = \delta_{ij}.
  \end{align*}
If $\alpha= 1 = \beta$ this commutative
algebra is just
$\Oc(M_{n}(\k))$. Thus $\Oc_{\alpha,\beta}(M_{n})$
defines
a two-parameter quantization $M_{\alpha,\beta}(n,\k)$ 
of the semigroup scheme
$M(n, \k)$.

\par The \textit{quantum determinant} $g= \vert X \vert$, where 
$X = (x_{ij})_{1\leq i,j\leq n}$
denotes the $n\times n$-matrix with coefficients $x_{ij}$, is
defined by
  \begin{align*}
   g &= \sum_{\sigma \in \S_{n}}^{n}
   (-\beta)^{-\ell(\sigma)}x_{\sigma(1),1}\cdots
   x_{\sigma(n),n}
     = \sum_{\sigma \in \S_{n}}^{n}
   (-\alpha^{-1})^{-\ell(\sigma)}x_{1,\sigma(1)}\cdots x_{n,\sigma(n)},
\end{align*}
It is a
group-like element and we have that
$ x_{ij}g=(\beta\alpha)^{i-j}gx_{ij}$ for all $1\leq i,j\leq n$.
Thus, the powers of $g$ satisfy the left and right Ore condition.
The localization of $\Oc_{\alpha,\beta}(M_{n})$ at the powers of $g$
gives the Hopf algebra
$\Oc_{\alpha,\beta}(GL_{n}) := \Oc_{\alpha,\beta}(M_{n})[g^{-1}]$, which is
the Hopf algebra $A_{\alpha^{-1},\beta}$ in \cite{Tk}. This
Hopf algebra corresponds to the quantum group
$GL_{\alpha,\beta}(n)$. The antipode $\SS$ is given by
  \begin{equation}\label{eq:antipoda}
   \SS(x_{ij}) = (-\beta)^{j-i}g^{-1}\vert X_{ji}\vert
   = (-\alpha^{-1})^{j-i}\vert X_{ji}\vert g^{-1},
  \end{equation}
where $\vert X_{ji}\vert$ denotes the quantum determinant
of the $(n-1)\times(n-1)$ minor obtained by removing the $j$-th row
and the $i$-th column. Hence
  \begin{equation}\label{eq:cuadrado de la antipoda}
   \SS^{2}(x_{ij}) = (\alpha^{-1}\beta)^{j-i}x_{ij}.
  \end{equation}

\begin{obs}\label{rmk:valores-de-a-b}
$(a)$ By taking different values of $\alpha,\beta$, e.g
$(\alpha, \beta) = (q^{-1},q)$ or $(1,q)$ one obtains the well-known one
parameter deformations of $\Oc (GL_{n})$, the standard in the first
case and the Dipper-Donkin \cite{DD} deformation in the second.
Hence we will assume that 
$\alpha^{-1} \neq \beta$  and $\alpha
\neq 1 \neq \beta$.

\smallbreak
$(b)$ In \cite{Tk3}, it is studied the problem of the cocycle
deformations of the quantum groups $\Oc_{\alpha^{-1}, \beta}(GL_{n})$.
It is proved in Thm. 2.6 that the bialgebra $\Oc_{\alpha^{-1}, \beta}(M_{n})$ is
isomorphic to a cocycle deformation of $\Oc_{\alpha'^{-1}, \beta'}(M_{n})$ 
if and only if $\alpha^{-1} \beta = \alpha'^{-1} \beta'$ or 
$\alpha \beta = (\alpha'^{-1} \beta')^{-1}$ (here we have used
Takeuchi's notation to avoid confusion with the reference).
Moreover, one has by Cor. 2.8 that if $\alpha^{-1}\beta\neq 1$
then $\Oc_{\alpha^{-1}, \beta}(GL_{n})$
is not a cocycle deformation of commutative Hopf
algebras, and $\Oc_{\alpha^{-1}, \beta}(M_{n})$ is not a cocycle 
deformation of commutative bialgebras.

\smallbreak
$(c)$ It is not difficult to see that $\Oc_{\alpha^{-1}, \beta}(GL_{n})$
and $\Oc_{\alpha, \beta^{-1}}(GL_{n})$ are isomorphic as Hopf algebras.
The isomorphism is determined by defining $x_{ij} \mapsto y_{n+1-i, n+1-j}$,
where the elements 
$x_{ij}$ and $y_{ij}$ denote the canonical generators of
$\Oc_{\alpha^{-1}, \beta}(M_{n})$
and $\Oc_{\alpha, \beta^{-1}}(M_{n})$, respectively (see 
\cite[Prop. 1.11]{DPW}). Moreover, in \cite[Thm. 2.4, Cor. 2.6]{DPW} 
it is 
proved that if $\alpha^{-1}\beta = \alpha'^{-1}\beta'$ or 
$\alpha^{-1}\beta = (\alpha'^{-1}\beta')^{-1}$ 
then $\Oc_{\alpha^{-1}, \beta}(GL_{n})$
and $\Oc_{\alpha'^{-1}, \beta'}(GL_{n})$ are isomorphic as coalgebras;
in particular, their categories of comodules are equivalent, 
compare with Remark $(b)$.
\end{obs}

\subsection{Quantum Borel subgroups of $GL_{\alpha,\beta}(n)$}
\label{subsubsec:borel-qsubgroups}
Let $J_{+}$ be the ideal of $\Oc_{\alpha,\beta}(GL_{n})$ generated
by the elements $\{x_{ij}\}_{i>j}$. By the relations in Def.
\ref{def:2-par-quant-ogln}, it is a two-sided ideal and since $\com
(x_{ij}) = \sum_{k=1}^{n} x_{ik}\ot x_{kj}$ and $\eps(x_{ij}) =
\delta_{ij}$, it follows that $J_{+}$ is also a coideal. Moreover,
by equation \eqref{eq:antipoda} we know that $ \SS(x_{ij}) =
(-\beta)^{j-i}g^{-1}\vert X_{ji}\vert$ and $\vert X_{ji}\vert \in
I_{+}$, whence $\SS(J_{+}) \subseteq J_{+}$ and $J_{+}$ is a Hopf
ideal. The Hopf algebra quotient $\Oc_{\alpha,\beta}(GL_{n})/J_{+}$
corresponds to the two-parameter deformation of a Borel subalgebra
of $GL_{n}(\k)$ and it is denoted by $\Oc_{\alpha,\beta}(B^{+})$.
Denote by
$$t_{+}: \Oc_{\alpha,\beta}(GL_{n}) \to
\Oc_{\alpha,\beta}(B^{+})$$
the canonical Hopf algebra quotient and $t_{+}(x_{ij}) =
\bar{x}_{ij}$ for all $1\leq i,j\leq n$. Then
$\Oc_{\alpha,\beta}(B^{+})$ is generated as an algebra by the
elements $\{\bar{x}_{ij}\vert \ 1\leq i\leq j\leq n\}$ satisfying
the relations
\begin{eqnarray}
    \bar{x}_{ik}\bar{x}_{ij} &=&\alpha^{-1} \bar{x}_{ij}\bar{x}_{ik}
    \qquad\mbox{ if }j<k,  \label{eq:borel}\\
    \nonumber \bar{x}_{jk}\bar{x}_{ik}&= &\beta \bar{x}_{ik}\bar{x}_{jk}
    \qquad\mbox{ if }i<j,  \\
    \nonumber \bar{x}_{jk}\bar{x}_{il}&=&\beta\alpha \bar{x}_{il}\bar{x}_{jk}
    \qquad\mbox{ and } \\
    \nonumber \bar{x}_{jl}\bar{x}_{ik} & = & \bar{x}_{ik}\bar{x}_{jl}
    \qquad\mbox{ if }i<j \mbox{ and }k<l.
  \end{eqnarray}
The elements $\{\bar{x}_{ii}\}_{1\leq i\leq n}$ are
invertible group-like elements which commute with each other and
$t_{+}(g^{-1}) = \bar{x}_{11}^{-1}\cdots \bar{x}_{nn}^{-1}$.
Moreover, the set 
$$\Big\{\prod_{i< j}^{} \bar{x}_{ij}^{e_{ij}}\prod_{i
}^{} \bar{x}_{ii}^{e_{i}} \vert\ e_{ij} \in \NN_{0}, e_{i}\in
\ZZ\Big\}$$ 
is a linear basis of $\Oc_{\alpha,\beta}(B^{+})$,
for some fixed ordering of $\{\bar{x}_{ij}\vert 1\leq i < j \leq n\}$. In
particular, it has no non-zero divisors.

\par Analogously, 
by taking the ideal $J_{-}$ generated by the elements
$\{x_{ij}\}_{i<j}$, one defines the borel subalgebra
$\Oc_{\alpha,\beta}(B^{-}):= \Oc_{\alpha,\beta}(GL_{n})/J_{-}$.
Denote by $t_{-}: \Oc_{\alpha,\beta}(GL_{n}) \to
\Oc_{\alpha,\beta}(B^{-})$ the Hopf algebra quotient and
$t_{-}(x_{ij}) = \hat{x}_{ij}$ for all $1\leq i,j\leq n$.
The following lemma is \cite[Thm. 8.1.1]{PW} in the
case of two-parameter deformations.

\begin{lema}\label{lema:delta-injective}
The following algebra map is injective 
$$\delta = (t_{+}\ot t_{-})\circ \com :
\Oc_{\alpha,\beta}(GL_{n}) \to \Oc_{\alpha,\beta}(B^{+})\ot
\Oc_{\alpha,\beta}(B^{-}).$$ 
\end{lema}

\begin{proof}
Since $\delta (x_{ij}) = \sum_{i,j \leq k } \bar{x}_{ik}\ot
\hat{x}_{kj}$, then $\delta (x_{ij}) \neq 0$ for all
$1\leq i,j\leq n$. Moreover, since $\delta$ is an algebra map,
$\delta(x) = 0$ if and only if $\delta(x)\delta(g^{t}) =
\delta(xg^{t})= 0$ for all $t \in \ZZ$. As $\{g^{-t}\prod_{i,j}^{}
x_{ij}^{e_{ij}}\vert\ t, e_{ij} \in \NN_{0}\}$ is a set of
generators of $\Oc_{\alpha,\beta}(GL_{n})$, we may assume that if
$\delta(x) = 0$, then $x \in \Oc_{\alpha,\beta}(M_{n})$. 

\par By Def. \ref{def:2-par-quant-ogln}, we know that
$\Oc_{\alpha,\beta}(M_{n})$ has a linear basis consisting of monomials
of the form 
$$ m_{\textbf{e}} = \prod_{i<j} x_{ij}^{e_{ij}}\cdot 
\prod_{i=1}^{n} x_{ii}^{e_{ii}}\cdot \prod_{i>j} x_{ij}^{e_{ij}},$$
where $\textbf{e} = (e_{ij})_{1\leq i,j\leq n}$ runs over the 
set of $n\times n$ matrices with coefficients in $\ZZ_{+}$, and the 
product of the $x_{ij}$'s is taken with respect to a fixed order.
Define the degree of a monomial 
$m_{\textbf{e}}$ to be the matrix $\textbf{e}$. Then the opposite
lexicographic order (i.e. $x_{ij} \geq x_{kl}$ if $i < k$ and if
$i = k$ then $j\leq l$), induces a partial order in the monomials 
according to their degree. Thus
$$\delta(m_{\textbf{e}}) = 
c\prod_{i<j} \bar{x}_{ij}^{e_{ij}}\cdot 
\prod_{i} \bar{x}_{ii}^{e_{ii}+ \sum_{i>j} e_{ij}}\ot
\prod_{j} \hat{x}_{jj}^{e_{jj}+ \sum_{i<j} e_{ij}} 
\cdot \prod_{i>j} \hat{x}_{ij}^{e_{ij}}
+ \mbox{ lower terms},$$
with $c\neq 0$, since by Def. \ref{def:2-par-quant-ogln} 
and equations \eqref{eq:borel}
changing the order of the factors
in a monomial only results in a nonzero scalar factor and
some lower terms. From this follows that if 
$\textbf{e} \neq \textbf{f} = (f_{ij})_{1\leq i,j\leq n}$ 
then $\delta(m_{\textbf{e}})$ and $\delta(m_{\textbf{f}})$
have different leading terms. This implies the linear independence
of the $\delta(m_{\textbf{e}})$'s, which implies that $\delta$ is 
injective on $ \Oc_{\alpha,\beta}(M_{n})$. 
\end{proof}

\subsection{The quantum group
$U_{\alpha, \beta}(\liegl_{n})$}\label{subsec:U-alpha-beta} Now we
recall the definition of the two-parameter deformation $U_{\alpha,
\beta}(\liegl_{n})$ of the enveloping algebra of $\liegl_{n}$, 
following \cite{BW3} but changing the notation of the parameters
$\alpha = r$ and $\beta = s$ to stress the relation with the
two-parameter deformation $\Oc_{\alpha,\beta}(GL_{n})$.
 
\begin{definition}\label{def:Ualfa-beta}\cite[Sec. 1]{BW3}
Assume that $\alpha\neq \beta$. The algebra $U_{\alpha,
\beta}(\liegl_{n})$ is the $\k$-algebra generated by the elements
$\{a_{i}, a_{i}^{-1}, b_{i}, b_{i}^{-1}, e_{j}, f_{j}:\ 1\leq i\leq
n, 1\leq j < n\}$ satisfying the following relations: for all $1\leq
i,k\leq n,\ 1\leq j,l< n$,
\begin{eqnarray}
    \nonumber a_{i}, b_{k} & &\mbox{commute with each other and } \\
    \nonumber a_{i}a_{i}^{-1}  & = & a_{i}^{-1}a_{i} = b_{i}b_{i}^{-1} =
     b_{i}^{-1}b_{i} = 1, \\
    \nonumber a_{i}e_{j}& = &\alpha^{\delta_{ij}-\delta_{i,j+1}} e_{j}a_{i}, \qquad
     b_{i}e_{j}= \beta^{\delta_{ij}-\delta_{i,j+1}} e_{j}b_{i} \\
    \nonumber a_{i}f_{j} &= &\alpha^{-\delta_{ij}+\delta_{i,j+1}}
     f_{j}a_{i}, \qquad
     b_{i}f_{j}= \beta^{-\delta_{ij}+\delta_{i,j+1}} f_{j}b_{i}, \\
    \nonumber [e_{j},f_{l}] &=& \frac{\delta_{jl}}{\alpha -
    \beta} (a_{j}b_{j+1}- a_{j+1}b_{j}),\\
    \nonumber [e_{j},e_{l}]& =&[f_{j},f_{l}]=0\qquad \mbox{if }\vert j-l\vert >1, \\
    [0.1pt] 0 & = & e_{j}^{2}e_{j+1} - (\alpha + \beta)e_{j}e_{j+1}e_{j} 
+ \alpha\beta e_{j+1}e_{j}^{2},
     \label{rel:adei}\\
     \nonumber 0 & = & e_{j}e_{j+1}^{2} - (\alpha + \beta)e_{j+1}e_{j}e_{j+1} 
+ \alpha\beta
     e_{j+1}^{2}e_{j},\\
    [0.1pt] 0 & = & f_{j}^{2}f_{j+1} - (\alpha^{-1} + \beta^{-1})
    f_{j}f_{j+1}f_{j} + \alpha^{-1}\beta^{-1} 
     f_{j+1}f_{j}^{2}, \label{rel:adfi}\\
     \nonumber 0 & = & f_{j}f_{j+1}^{2} - (\alpha^{-1} + 
\beta^{-1})f_{j+1}f_{j}f_{j+1} + \alpha^{-1}\beta^{-1}
     f_{j+1}^{2}f_{j},
   \end{eqnarray}
\end{definition}

Let $w_{j}= a_{j}b_{j+1}$ and $w_{j}'= a_{j+1}b_{j}$ for
all $1\leq j <n$. The algebra $U_{\alpha, \beta}(\liegl_{n})$ has a
Hopf algebra structure determined by the elements $a_{i}$, $b_{i}$
being group-likes, the $e_{j}$ being $(w_{j},1)$-primitives and the
$f_{j}$ being $(1,w_{j}')$-primitives, for all $1\leq i \leq n$,
$1\leq j<n$. That is,
  \begin{align*}
    \com(a_{i}) &= a_{i}\ot a_{i},&  \com(b_{i}) &= b_{i}\ot b_{i},\\
    \com(e_{j}) & = e_{j} \ot 1 + w_{j} \ot e_{j}
    \qquad\mbox{ and }& \com(f_{j}) &=  1\ot f_{j} + f_{j} \ot w_{j}' .
  \end{align*}

Similarly as before, by taking different values of the parameters
$\alpha,\beta$, one obtains the well-known one parameter
deformations of $U (\liegl_{n})$ as quotients of this one. For
example, if $(\alpha, \beta) = (q,q)$, then the group-like elements
$a_{i}b_{i}^{-1}$ are central and the quotient of $U_{q,q}$ by the
Hopf ideal $(a_{i}b_{i}^{-1} -1:\ 1\leq i\leq n)$ can be identified
with the Drinfeld-Jimbo Hopf algebra $U_{q}(\liegl_{n})$.

\par There is a canonical triangular
decomposition
\begin{equation}\label{eq:desc-triangular-gl}
 U_{\alpha, \beta}(\liegl_{n}) =
U_{\alpha, \beta}^{-}(\liegl_{n})\ot U_{\alpha,
\beta}^{\circ}(\liegl_{n})\ot U_{\alpha, \beta}^{+}(\liegl_{n}),
\end{equation}
where $U_{\alpha, \beta}^{+}(\liegl_{n})$, $U_{\alpha,
\beta}^{-}(\liegl_{n})$ and $U_{\alpha, \beta}^{\circ}(\liegl_{n})$
are the subalgebras generated by $\{e_{i}\}_{1\leq i <n}$,
$\{f_{i}\}_{1\leq i <n}$ and $\{a_{j}^{\pm 1}, b_{j}^{\pm
1}\}_{1\leq j \leq n}$, respectively.

\par Let $U_{\alpha, \beta}(\lieb^{+})$ be the Hopf
subalgebra of $U_{\alpha, \beta}(\liegl_{n})$ generated by the
elements $\{ e_{j}, w_{j}^{\pm 1}, a_{n}^{\pm 1}\vert 1\leq j <n\}$
and $U_{\alpha, \beta}(\lieb^{-})$ be the Hopf subalgebra generated
by the elements $\{ f_{j}, (w_{j}')^{\pm 1}, b_{n}^{\pm 1}\vert
1\leq j <n\}$. Then, by the triangular decomposition
\eqref{eq:desc-triangular-gl}, the multiplication
$ m: U_{\alpha, \beta}(\lieb^{+}) \ot U_{\alpha, \beta}(\lieb^{-})
\to U_{\alpha, \beta}(\liegl_{n})$,
is surjective. Thus, its transpose
$^{t}m: U_{\alpha, \beta}(\liegl_{n})^{*} \to [U_{\alpha,
\beta}(\lieb^{+}) \ot U_{\alpha, \beta}(\lieb^{-})]^{*},
$ defines an injective map. Moreover, by \cite[Lemma 2.2 and
Thm. 2.7]{BW2} there exists a Hopf algebra pairing between
$U_{\alpha, \beta}(\lieb^{+})$ and $U_{\alpha, \beta}(\lieb^{-})$,
and $U_{\alpha, \beta}(\liegl_{n}) \simeq D(U_{\alpha,
\beta}(\lieb^{+}), U_{\alpha, \beta}(\lieb^{-})^{\cop})$, the double
related to the pairing.

\subsubsection{A PBW-type basis of  $U_{\alpha,
\beta}(\liegl_{n})$}
If $\alpha \neq - \beta$, $U_{\alpha,
\beta}(\liegl_{n})$ admits a PBW-type basis because of the next
theorem, which is a special case of \cite[Thm. $\textbf{A}_{n}$]{K},
see \cite[Thm. 3.2]{BKL}. First, let $\{\Ee_{i,j}:\ 1\leq j \leq i <
n\}$ be the elements defined by
$$\Ee_{j,j} = e_{j}\quad\mbox{ and }\quad\Ee_{i,j} =
e_{i}\Ee_{i-1,j} - \alpha^{-1}\Ee_{i-1,j}e_{i} =
[e_{i},\Ee_{i-1,j}]_{\alpha^{-1}},$$
for all $1\leq j < i <n$. Then, relation \eqref{rel:adei}
can be reformulated by saying
$$  e_{i+1}\Ee_{i+1,i} = \beta^{-1} 
  \Ee_{i+1,i}e_{i+1},\qquad
\Ee_{i+1,i}e_{i}=  \beta^{-1} e_{i}\Ee_{i+1,i}.
$$
Analogously, define $\{\Ff_{i,j}:\ 1\leq j \leq
i < n\}$ by letting $\Ff_{j,j} = f_{j}$ and $\Ff_{i,j} =
f_{i}\Ff_{i-1,j} - \beta^{-1}\Ff_{i-1,j}f_{i}$ for $1\leq j < i <
n$. By \cite[p. 5]{BW} we have	
\begin{equation}\label{eq:w-s}
w_{s}\Ee_{k,l} = \alpha^{\delta_{s,k}+ \cdots + \delta_{s,l}}
\beta^{\delta_{s+1,k}+ \cdots + \delta_{s+1,l}}\Ee_{k,l}w_{s},
\end{equation}	
for all $1\leq s<n$. 
For $k\geq l$, denote $w_{k,l} = w_{k}w_{k-1}\cdots w_{l}$ 
and set $\zeta = 1 - \alpha^{-1}\beta$.
The following lemma is due to Benkart and Witherspoon.

\begin{lema}\label{lema:comult-E}\cite[Lemma 2.19]{BW} 
For $1\leq l \leq k < n$, 
\begin{align*} 
\com(\Ee_{k,l}) = \Ee_{k,l}\ot 1 + w_{k,l}\ot\Ee_{k,l} +
\zeta \sum_{j=l}^{k-1}
\Ee_{k,j+1}w_{j,l}\ot\Ee_{j,l}.
\qed\end{align*}	
\end{lema}

The following theorem gives a PBW-basis of 
$U_{\alpha, \beta}(\liegl_{n})$.
\begin{theorem}\cite[Thm. $\textbf{A}_{n}$]{K}
\label{thm:PBW-basis-gl}
Assume that $\alpha\neq -\beta$. Then
 \begin{itemize}
  \item[$(i)$] $\mathcal{B}_{0}^{+} =
    \{\Ee_{i_{1},j_{1}}\Ee_{i_{2},j_{2}}\cdots \Ee_{i_{p},j_{p}}\vert\
      (i_{1},j_{1})\leq (i_{2},j_{2})\leq \cdots \leq (i_{p},j_{p}) \}$ and
     $\mathcal{B}_{0}^{-} =    \{\Ff_{i_{1},j_{1}}\Ff_{i_{2},j_{2}}\cdots
      \Ff_{i_{p},j_{p}}\vert\
      (i_{1},j_{1})\leq (i_{2},j_{2})\leq \cdots \leq (i_{p},j_{p}) \}$
       (lexicographical ordering) 
       are linear basis of $U^{+}_{\alpha, \beta}(\liegl_{n})$
       and
        $U^{-}_{\alpha, \beta}(\liegl_{n})$ respectively.
  \item[$(ii)$] $\mathcal{B}_{1}^{+} =
    \{e_{i_{1},j_{1}}e_{i_{2},j_{2}}\cdots e_{i_{p},j_{p}}\vert\
      (i_{1},j_{1})\leq (i_{2},j_{2})\leq \cdots \leq (i_{p},j_{p}) \}$ and
     $\mathcal{B}_{1}^{-} =    
      \{f_{i_{1},j_{1}}f_{i_{2},j_{2}}\cdots f_{i_{p},j_{p}}\vert\
      (i_{1},j_{1})\leq (i_{2},j_{2})\leq \cdots \leq (i_{p},j_{p}) \}$
       (lexicographical ordering) 
      are linear basis of $U^{+}_{\alpha, \beta}(\liegl_{n})$
       and
        $U^{-}_{\alpha, \beta}(\liegl_{n})$ respectively, where
        $e_{i,j} = e_{i}e_{i-1}\cdots e_{j}$
         and $f_{i,j} = f_{i}f_{i-1}\cdots f_{j}$ for $i\geq j$.\qed
 \end{itemize}
\end{theorem}

Since $U_{\alpha, \beta}^{\circ}(\liegl_{n})$ is a group algebra,
combining these basis and using the
triangular decomposition, one obtains a PBW-basis for $U_{\alpha,
\beta}(\liegl_{n})$.

\par There is also a two-parameter analog of
$U_{q}(\liesl_{n})$ which is given by the subalgebra of 
$U_{q}(\liegl_{n})$ generated by the elements $e_{j}, f_{j}, w_{j}$,
$w_{j}^{-1}$, $w_{j}'$ and $(w_{j}')^{-1}$ for $1\leq j<n$. Clearly,
if $(\alpha, \beta) = (q,q)$, this subalgebra of $U_{q}(\liegl_{n})$
is precisely $U_{q}(\liesl_{n})$. However, 
one can not define using Def. \ref{def:2-par-quant-ogln} 
a two-parameter analog of $\Oc_{q}(SL_{n})$ since 
the quantum determinant $g$ is not central.  

\subsubsection{A Hopf pairing between $\Oc_{\alpha,\beta}(GL_{n})$ 
and $U_{\alpha,
\beta}(\liegl_{n})$}\label{subsub:hopf-pairing}
The Hopf algebras $\Oc_{\alpha,\beta}(GL_{n})$ and $U_{\alpha,
\beta}(\liegl_{n})$ are associated with each other by a Hopf pairing
$\langle -,- \rangle: U_{\alpha, \beta}(\liegl_{n}) \times
\Oc_{\alpha,\beta}(GL_{n}) \to \k$ which is defined on the
generators by
   \begin{align*}
     \langle a_{i},x_{st} \rangle & =
     \delta_{st}\alpha^{\delta_{is}},\qquad
     &\langle b_{i},x_{st} \rangle &=
     \delta_{st}\beta^{\delta_{is}},\\
     \langle e_{j},x_{st} \rangle & =
     \delta_{js}\delta_{j+1,t},\qquad
     &\langle f_{j},x_{st} \rangle  &=
     \delta_{j+1,s}\delta_{jt},
   \end{align*}
where $1\leq i\leq n,\ 1\leq j <n, 1\leq s,t\leq n$,
see \cite{Tk}. 
Clearly, the Hopf pairing defines a Hopf algebra map 
\begin{equation}\label{eq:map-hopf-pair}
 \psi: \Oc_{\alpha,\beta}(GL_{n}) \to 
U_{\alpha, \beta}(\liegl_{n})^{\circ}, 
\end{equation}
given by 
$\psi(x_{st}) (u) = \langle u ,x_{st} \rangle$,
for all $1\leq s,t\leq n$ and 
for $u= a_{i}, b_{i}, e_{j}, f_{j}$ with $1\leq i\leq n$, $1\leq j < n$.
Following \cite[Sec. 4]{Tk2}, we say that
$\Oc_{\alpha,\beta}(GL_{n})$ is
\textit{connected} if $\psi$ is injective.In 
case that $\alpha$ and $\beta$
are roots of unity, 
$\Oc_{\alpha,\beta}(GL_{n})$ is not
connected. However, a map induced by $\psi$
is injective on certain quotients of these quantum groups and
this fact is needed to do the first step in the contruction of 
the quantum subgroups of $\Oc_{\alpha,\beta}(GL_{n})$. 
The following technical lemmata will be needed in the sequel.
\begin{lema}\label{lema:pair-e-f} 
For all $1\leq i,j\leq n$ and $1\leq l\leq k < n$,
\begin{enumerate}
 \item[$(i)$] $ \langle \Ee_{k,l} , x_{ij} \rangle = (-1)^{k-l}
   \alpha^{l-k} \delta_{l,i}\delta_{k+1,j}$,
 \item[$(ii)$] $ \langle \Ff_{k,l} , x_{ij} \rangle = 
\delta_{k+1,i}\delta_{l,j}$.
\end{enumerate}
\end{lema}

\pf We prove it by induction on $k\geq l$.
By definition we know that 
$\langle \Ee_{l,l} , x_{ij} \rangle =
\langle e_{l},x_{ij} \rangle  =
     \delta_{l,i}\delta_{l+1,j}$, hence the formula holds for $k = l$.
Now suppose the formula holds for 
$k-1 \geq l$, then 
\begin{align*}
\langle \Ee_{k,l} , x_{ij} \rangle & 
= \langle e_{k}\Ee_{k-1,l} , x_{ij} \rangle
- \alpha^{-1}\langle \Ee_{k-1,l}e_{k} , x_{ij} \rangle\\
& = \sum_{m=1}^{n}\langle e_{k}, x_{im} \rangle 
\langle \Ee_{k-1,l}, x_{mj} \rangle
- \alpha^{-1}\sum_{m=1}^{n}\langle \Ee_{k-1,l}, x_{im}\rangle 
\langle e_{k}, x_{mj} \rangle\\
& = \delta_{k,i} 
\langle \Ee_{k-1,l}, x_{k+1,j} \rangle
- \alpha^{-1}\langle \Ee_{k-1,l}, x_{i,k}\rangle 
\delta_{k+1,j}\\
& = \delta_{k,i}(-1)^{k-1-l} \alpha^{l-k+1} 
\delta_{l,k+1}\delta_{k,j} 
- \alpha^{-1}(-1)^{k-1-l} \alpha^{l-k+1}\delta_{l,i}
\delta_{k+1,l}\\
& = (-1)^{k-l} \alpha^{l-k}\delta_{l,i} 
\delta_{k+1,j},
\end{align*}
which finishes the proof of part $(i)$. The proof of 
$(ii)$ is analogous. By definition we know that 
$\langle \Ff_{l,l} , x_{ij} \rangle =
\langle f_{l},x_{ij} \rangle  =
     \delta_{l+1,i}\delta_{l,j}$, hence the formula holds for $k = l$.
Now suppose the formula holds for 
$k-1 \geq l$, then 
\begin{align*}
\langle \Ff_{k,l} , x_{ij} \rangle & 
= \langle f_{k}\Ff_{k-1,l} , x_{ij} \rangle
- \beta^{-1}\langle \Ff_{k-1,l}f_{k} , x_{ij} \rangle\\
& = \sum_{m=1}^{n}\langle f_{k}, x_{im} \rangle 
\langle \Ff_{k-1,l}, x_{mj} \rangle
- \beta^{-1}\sum_{m=1}^{n}\langle \Ff_{k-1,l}, x_{im}\rangle 
\langle f_{k}, x_{mj} \rangle\\
& = \delta_{k+1,i} 
\langle \Ff_{k-1,l}, x_{k,j} \rangle
- \beta^{-1}\langle \Ff_{k-1,l}, x_{i,k+1}\rangle 
\delta_{k,j}\\
& = \delta_{k+1,i} \delta_{k,k}\delta_{l,j} 
- \beta^{-1}\delta_{k,i}\delta_{l,j} 
\delta_{k,j}\\
& = \delta_{k+1,i}\delta_{l,j},
\end{align*}
which finishes the proof of part $(ii)$.
\epf


\section{Finite quantum subgroups of $GL_{\alpha,\beta}(n)$, 
$\alpha^{-1}\beta$
not a root of 1}\label{sec:a-b-not-root}
In this section we prove Thm. \ref{thm:main} $(a)$,  
as in \cite[Thm. 4.1]{Mu}.

\begin{theorem}\label{teo:cocientes-a-b-no-raices}
If $\alpha$ or $\beta$ is not a root of unity, then the
finite-dimensional quotients of $\Oc_{\alpha,\beta}(GL_{n})$ are
just the function algebras of finite subgroups of the diagonal torus
in $GL_{n}(\k)$.
\end{theorem}

\begin{proof}
Let $q: \Oc_{\alpha,\beta}(GL_{n}) \to A$ be a surjective Hopf
algebra map such that $\dim A$ is finite. Then by \cite{R}, the
antipode $\SS_{A}$ of $A$ has finite (even) order, say $2t$. Then by
\eqref{eq:cuadrado de la antipoda} it follows that
$ q(x_{ij}) = \SS_{A}^{2t}(q(x_{ij})) = q(\SS^{2t}(x_{ij})) =
(\alpha\beta)^{2t(j-i)}q(x_{ij})$.
Since $\alpha\beta$ is not a root of unity, we have that
$q(x_{ij}) = 0$ for all $1\leq i,j\leq n$ with $i\neq j$. Hence, $A$ is a
finite-dimensional quotient of $\k[x_{11},x_{22},\ldots, x_{nn}]$, the
coordinate algebra of the diagonal torus $\T \subseteq GL_{n}(\k)$. Thus,
$A$ is isomorphic to the algebra of functions of a finite subgroup of $\T$.
\end{proof}

\begin{obs}
 If $A$ is infinite-dimensional, it may not be commutative: take
the quotient given by $q:\Oc_{\alpha,\beta}(GL_{n}) \to
 \Oc_{\alpha,\beta}(B^{+})$, where $\Oc_{\alpha,\beta}(B^{+})$
is the Borel quantum subgroup defined in Sec.
\ref{subsubsec:borel-qsubgroups}. Then by construction,
$\Oc_{\alpha,\beta}(B^{+})$ is not commutative, non-semisimple and
$\SS^{2}\neq \id$.
\end{obs}

\section{Quantum subgroups of $GL_{\alpha,\beta}(n)$,
$\alpha^{-1}\beta$ a primitive root of 1}\label{sec:a-b-root}
In this section we determine \textit{all} the quotients of
$\Oc_{\alpha,\beta}(GL_{n})$ when
$\alpha$ and $\beta$ are roots of unity and satisfiy 
certain mild conditions. The proof follows
the ideas of \cite{AG}, but hard computational methods are
required.

\subsection{Central Hopf subalgebras and PI-Hopf triples} 
Let $\ell
\in \NN$ be an odd natural number such that 
$\alpha^{-1}\beta$ is a primitive 
$\ell$-th root of unity and $\alpha^{\ell}=1 = \beta^{\ell}$.
In \cite[Thm.
3.1]{DPW}, Du, Parshall and Wang define a generalization of the
quantum Frobenius map
$$F^{\#}: \Oc(GL_{n})
 \to \Oc_{\alpha,\beta}(GL_{n}),
\quad X_{ij} \mapsto x_{ij}^{\ell}, \quad g\mapsto g^{\ell},
$$
which is a Hopf algebra monomorphism that
corresponds \textit{intuitively} to a surjective map
of quantum groups $F: GL_{\alpha,\beta}(n) \to
GL(n)$.

\par For this reason, Thm.
\ref{thm:PBW-basis-gl} and the definition 
of $U_{\alpha, \beta}(\liegl_{n})$
we will assume from now on that
\begin{equation}\label{eq:cond-alpha-beta}
\alpha \neq \pm \beta, \beta^{-1},\  
\alpha^{-1}\beta\ is\ a\ primitive\ \ell-th\ root\ of\ 1\ and\
\alpha^{\ell}= 1 = \beta^{\ell},
 \end{equation}
where $\ell$ is supposed to be odd and $\ell \geq 3$.

\begin{prop}\label{prop:B-central}\cite[Prop. 3.8]{DPW}
With the above notation and assumptions and identifying
$\Oc(GL_{n})= F^{\#}(\Oc(GL_{n}))$ 
we have
 \begin{itemize}
  \item[$(i)$] The Hopf subalgebra
  $ \Oc(GL_{n})$ is central in $\Oc_{\alpha,\beta}(GL_{n})$.
  \item[$(ii)$] $\Oc_{\alpha,\beta}(GL_{n})$ 
is faithfully flat over $\Oc(GL_{n})$.
  \item[$(iii)$] $\overline{H}:= $
       $ \Oc_{\alpha,\beta}(GL_{n}) /
     \Oc(GL_{n})^{+}\Oc_{\alpha,\beta}(GL_{n})$ 
 has dimension $\ell^{n^{2}}$.\qed
 \end{itemize}
\end{prop}

Although the construction of the quantized 
coordinate rings differs from the
one given in \cite{DL},  
the explicit definition given by generators and relations
allows us to give as in \cite{AG} a coalgebra section to $\pi$.

\begin{cor}\label{cor:def-gamma}
The Hopf algebra surjection 
$\pi: \Oc_{\alpha,\beta}(GL_{n}) \to \overline{H}$ 
admits a coalgebra section $\gamma$.
\end{cor}

\begin{proof}
From \cite[Prop. 3.5]{DPW} it follows that
$\big\{\prod_{i, j}^{} \bar{x}_{ij}^{e_{ij}} 
\vert\ 0\leq e_{ij} < \ell \big\}$ is 
a basis of $\overline{H}$ for some fixed ordering 
of the $\pi(x_{ij}) = \bar{x}_{ij}$.
With this in mind, define the linear map
\begin{equation}\label{def:seccion}
\gamma : \overline{H} \to \Oea, \qquad 
\gamma\big(\prod_{i, j}^{} \bar{x}_{ij}^{e_{ij}}\big)
= \prod_{i, j}^{} x_{ij}^{e_{ij}}.
\end{equation}
Clearly, $\gamma$ is a linear section of $\pi$. Moreover, 
a direct calculation shows that $\gamma$ is also a coalgebra map.
\end{proof}

We end this subsection with the following corollary.

\begin{cor}\label{cor:PI-hopf-triple-example}
$(\Oc(GL_{n}),\Oc_{\alpha,\beta}(GL_{n}), 
\overline{H})$ is a PI-Hopf triple
and one has the central extension
$
  1\to \Oc(GL_{n}) \xrightarrow{\iota}
  \Oc_{\alpha,\beta}(GL_{n}) \xrightarrow{\pi}
   \overline{H} \to 1$. \qed
\end{cor}

\subsection{Restricted two-parameter quantum groups}
We recall now the definition of the restricted
two-parameter quantum groups given in \cite{BW}. 
They are finite-dimensional quotients of
the two-parameter quantum groups $U_{\alpha, \beta}(\liegl_{n})$
given by Def. \ref{def:Ualfa-beta}.
Since $\alpha$ and $\beta$ are roots of unity,
$U_{\alpha, \beta}(\liegl_{n})$ contains central elements which
generate a Hopf ideal.

\par The following theorem is a very 
small variation of results in
\cite{BW} for $U_{\alpha,\beta}(\liesl_{n})$, 
which hold with the same proofs,
since the only difference relays on central 
group-like elements: use for example that
$w_{k}^{\ell} - 1 := a_{k}^{\ell}b_{k+1}^{\ell} - 1 =
a_{k}^{\ell}(b_{k+1}^{\ell} - 1) + (a_{k}^{\ell} - 1 )$ 
for all $1\leq k < n$.

\begin{theorem}\label{thm:elem-centrales-U}
\begin{enumerate}
 \item[$(i)$]\cite[Thm. 2.6]{BW} 
The elements $\Ee_{k,l}^{\ell},\
 \Ff_{k,l}^{\ell},\ a_{i}^{\ell},\
b_{i}^{\ell}$ with $1\leq l\leq k< n$ and $1\leq i \leq n$ are
central in $U_{\alpha,\beta}(\liegl_{n})$.
\item[$(ii)$]\cite[Thm. 2.17]{BW} Let $I_{n}$ be the ideal
generated by the elements $\Ee_{k,l}^{\ell},\ \Ff_{k,l}^{\ell},\
a_{i}^{\ell}-1,\ b_{i}^{\ell}-1$ with $1\leq l\leq k< n$ and $1\leq
i \leq n$. Then $I_{n}$ is a Hopf ideal.\qed
\end{enumerate}
\end{theorem}

\begin{obs}
Let $B$ be the subalgebra of $U_{\alpha,\beta}(\liegl_{n})$ 
generated by
 $e_{k}^{\ell},\
 f_{k}^{\ell},\ a_{i}^{\pm \ell}$,
$b_{i}^{\pm \ell}$ with $1\leq k< n$ and $1\leq i \leq n$. 
Clearly, it is central by
Thm. \ref{thm:elem-centrales-U} $(i)$. Moreover, by
\cite[(2.24)]{BW}, we know that 
$$\com(e_{k}^{\ell}) = e_{k}^{\ell} \ot 1 + w_{k}^{\ell}\ot e_{k}^{\ell}
\mbox{ and }
\com(f_{k}^{\ell}) = 1 \ot f_{k}^{\ell} + f_{k}^{\ell}\ot (w'_{k})^{\ell}.$$
This implies that $B$ is indeed a Hopf subalgebra. 
\end{obs}

\begin{definition}\label{def:rest-u}\cite[Def. 2.15]{BW}
The \emph{restricted two-parameter quantum group}
${\bf u}_{\alpha,\beta}(\liegl_{n})$ is the quotient
$$ {\bf u}_{\alpha,\beta}(\liegl_{n}) := 
U_{\alpha, \beta}(\liegl_{n}) / I_{n},$$
given by Thm. \ref{thm:elem-centrales-U}. 
Denote by
$r: U_{\alpha, \beta}(\liegl_{n}) \to  
{\bf u}_{\alpha,\beta}(\liegl_{n})$ 
the canonical surjective Hopf
algebra map.
\end{definition}

\begin{obs}
The restricted two-parameter quantum group
${\bf u}_{\alpha,\beta}(\liesl_{n})$ defined in \cite{BW} is just the
quotient of $U_{\alpha, \beta}(\liesl_{n})$ by the ideal $J_{n}$
generated by the elements
$\Ee_{k,j}^{\ell},\ \Ff_{k,j}^{\ell},\
(a_{k}b_{k+1})^{\ell}-1,\ (a_{k+1}b_{k})^{\ell}-1$ with
$1\leq j\leq k< n$. In this case, ${\bf u}_{\alpha,\beta}(\liesl_{n})$ is
a pointed Hopf algebra which is a Drinfeld double and $\dim
{\bf u}_{\alpha,\beta}(\liesl_{n}) = \ell^{(n+2)(n-1)}$. 
Moreover, one has
the commutative diagram
$$\xymatrix{U_{\alpha, \beta}(\liesl_{n})
\ar@{^{(}->}[r]^{\iota} \ar@{->>}[d]_{} &
U_{\alpha, \beta}(\liegl_{n}) \ar@{->>}[d]^{r }\\
{\bf u}_{\alpha, \beta}(\liesl_{n})\ar@{^{(}->}[r]_{j} & {\bf u}_{\alpha,
\beta}(\liegl_{n}) .}$$
Indeed, one can use the PBW-basis
 of $U_{\alpha, \beta}(\liegl_{n})$
and $U_{\alpha, \beta}(\liesl_{n})$ 
--  see Thm. \ref{thm:PBW-basis-gl}--
to prove that the map $j$ is injective,
since $I_{n}\cap U_{\alpha, \beta}(\liesl_{n}) = J_{n}$.
\end{obs}

\begin{lema}\label{lema:dim-u}
${\bf u}_{\alpha,\beta}(\liegl_{n})$ 
is a pointed Hopf algebra of
dimension $\ell^{n^{2}+n}$.
\end{lema}

\begin{proof}
Since $U_{\alpha,\beta}(\liegl_{n})$ 
is pointed, ${\bf u}_{\alpha,\beta}(\liegl_{n})$ is pointed by 
\cite[Cor. 5.3.5]{Mo}.
By Thm. \ref{thm:PBW-basis-gl}, 
${\bf u}_{\alpha,\beta}(\liegl_{n})$
has a linear basis consisting of the elements
\begin{equation}\label{eq:PBW-basis-u-alpha-beta}
\Ee_{i_{1},j_{1}}^{m_{1}}\Ee_{i_{2},j_{2}}^{m_{2}}
\cdots \Ee_{i_{p},j_{p}}^{m_{p}}
a_{1}^{r_{1}}a_{2}^{r_{2}}\cdots 
a_{n}^{r_{n}} b_{1}^{s_{1}}b_{2}^{s_{2}}
\cdots b_{n}^{s_{n}} \Ff_{k_{1},l_{1}}^{t_{1}}
\Ff_{k_{2},l_{2}}^{t_{2}}
\cdots \Ff_{k_{q},l_{q}}^{t_{q}} , 
\end{equation}
where $ (i_{1},j_{1})< (i_{2},j_{2})<\cdots < (i_{p},j_{p})$
 and
$ (k_{1},l_{1})< (k_{2},l_{2})< \cdots < (k_{q},l_{q})$
are lexicographically ordered
and all powers range between $0$ and $\ell -1$.
Then $\dim {\bf u}_{\alpha,\beta}(\liegl_{n}) =\ell^{n^{2}+n}  $.
\end{proof}

Let $q = \alpha^{-1}\beta$ and denote $h_{i} = a_{i}^{-1}b_{i}$  for all
$1\leq i \leq n$. Since $\alpha, \beta  \in \GG_{\ell}$,
there exist $0< n_{\alpha}, n_{\beta} < \ell$ such
that $q^{n_{\alpha}} = \alpha$ and $q^{n_{\beta}} = \beta$,
by the assumptions in \eqref{eq:cond-alpha-beta}.
Let $\mathcal{I}_{\ell}$ be the ideal of ${\bf u}_{\alpha,
\beta}(\liegl_{n})$ generated by the central group-like elements
$\{h_{i}^{n_{\alpha}}a_{i}^{-1}-1, h_{i}^{n_{\beta}}
b_{i}^{-1}-1\vert\ 1\leq i\leq n\}$ and
define
$$\hat{{\bf u}}_{\alpha, \beta}(\liegl_{n}) := {\bf u}_{\alpha,
\beta}(\liegl_{n}) / \mathcal{I}_{\ell}.$$

\begin{lema}\label{lema:u-hat-pointed}
$\hat{{\bf u}}_{\alpha, \beta}(\liegl_{n})$ 
is a pointed Hopf algebra of dimension
$\ell^{n^{2}}$.
\end{lema}

\begin{proof}
First, $ \hat{{\bf u}}_{\alpha, \beta}(\liegl_{n})$ 
is pointed by \cite[Cor. 5.3.5]{Mo}. Since
$ \hat{{\bf u}}_{\alpha, \beta}(\liegl_{n})$ 
is the quotient of $ {\bf u}_{\alpha, \beta}(\liegl_{n})$
by the ideal $\mathcal{I}_{\ell}$ generated 
by central group-like elements, from
Thm. \ref{thm:PBW-basis-gl} and the 
proof of Lemma \ref{lema:dim-u} 
-- see \eqref{eq:PBW-basis-u-alpha-beta}, 
it follows that $ \hat{{\bf u}}_{\alpha, \beta}(\liegl_{n})$ 
has a linear basis consisting of the elements
\begin{equation}\label{eq:PBW-basis-u-hat--alpha-beta}
\Ee_{i_{1},j_{1}}^{m_{1}}\Ee_{i_{2},j_{2}}^{m_{2}}\cdots 
\Ee_{i_{p},j_{p}}^{m_{p}}
h_{1}^{r_{1}}h_{2}^{r_{2}}\cdots h_{n}^{r_{n}} 
\Ff_{k_{1},l_{1}}^{t_{1}}\Ff_{k_{2},l_{2}}^{t_{2}}
\cdots \Ff_{k_{q},l_{q}}^{t_{q}} , 
\end{equation}
where $ (i_{1},j_{1})< \cdots < (i_{p},j_{p})$ and
$ (k_{1},l_{1})< \cdots < (k_{q},l_{q})$
are lexicographically ordered
and $0\leq m_{i},r_{j},t_{k}\leq \ell -1$.
Then $\dim \hat{{\bf u}}_{\alpha,\beta}(\liegl_{n}) =\ell^{n^{2}}.$
\end{proof}

As in the case of one-parameter deformations of classical objects, the
restricted quantum group $\hat{{\bf u}}_{\alpha,\beta}(\liegl_{n})$ is
associated to the Hopf algebra $\overline{H}$ given by Prop.
\ref{prop:B-central} in terms of a Hopf pairing, induced
from the Hopf pairing between $U_{\alpha,\beta}(\liegl_{n})$ and
$\Oc_{\alpha,\beta}(\liegl_{n})$, see \cite[6.1]{DL},
\cite[III.7.10]{BG}.

\begin{lema}\label{lema:hopf-pairing-restricted}
The Hopf pairing $\langle -,- \rangle: U_{\alpha, \beta}(\liegl_{n})
\times \Oc_{\alpha,\beta}(GL_{n}) \to \k$ from \emph{Subsection
\ref{subsub:hopf-pairing}} induce a Hopf pairing
$
  \langle -,- \rangle': \hat{{\bf u}}_{\alpha, \beta}(\liegl_{n}) \times
  \overline{H} \to \k.$
 In particular, there exists a Hopf algebra map
$\overline{\psi}: \overline{H} \to 
\hat{{\bf u}}_{\alpha, \beta}(\liegl_{n})^{*}$ given by
$\overline{\psi}(\pi(x))(r(u)) =
\langle r(u) , \pi(x) \rangle' $ for all $x\in \Oc_{\alpha,\beta}(GL_{n})$, 
$u \in U_{\alpha, \beta}(\liegl_{n})$ such 
that the following diagram commutes
\begin{equation}\label{diag:H-barra-conmuta}
\xymatrix{ \Oc_{\alpha,\beta}(GL_{n})
\ar[r]^{\psi} \ar@{->>}[d]_{\pi} &
U_{\alpha, \beta}(\liegl_{n})^{\circ} \\
\overline{H} \ar[r]_{\overline{\psi}} & \hat{{\bf u}}_{\alpha,
\beta}(\liegl_{n})^{*} \ar@{^{(}->}[u]_{^{t}r},} 
\end{equation}
where $\psi(x)(u) = \langle u,x\rangle$ for all $u\in 
U_{\alpha, \beta}(\liegl_{n})$ and
$x \in \Oc_{\alpha,\beta}(GL_{n})$.
\end{lema}

\begin{proof}
First note that $ \hat{{\bf u}}_{\alpha, \beta}(\liegl_{n})$ 
is the quotient of 
$U_{\alpha, \beta}(\liegl_{n})$ by the ideal $\mathcal{I}_{n}$ generated 
by the central elements $\Ee_{k,j}^{\ell},\Ff_{k,j}^{\ell},
h_{i}^{\ell}-1, 
h_{i}^{n_{\alpha}}a_{i}^{-1}-1, h_{i}^{n_{\beta}}b_{i}^{-1}-1$ with
$1\leq j\leq k< n$, $1\leq i \leq n$. Thus, 
to see that the pairing $\langle -, - \rangle'$ 
is well-defined, it is enough to
prove $\Oc(GL(n))^{+} \subseteq J_{r}$ and 
$\mathcal{I}_{n} \subseteq J_{l}$, where
 $J_{r}$ and $J_{l}$ are the right and left radicals of the pairing.
Since by Prop. \ref{prop:B-central} $(i)$,  
$\com (x_{st}^{\ell}) = \sum_{k=1}^{n} x_{sk}^{\ell}\ot x_{kt}^{\ell}$,
it suffices to show for the first inclusion that 
$\langle u, x_{st}^{\ell} \rangle = \delta_{st}\eps(u)$ for all
$u = a_{i},b_{i}, e_{j}, f_{j}$, with $1\leq i\leq n$, $1\leq j<n$. 
But this follows from 
the definition in Section 
\ref{subsub:hopf-pairing}; for example,
$$\langle a_{i}, x_{st}^{\ell} \rangle =
\langle \com^{(\ell -1)}(a_{i}),  x_{st}^{\ot \ell} \rangle =
\langle a_{i},  x_{st} \rangle^{\ell}
= \delta_{st} (\alpha^{\delta_{is}})^{\ell} = \delta_{st} 
= \delta_{st}\eps(a_{i}),$$
and since $\com^{(\ell -1)}(e_{i}) 
= \sum_{j=0}^{\ell-1} (a_{i}b_{i+1})^{\ot j}\ot e_{i} \ot
1^{\ot \ell - j - 1}$ it follows that

\begin{align*}
\langle e_{i}, x_{st}^{\ell} \rangle  & = 
\langle \com^{(\ell -1)}(e_{i}),  x_{st}^{\ot \ell} \rangle =
\sum_{j=0}^{\ell-1} \langle (a_{i}b_{i+1})^{\ot j}\ot e_{i} \ot
1^{\ot \ell - j - 1},  x_{st}^{\ot \ell} \rangle \\
& =\sum_{j=0}^{\ell-1} \delta_{st}^{j} \delta_{is}\delta_{i+1,t}
\delta_{st}^{\ell-j-1} (\alpha^{\delta_{is}}\beta^{\delta_{i+1,s}})^{j}\\
& = \delta_{st}^{\ell-1} \delta_{is}\delta_{i+1,t}\sum_{j=0}^{\ell-1} 
(\alpha^{\delta_{is}}\beta^{\delta_{i+1,s}})^{j}
= 0  = \delta_{st}\eps(e_{i}).
\end{align*}
The proof for the other elements is similar. Now, to prove 
that $\mathcal{I}_{n} \subseteq J_{l}$, it is enough to 
show it for the generators of the ideal $\mathcal{I}_{n}$ which
are central elements. But for all $1\leq s,t\leq n$ we have
\begin{align*}
& \langle h_{i}^{\ell}, x_{st} \rangle   = 
\langle h_{i}^{\ot \ell}, \com^{(\ell-1)}x_{st} \rangle
= \sum_{1\leq r_{1},\ldots ,r_{\ell -1}\leq n} 
\langle h_{i}^{\ot \ell},  x_{s,r_{1}}\ot x_{r_{1},r_{2}}\ot 
\cdots \ot x_{r_{\ell-1},t} \rangle \\
 & =\sum_{1\leq r_{m}\leq n} 
\delta_{s,r_{1}}\delta_{r_{1},r_{2}}\cdots \delta_{r_{\ell-1},t}
\alpha^{-(\delta_{i,s} + \delta_{i,r_{1}} 
+ \cdots + \delta_{i,r_{\ell -1}})}
\beta^{\delta_{i+1,s} + \delta_{i+1,r_{1}} 
+ \cdots + \delta_{i+1,r_{\ell -1}}} \\
&  =  \delta_{s,t} \alpha^{-\ell \delta_{i,s}}
\beta^{\ell \delta_{i+1,s}} = \delta_{st}.
\end{align*}
Hence $h_{i}^{\ell} - 1 \in J_{l}$ for all $1\leq i\leq n$. 
Analogously, for all $1\leq s,t\leq n$ we have
\begin{align*}
\langle h_{i}^{n_{\alpha}}a_{i}^{-1}, x_{st} \rangle  
& = \sum_{m=1}^{n} \langle h_{i}^{n_{\alpha}}, x_{sm} \rangle
 \langle a_{i}^{-1}, x_{mt}\rangle \\
& = \sum_{m=1}^{n} \delta_{s,m} \alpha^{-n_{\alpha}\delta_{i,s}}
\beta^{n_{\alpha}\delta_{i,m}} \delta_{m,t}\alpha^{-\delta_{i,m}}
=\delta_{s,t} \alpha^{-n_{\alpha}\delta_{i,s}}
\beta^{n_{\alpha}\delta_{i, s}} \alpha^{-\delta_{i,s}}\\
& =\delta_{s,t} q^{n_{\alpha}\delta_{i,s}} \alpha^{-\delta_{i,s}}
= \delta_{s,t} \alpha^{\delta_{i,s}} \alpha^{-\delta_{i,s}} 
= \delta_{s,t},
\end{align*}
where the second equality follows from the calculations
made for $h_{i}^{\ell}$. Thus, 
$h_{i}^{n_{\alpha}}a_{i}^{-1} - 1 \in J_{l}$ for all $1\leq i\leq n$.
By a similar calculation one can show that 
$h_{i}^{n_{\beta}}b_{i}^{-1} - 1 \in J_{l}$ for all $1\leq i\leq n$.
Indeed, for all $1\leq s,t\leq n$ we have
\begin{align*}
\langle h_{i}^{n_{\beta}}b_{i}^{-1}, x_{st} \rangle  
& = \sum_{m=1}^{n} \langle h_{i}^{n_{\beta}}, x_{sm} \rangle
 \langle b_{i}^{-1}, x_{mt}\rangle 
 =\delta_{s,t} \alpha^{-n_{\beta}\delta_{i,s}}
\beta^{n_{\beta}\delta_{i, s}} \beta^{-\delta_{i,s}}\\
& =\delta_{s,t} q^{n_{\beta}\delta_{i,s}} \beta^{-\delta_{i,s}}
 = \delta_{s,t} \beta^{\delta_{i,s}} \beta^{-\delta_{i,s}} = \delta_{s,t}.
\end{align*}
Finally, we have to show that $\Ee_{k,j}^{\ell}$ and $\Ff_{k,j}^{\ell}$
are in $J_{l}$ for all $1\leq j \leq k <n$. Let $1\leq s,t\leq n$, then

\begin{align*}
& \langle \Ee_{k,j}^{\ell}, x_{st} \rangle
= \sum_{1\leq r_{1},\ldots ,r_{\ell -1}\leq n} 
\langle \Ee_{k,j},  x_{s,r_{1}}\rangle \langle  \Ee_{k,j},x_{r_{1},r_{2}} 
\rangle \cdots \langle  \Ee_{k,j}, x_{r_{\ell-1},t} \rangle \\
&  \quad =\sum_{1\leq r_{m}\leq n} 
(-1)^{\ell (k-j)}\alpha^{\ell(k-i)}
\delta_{j,s}\delta_{k+1,r_{1}}\delta_{j,r_{1}}\delta_{k+1,r_{2}}\cdots 
\delta_{j,r_{\ell-1}}\delta_{k+1,t}\\
&  \quad =  (-1)^{\ell (k-j)}\delta_{s,t} \delta_{k+1,j} \delta_{k+1,t} = 0,
\end{align*}
where the second equality follows from 
Lemma \ref{lema:pair-e-f} $(i)$
and the last one from  $1\leq j \leq k <n$. Hence
$\Ee_{k,j}^{\ell} \in J_{l}$ for all $1\leq j \leq k <n$. 
Moreover, 
\begin{align*}
& \langle \Ff_{k,j}^{\ell}, x_{st} \rangle
= \sum_{1\leq r_{1},\ldots ,r_{\ell -1}\leq n} 
\langle \Ff_{k,j},  x_{s,r_{1}}\rangle \langle  \Ff_{k,j},x_{r_{1},r_{2}} 
\rangle \cdots \langle  \Ff_{k,j}, x_{r_{\ell-1},t} \rangle \\
& \quad =\sum_{1\leq r_{m}\leq n} 
\delta_{k+1,s}\delta_{j,r_{1}}\delta_{k+1,r_{1}}\delta_{j,r_{2}}\cdots 
\delta_{k+1,r_{\ell-1}}\delta_{j,t}
=  \delta_{k+1,s} \delta_{j,k+1} \delta_{j,t} = 0,
\end{align*}
where the second equality follows from 
Lemma \ref{lema:pair-e-f} $(ii)$
and the last one from  $1\leq j \leq k <n$. Hence
$\Ff_{k,j}^{\ell} \in J_{l}$ for all $1\leq j \leq k <n$ and
thus $\mathcal{I}_{n} \subseteq J_{l}$.

\par Since $\Oc(GL(n))^{+} \subseteq J_{r}$, 
there exists a Hopf algebra map
$\overline{\psi}: \overline{H} \to 
U_{\alpha, \beta}(\liegl_{n})^{\circ} $ such that 
$\overline{\psi}\circ\pi = \psi$. 
Thus, to prove the last assertion, we need to show that
the $\Img \overline{\psi}\subseteq\ ^{t}r (\hat{{\bf u}}_{\alpha,
\beta}(\liegl_{n})^{*})$. 
But since $\mathcal{I}_{n} \subseteq J_{l}$, it follows
that $\overline{\psi}(\bar{x}_{st})(h) 
= \langle x_{st}, h \rangle = 0$ for all
$h\in \mathcal{I}_{n}$ and the map 
$\overline{\psi}(\bar{x}_{st})(-)$ given by
$$\overline{\psi}(\bar{x}_{st})(r(h)) = 
\langle r(h), \bar{x}_{st}  \rangle' = \langle h, x_{st} \rangle = 
\overline{\psi}(\bar{x}_{st})(h),$$ 
defines an element in 
$\hat{{\bf u}}_{\alpha,
\beta}(\liegl_{n})^{*}$ for all $1\leq s,t\leq n$. 
\end{proof}

\subsubsection{Connectedness}
The following proposition is a key step for the construction of the
quotients. In the case of one-parameter deformations, this result 
is well-known, see \cite[III.7.10]{BG}. In the terminology of Takeuchi
\cite{Tk2}, the proposition says that the Hopf algebra $\overline{H}$
is connected.
Since it is finite-dimensional this also proves that it is isomorphic
to the dual of the restricted (pointed) quantum group 
$\hat{{\bf u}}_{\alpha, \beta}(\liegl_{n})$. Since the proof
of this fact is rather technical,
for the sake
of clearness, we divide it 
in several lemmata.

\begin{prop}\label{prop:psi-injective}
 $\overline{\psi}:\overline{H} \to \hat{{\bf u}}_{\alpha,
     \beta}(\liegl_{n})^{*}$ is injective and 
     $\overline{H} \simeq \hat{{\bf u}}_{\alpha, \beta}(\liegl_{n})^{*}$.
 \end{prop}

As pointed out in the proof of Lemma \ref{lema:hopf-pairing-restricted},
$\hat{{\bf u}}_{\alpha,  \beta}(\liegl_{n})$ is the quotient 
of $U_{\alpha, \beta}(\liegl_{n})$ by the ideal 
$\mathcal{I}_{n}$ generated by the central elements 
$\Ee_{k,j}^{\ell},\Ff_{k,j}^{\ell},
h_{i}^{\ell}-1, 
h_{i}^{n_{\alpha}}a_{i}^{-1}-1, h_{i}^{n_{\beta}}b_{i}^{-1}-1$ with
$1\leq j\leq k< n$, $1\leq i \leq n$. Hence, the triangular decomposition 
\eqref{eq:desc-triangular-gl} of $U_{\alpha, \beta}(\liegl_{n})$ induces 
a triangular decomposition 
\begin{equation}\label{eq:desc-triangular-u-hat}
\hat{{\bf u}}_{\alpha,  \beta}(\liegl_{n}) \cong
\hat{{\bf u}}_{\alpha,  \beta}^{-}(\liegl_{n})\ot 
\hat{{\bf u}}_{\alpha,  \beta}^{\circ}(\liegl_{n})\ot 
\hat{{\bf u}}_{\alpha,  \beta}^{+}(\liegl_{n}),
\end{equation}
where $\hat{{\bf u}}_{\alpha,  \beta}^{\pm}(\liegl_{n}) 
= r(U_{\alpha,
\beta}^{\pm}(\liegl_{n}))$ and 
$\hat{{\bf u}}_{\alpha,  \beta}^{\circ}(\liegl_{n}) = r(U_{\alpha,
\beta}^{\circ}(\liegl_{n}))$, and $r$ is the 
canonical map, see Lemmata \ref{lema:dim-u}
and \ref{lema:u-hat-pointed}. In particular, these 
subalgebras are generated by $\{r(e_{i})\}_{1\leq i <n}$,
$\{r(f_{i})\}_{1\leq i <n}$ and $\{r(h_{j})\}_{1\leq j \leq n}$,
 respectively.
Take $\hat{{\bf u}}_{\alpha, \beta}(\lieb^{+}) 
= r(U_{\alpha, \beta}(\lieb^{+}))$ 
and $\hat{{\bf u}}_{\alpha, \beta}(\lieb^{-}) 
= r(U_{\alpha, \beta}(\lieb^{-}))$. These
subalgebras of $\hat{{\bf u}}_{\alpha,  \beta}(\liegl_{n})$ are
generated by the
elements $\{ r(e_{j}), r(w_{j}^{\pm 1}), 
r(a_{n})^{\pm 1}\vert\ 1\leq j <n\}$
and $\{ r(f_{j}), r((w_{j}')^{\pm 1}), r(b_{n}^{\pm 1})\vert\ 
1\leq j <n\}$, respectively. Then, by the triangular decomposition
\eqref{eq:desc-triangular-u-hat}, the map given 
by the multiplication
$m: \hat{{\bf u}}_{\alpha, \beta}(\lieb^{+}) 
\ot \hat{{\bf u}}_{\alpha, \beta}(\lieb^{-})
\to \hat{{\bf u}}_{\alpha, \beta}(\liegl_{n})$ 
is surjective. Thus, its transpose
\begin{equation}\label{eq:transp-mult-u-hat}
^{t}m: \hat{{\bf u}}_{\alpha, \beta}(\liegl_{n})^{*} 
\to \hat{{\bf u}}_{\alpha,
\beta}(\lieb^{+})^{*} \ot \hat{{\bf u}}_{\alpha, \beta}(\lieb^{-})^{*},
\end{equation}
defines an injective map. 

\par Let $K_{+} := \overline{H} / \pi (J_{+})$ and 
$K_{-} := \overline{H} / \pi (J_{-})$ (see Subsection 
\ref{subsubsec:borel-qsubgroups}), 
and denote by $\bar{t}_{\pm} : \overline{H} 
\to K_{\pm}$ the Hopf algebra
quotients. For example, 
$K_{+}$ is generated as an algebra by the
elements $\{\bar{x}_{ij}\vert \ 1\leq i\leq j\leq n\}$ satisfying
the relations
\begin{eqnarray}
    \nonumber \bar{x}_{ij}^{\ell}  &=&\delta_{ij}
    \qquad\mbox{ for all }i\leq j,  \\
    \nonumber \bar{x}_{ik}\bar{x}_{ij} &=&\alpha^{-1} \bar{x}_{ij}\bar{x}_{ik}
    \qquad\mbox{ if }j<k,  \\
    \nonumber \bar{x}_{jk}\bar{x}_{ik}&= &\beta \bar{x}_{ik}\bar{x}_{jk}
    \qquad\mbox{ if }i<j,  \\
    \nonumber \bar{x}_{jk}\bar{x}_{il}&=&\beta\alpha\bar{x}_{il}\bar{x}_{jk}
    \qquad\mbox{ and } \\
    \nonumber \bar{x}_{jl}\bar{x}_{ik} & = & \bar{x}_{ik}\bar{x}_{jl}
    \qquad\mbox{ if }i<j \mbox{ and }k<l.
  \end{eqnarray}
Moreover, the elements $\{\bar{x}_{ii}\}_{1\leq i\leq n}$ are
invertible group-like elements which commute with each other 
and the set $\{\prod_{i< j}^{} \bar{x}_{ij}^{d_{ij}}\prod_{i
}^{} \bar{x}_{ii}^{d_{i}} \vert\ 0\leq d_{ij}, d_{i} < \ell \}$ 
is a linear basis of $K_{+}$,
for some fixed ordering of $\{\bar{x}_{ij}\vert 1\leq i < j \leq n\}$. 
Thus, by the very definition of this quotients, there exist
surjective Hopf algebra maps 
$\pi_{\pm}: \Oc_{\alpha,\beta}(B^{\pm }) \to K_{\pm}$
such that the following diagram commute 
\begin{equation}\label{diag:bar-h-K-mas}
\xymatrix{ \Oc_{\alpha,\beta}(GL_{n})
  \ar[r]^{t_{\pm}} \ar[d]_{\pi} &
 \Oc_{\alpha,\beta}(B^{\pm }) \ar[d]^{\pi_{\pm}}\\
  \overline{H} \ar[r]^{\bar{t}_{\pm}} & K_{\pm}.}
 \end{equation}
Now let $\overline{\delta}:= (\bar{t}_{+} \ot \bar{t}_{-})\com: 
\overline{H} \to K_{+} \ot K_{-}$. 
Then the commutativity of diagram \eqref{diag:bar-h-K-mas} implies that
the diagram
\begin{equation}\label{diag:bar-delta}
\xymatrix{ \Oc_{\alpha,\beta}(GL_{n})
  \ar[r]^{\pi} \ar[d]_{\delta} & \overline{H}
 \ar[d]^{\overline{\delta}}\\
  \Oc_{\alpha,\beta}(B^{+ }) \ot \Oc_{\alpha,\beta}(B^{- })  
\ar[r]^(.65){\pi_{+}\ot\pi_{-}} & K_{+}\ot K_{-}.}
 \end{equation}
commutes. This implies that $\overline{\delta}$ is injective. 
Indeed,
let $h \in \overline{H}$ and $x \in  \Oc_{\alpha,\beta}(GL_{n})$ such 
that $\pi(x) = h$.  If $\overline{\delta}(h) = 0$, then 
\begin{align*}
0 &= (\bar{t}_{+} \ot \bar{t}_{-}) \com ( \pi(x)) 
= (\bar{t}_{+} \ot \bar{t}_{-})(\pi\ot \pi)\com ( x) =
(\bar{t}_{+}\pi \ot \bar{t}_{-}\pi)\com ( x) \\
& = (\pi_{+}t_{+} \ot \pi_{-}t_{-})\com(x)  
= (\pi_{+}\ot \pi_{-})\delta(x),
\end{align*}
where the fourth and the fifth 
equalities follow from the commutativity
of diagrams \eqref{diag:bar-h-K-mas} and 	
\eqref{diag:bar-delta}. Thus 
$\delta(x) \in \Ker \pi_{+}\ot \pi_{-}$. But since $\delta$ is 
injective by Lemma \ref{lema:delta-injective}, 
it follows that $x \in \Ker \pi$ and
whence $h = 0$.

\begin{lema}\label{lema:existe-psi-+}
There
exist Hopf algebra maps $\overline{\psi}_{\pm}: 
K_{\pm} \to \hat{{\bf u}}_{\alpha,
\beta}(\lieb^{\pm})^{*}$ such that the following diagram commutes 
\begin{equation}\label{diag:bar-psi-plus-minus}
\xymatrix{
  \overline{H} \ar[r]^(.4){\overline{\psi}} \ar@{->>}[d]_{\bar{t}_{\pm}} &
 \hat{{\bf u}}_{\alpha,\beta}(\liegl_{n})^{*} 
\ar@{->>}[d]^{}\\
K_{\pm}  \ar[r]^(.4){\overline{\psi}_{+}} & \hat{{\bf u}}_{\alpha,
\beta}(\lieb^{\pm})^{*}.}
 \end{equation}
\end{lema}

\pf We prove it only for $\overline{\psi}_{+}$ since 
the proof for $\overline{\psi}_{-}$ is completely analogous.
Consider 
$\bar{x}_{ij} \in \pi(I_{+}) \subseteq \overline{H}$, 
that is, $i>j$.
Then $\overline{\psi}(\bar{x}_{ij})(r(u)) = 
\langle r(u), \bar{x}_{ij}  \rangle' = \langle u, x_{ij} \rangle = 0$ 
for all $u \in \hat{u}_{\alpha,
\beta}(\lieb^{+})$, since by the definition of the pairing 
it holds for the generators of this Hopf subalgebra. 
This implies that there
exists a Hopf algebra map $\overline{\psi}_{+}: 
K_{+} \to \hat{{\bf u}}_{\alpha,
\beta}(\lieb^{+})^{*}$ such that the diagram 
\eqref{diag:bar-psi-plus-minus} commutes. 
\epf

The following lemma will be needed also for the proof
of Prop. \ref{prop:psi-injective}. Regrettably,
it is quite technical and its proof involves lots of
calculations, but we think the computations can not be avoid,
which seems to be in general the case in the theory of quantum groups.
In order to make the notation not so heavy, we write also 
by $ \Ee_{i,j} $ the elements of 
$\hat{{\bf u}}_{\alpha,
\beta}(\lieb^{+})$. To prove it, we apply
the method in \cite{Tk2}.
 
\par Let $M = (M_{ij})_{1\leq i,j\leq n}$ and 
$N = (N_{ij})_{1\leq i,j\leq n}$ be upper triangular
matrices and denote 
$\Ee^{M}= \Ee^{M_{n}}\cdots \Ee^{M_{1}}$  and 
$\bar{x}^{N}= \bar{x}^{N_{1}}\cdots \bar{x}^{N_{n}}
$, where $\Ee^{M_{i}} = 
\Ee_{n-1,i}^{M_{in}}\cdots \Ee_{i,i}^{M_{i+1,i}}$ and
$\bar{x}^{N_{i}} = \bar{x}_{ii}^{N_{ii}}\cdots \bar{x}_{in}^{N_{i,n}}$. 
Consider the subalgebras $A$, $B$
of $K_{+}$ and $U$, $V$ of 
$ \hat{{\bf u}}_{\alpha,\beta}(\lieb^{+}) $
given by 
\begin{align*}
A & =\CC\{\bar{x}_{ij}, \bar{x}_{kk}^{-1}\vert\ 
1< j,k\leq n\},\qquad\  
B =\CC\{\bar{x}_{1,j} , \bar{x}_{11}^{-1}\vert\ 
1\leq j\leq n\},\\
U & =\CC\{w_{t}, \Ee_{i,j}\vert\ 1<j , 1<t  \},\qquad\quad
V  =\CC\{\Ee_{i,1}\vert\ 1\leq i\leq n-1\},
\end{align*}
where the notation means that the 
elements are generators. 
Note that $A$ and $U$ are Hopf subalgebras
and $ \com(B) \subseteq B\ot K_{+} $, $ \com(V) \subseteq  
\hat{{\bf u}}_{\alpha,\beta}(\lieb^{+}) \ot V$.

\begin{lema}\label{lema:e-M-x-N} For $1\leq r,s < \ell$
and $a\in A$, $ b\in B $, $ u\in U $ and $ v\in V $,
\begin{enumerate}
\item[$(i)$] $\langle \Ee_{n-1,1}^{r}, \bar{x}_{1,n}^{s}
\bar{x}_{1,k_{s+1}}\cdots \bar{x}_{1,k_{t}} \rangle = 
\langle \Ee_{n-1,1}^{r}, \bar{x}_{1,n}^{s}\rangle 
\eps(\bar{x}_{1,k_{s+1}}\cdots \bar{x}_{1,k_{t}})$, 
for all $k_{l}\neq n $, with $s+1\leq l\leq t$ and $0\leq s$,
where $\langle \Ee_{n-1,1}^{r}, \bar{x}_{1,n}^{s}\rangle=$ $ 
\delta_{r,s}\beta^{r(r-1)/2}\langle \Ee_{n-1,1}, 
\bar{x}_{1,n}\rangle^{r}\prod_{j=0}^{r-1}\Phi_{r-j}(\alpha^{2}) 
$.
\item[$(ii)$] $\langle v,a\rangle = \eps(v)\eps(a)$ and 
$\langle u,b\rangle = \eps(u)\eps(b)$.
\item[$(iii)$] $\langle \Ee^{M}, \bar{x}^{N} \rangle= \lambda
\delta_{M,N} $, where $\lambda$ is a non-zero scalar.
\end{enumerate}
\end{lema}

\pf $ (i) $
First we need to do some computations that 
will be needed in the sequel. For all 
$1\leq m_{j}\leq n$, $1\leq j\leq t$ we have 
\begin{enumerate}
\item[$ (a) $]  
$ \langle \Ee_{n-1,1}, 
\bar{x}_{1,m_{1}}\cdots \bar{x}_{1,m_{t}} \rangle = 
\langle \Ee_{n-1,1}, 
\bar{x}_{1,n} \rangle 
\sum_{i=1}^{t}\delta_{n, m_{i}}\alpha^{i-1}
\prod_{j\neq i}\delta_{1,m_{j}}$,
\item[$ (b) $] $\langle \Ee_{n-1,1}, 
\bar{x}_{n,n}^{r}\bar{x}_{1, m_{1}}\cdots \bar{x}_{1,m_{t}} \rangle =
\beta^{r} \langle \Ee_{n-1,1}, 
\bar{x}_{1,m_{1}}\cdots \bar{x}_{1,m_{t}} \rangle  $.
\end{enumerate}
Indeed, for $(a)$ we have
\begin{align*}
& \langle \Ee_{n-1,1}, 
\bar{x}_{1,m_{1}}\cdots \bar{x}_{1,m_{t}} \rangle = \\
& \quad \langle \Ee_{n-1,1}, 
\bar{x}_{1,m_{1}} \rangle 
\eps(\bar{x}_{1,m_{2}}\cdots \bar{x}_{1,m_{t}})+ 
 \langle w_{n-1,1}, 
\bar{x}_{1,m_{1}} \rangle  \langle \Ee_{n-1,1}, 
\bar{x}_{1,m_{2}}\cdots \bar{x}_{1,m_{t}}\rangle +\\
& \qquad + \zeta \sum_{j=1}^{n-2}
\langle \Ee_{n-1,j+1}w_{j,1}, \bar{x}_{1,m_{1}}\rangle 
\langle \Ee_{j,1},
\bar{x}_{1,m_{2}}\cdots \bar{x}_{1,k_{t}}\rangle\\
& \qquad = \delta_{n, m_{1}}\delta_{1,m_{2}}
\cdots \delta_{1,_{m_{t}}} 
\langle \Ee_{n-1,1} , 
\bar{x}_{1,n} \rangle +\alpha 
\delta_{1, m_{1}} \langle \Ee_{n-1,1}, 
\bar{x}_{1,m_{2}}\cdots \bar{x}_{1,m_{t}}\rangle + \\
& \qquad + \zeta \sum_{j=1}^{n-2}(\sum_{k=1}^{n}
\langle \Ee_{n-1,j+1},\bar{x}_{1,k}\rangle 
\langle w_{j,1}, \bar{x}_{k,m_{1}}\rangle) 
\langle \Ee_{j,1},
\bar{x}_{1,m_{2}}\cdots \bar{x}_{1,m_{t}}\rangle\\
& \qquad = \delta_{n, m_{1}}\delta_{1,m_{2}}
\cdots \delta_{1,_{m_{t}}}
\langle \Ee_{n-1,1}, 
\bar{x}_{1,n} \rangle + \alpha \delta_{1, m_{1}} \langle \Ee_{n-1,1}, 
\bar{x}_{1,m_{2}}\cdots \bar{x}_{1,m_{t}}\rangle\\
& \qquad = \big(\delta_{n, m_{1}}\delta_{1,m_{2}}
\cdots \delta_{1,_{m_{t}}}
+ \alpha \delta_{1, m_{1}} \delta_{n,m_{2}}
\cdots \delta_{1,_{m_{t}}} +\\
&\qquad\quad + \alpha^{2} \delta_{1, m_{1}}
\delta_{1, m_{2}}\delta_{n,m_{3}}
\delta_{1,m_{4}}\cdots \delta_{1,_{m_{t}}} + \ldots +\\
&\qquad\quad + \alpha^{t-1} \delta_{1, m_{1}}
\cdots \delta_{1, m_{t-1}} \delta_{n,m_{t}}
\big) \langle \Ee_{n-1,1}, 
\bar{x}_{1,n} \rangle\\
& \qquad = \langle \Ee_{n-1,1}, 
\bar{x}_{1,n} \rangle 
\sum_{i=1}^{t}\delta_{n, m_{i}}\alpha^{i-1}
\prod_{j\neq i}\delta_{1,m_{j}},
\end{align*}
where the third equality follows from
Lemma \ref{lema:pair-e-f} $(i)$ and the fourth equality 
from the preceeding equalities. 
For $ (b) $ we have for all $r\geq 1$ that
\begin{align*}
& \langle \Ee_{n-1,1}, \bar{x}_{n,n}^{r}\bar{x}_{1,m_{1}}\cdots 
\bar{x}_{1,m_{t}} \rangle  =\\
&\qquad
 \langle \Ee_{n-1,1}, 
\bar{x}_{n,n}^{r}\rangle \eps(\bar{x}_{1,m_{1}}
\cdots \bar{x}_{1,m_{t}})
+ \langle w_{n-1,1}, 
\bar{x}_{n,n}^{r} \rangle \langle \Ee_{n-1,1}, 
\bar{x}_{1,m_{1}}\cdots \bar{x}_{1,m_{t}} \rangle 
+ \\
& \qquad\quad 
+ \zeta \sum_{j=1}^{n-2}
\langle \Ee_{n-1,j+1}w_{j,1}, \bar{x}_{n,n}^{r}\rangle 
\langle \Ee_{j,1},
\bar{x}_{1,m_{1}}\cdots \bar{x}_{1,m_{t}}\rangle\\
& \qquad = \langle \Ee_{n-1,1}, 
\bar{x}_{n,n}^{r}\rangle \eps(\bar{x}_{1,m_{1}}
\cdots \bar{x}_{1,m_{t}}) 
+ \beta^{r} \langle \Ee_{n-1,1}, 
\bar{x}_{1,m_{1}}\cdots \bar{x}_{1,m_{t}} \rangle\\
& \qquad =
\beta^{r} \langle \Ee_{n-1,1}, 
\bar{x}_{1,m_{1}}\cdots \bar{x}_{1,m_{t}} \rangle,
\end{align*}
where the second and the third equalities 
follows from the fact that the element $\bar{x}_{n,n}$ 
is group-like and the elements $\Ee_{i,j}$ are nilpotent in
$\hat{{\bf u}}_{\alpha,
\beta}(\lieb^{+})$.

\par Now we proceed with the proof.  
We prove it by induction on $s$. Suppose first that $s=0$, 
then, we have to prove that
$\langle \Ee_{n-1,1}^{r},
\bar{x}_{1,k_{1}}\cdots \bar{x}_{1,k_{t}} \rangle $ $= 0$ for all $r\geq 1$. 
The case $ r=1$ follows from $(a)$, since $k_{j}\neq n$ for all $1\leq j\leq t$.
Let $r>1$, then  
\begin{align*}
& \langle \Ee_{n-1,1}^{r}, \bar{x}_{1,k_{1}}
\cdots \bar{x}_{1,k_{t}} \rangle  = \\
& = 
\sum_{m_{j}= 1}^{n}\langle \Ee_{n-1,1}, \bar{x}_{1,m_{1}}
\cdots \bar{x}_{1,m_{t}} \rangle 
\langle \Ee_{n-1,1}^{r-1}, \bar{x}_{m_{1},k_{1}}
\cdots \bar{x}_{m_{t},k_{t}} \rangle
\end{align*}
\begin{align*}
&  = \sum_{m_{j}= 1}^{n}
\big(\langle \Ee_{n-1,1}, \bar{x}_{1,m_{1}} \rangle 
\eps( \bar{x}_{1,m_{2}}\cdots \bar{x}_{1,m_{t}}) 
+\\
&\quad 
+ \langle w_{n-1,1}, \bar{x}_{1,m_{1}} \rangle
\langle \Ee_{n-1,1},   \bar{x}_{1,m_{2}}
\cdots \bar{x}_{1,m_{t}}\rangle +\\
& \quad + 
\zeta \sum_{j=1}^{n-2}
\langle \Ee_{n-1,j+1}w_{j,1}, 
\bar{x}_{1,m_{1}}\rangle \langle \Ee_{j,1},
\bar{x}_{1,m_{2}}\cdots \bar{x}_{1,m_{t}}\rangle\big)  
\langle \Ee_{n-1,1}^{r-1}, \bar{x}_{m_{1},k_{1}}\cdots 
\bar{x}_{m_{t},k_{t}} \rangle\\
& = \sum_{m_{j}= 1}^{n}
\langle \Ee_{n-1,1}, x_{1,m_{1}} \rangle 
\eps( \bar{x}_{1,m_{2}}\cdots \bar{x}_{1,m_{t}}) 
\langle \Ee_{n-1,1}^{r-1}, \bar{x}_{m_{1},k_{1}}\cdots 
\bar{x}_{m_{t},k_{t}} \rangle +  \\
&\quad + \sum_{m_{j}= 1}^{n}
\langle w_{n-1,1}, x_{1,m_{1}} \rangle
\langle \Ee_{n-1,1},  \bar{x}_{1,m_{2}}
\cdots \bar{x}_{1,m_{t}} \rangle   
\langle \Ee_{n-1,1}^{r-1}, \bar{x}_{m_{1},k_{1}}\cdots 
\bar{x}_{m_{t},k_{t}} \rangle\\
& = (-1)^{n-2}\alpha^{2-n} 
\langle \Ee_{n-1,1}^{r-1}, 
\bar{x}_{n,k_{1}}\bar{x}_{1,k_{2}}
\cdots \bar{x}_{1,k_{t}} \rangle +  \\
&\quad + \sum_{m_{j}= 1}^{n} \alpha
\langle \Ee_{n-1,1},  \bar{x}_{1,m_{2}}
\cdots \bar{x}_{1,m_{t}} \rangle   
\langle \Ee_{n-1,1}^{r-1}, \bar{x}_{1,k_{1}}
\bar{x}_{m_{2},k_{2}} \cdots \bar{x}_{m_{t},k_{t}} \rangle\\
& = \sum_{m_{j}= 1}^{n} \alpha
\langle \Ee_{n-1,1},  \bar{x}_{1,m_{2}}
\cdots \bar{x}_{1,m_{t}} \rangle
\langle 
\Ee_{n-1,1}^{r-1}, \bar{x}_{1,k_{1}}\bar{x}_{m_{2},k_{2}}
\cdots \bar{x}_{m_{t},k_{t}} \rangle\\
&= \alpha
\langle \Ee_{n-1,1},  \bar{x}_{1,n}\rangle \big( \alpha 
\langle \Ee_{n-1,1}^{r-1}, \bar{x}_{1,k_{1}}
\bar{x}_{n,k_{2}}\bar{x}_{1,k_{3}}
\cdots \bar{x}_{m_{t},k_{t}} \rangle + \ldots +\\
&\quad\quad +
\alpha^{t-2} 
\Ee_{n-1,1}^{r-1}, \bar{x}_{1,k_{1}}\bar{x}_{1,k_{2}}
\cdots \bar{x}_{1,k_{t-1}}\bar{x}_{n,k_{t}} \rangle\big)
= 0,
\end{align*}
where the second equality follows from the comultiplication of
$ \Ee_{n-1,1}$, the third from Lemma \ref{lema:pair-e-f} $(i)$,
the fifth from the fact that $\bar{x}_{n,k_{1}} = 0$, because
$k_{1}\neq n$, the sixth by $ (a) $ 
and the last equality
from the fact that $\bar{x}_{n,k_{j}} = 0$ because
$k_{j}\neq n$. Now let $s>0$. For $r = 1$
we have that 
\begin{align*}
& \langle \Ee_{n-1,1}, 
\bar{x}_{1,n}^{s}\bar{x}_{1,k_{s+1}}
\cdots \bar{x}_{1,k_{t}} \rangle = 
 \langle \Ee_{n-1,1}, 
\bar{x}_{1,n} \rangle 
\eps(\bar{x}_{1,n}^{s-1}\bar{x}_{1,k_{s+1}}
\cdots \bar{x}_{1,k_{t}})+\\
& \qquad +
\langle w_{n-1,1}, 
\bar{x}_{1,n} \rangle \langle \Ee_{n-1,1}, 
\bar{x}_{1,n}^{s-1}\bar{x}_{1,k_{s+1}}
\cdots \bar{x}_{1,k_{t}} \rangle\\
&\qquad + \zeta \sum_{j=1}^{n-2}
\langle \Ee_{n-1,j+1}w_{j,1}, 
\bar{x}_{1,n}\rangle 	\langle \Ee_{j,1},
\bar{x}_{1,n}^{s-1}\bar{x}_{1,k_{s+1}}
\cdots \bar{x}_{1,k_{t}}\rangle\\
& \quad =   \langle \Ee_{n-1,1}, 
\bar{x}_{1,n} \rangle \delta_{s,1}
\eps(\bar{x}_{1,k_{s+1}}\cdots \bar{x}_{1,k_{t}}) +\\
&\qquad + \zeta \sum_{j=1}^{n-2}(\sum_{m=1}^{n}
\langle \Ee_{n-1,j+1},\bar{x}_{1,m}
\rangle \langle w_{j,1}, \bar{x}_{m,n}\rangle) 
\langle \Ee_{j,1},
\bar{x}_{1,n}^{s-1}\bar{x}_{1,k_{s+1}}
\cdots \bar{x}_{1,k_{t}}\rangle\\
&  \quad = \langle \Ee_{n-1,1}, 
\bar{x}_{1,n} \rangle \delta_{s,1}
\eps(\bar{x}_{1,k_{s+1}}\cdots \bar{x}_{1,k_{t}})
= \langle \Ee_{n-1,1}, 
\bar{x}_{1,n}^{s} \rangle 
\eps(\bar{x}_{1,k_{s+1}}\cdots \bar{x}_{1,k_{t}}),
\end{align*} 
where the last equality follows from part $ (i) $.
Assume now that $r>1 $ and $s>1$. Then

\begin{align*}
& \langle \Ee_{n-1,1}^{r}, 
\bar{x}_{1,n}^{s}\bar{x}_{1,k_{s+1}}
\cdots \bar{x}_{1,k_{t}} \rangle = \\
& \sum_{m_{j}= 1}^{n}
\langle \Ee_{n-1,1}, \bar{x}_{1,m_{1}}
\cdots \bar{x}_{1,m_{t}} \rangle 
\langle \Ee_{n-1,1}^{r-1}, \bar{x}_{m_{1},n}
\cdots \bar{x}_{m_{s},n} 
\bar{x}_{m_{s+1},k_{s+1}}\cdots 
\bar{x}_{m_{t},k_{t}} \rangle\\
& = 
\langle \Ee_{n-1,1}, \bar{x}_{1,n}\rangle 
\langle \Ee_{n-1,1}^{r-1}, 
\bar{x}_{n,n}\bar{x}_{1,n}\cdots \bar{x}_{1,n} 
\bar{x}_{1,k_{s+1}}\cdots \bar{x}_{1,k_{t}} \rangle +\\
&\quad + \alpha\langle \Ee_{n-1,1}, \bar{x}_{1,n}\rangle 
\langle \Ee_{n-1,1}^{r-1}, \bar{x}_{1,n}
\bar{x}_{n,n}\bar{x}_{1,n}\cdots \bar{x}_{1,n} 
\bar{x}_{1,k_{s+1}}\cdots 
\bar{x}_{1,k_{t}} \rangle +\cdots +\\
&\quad + \alpha^{s}
\langle \Ee_{n-1,1}, \bar{x}_{1,n}\rangle 
\langle \Ee_{n-1,1}^{r-1}, 
\bar{x}_{1,n}\bar{x}_{1,n}\cdots \bar{x}_{1,n} 
\bar{x}_{n,k_{s+1}}\cdots \bar{x}_{1,k_{t}} \rangle
 +\cdots +\\
&\quad + \alpha^{t-1}\langle 
\Ee_{n-1,1}, \bar{x}_{1,n}\rangle 
\langle \Ee_{n-1,1}^{r-1}, 
\bar{x}_{1,n}\bar{x}_{1,n}\cdots \bar{x}_{1,n} 
\bar{x}_{1,k_{s+1}}\cdots \bar{x}_{n,k_{t}} \rangle\\
& = 
\langle \Ee_{n-1,1}, \bar{x}_{1,n}\rangle 
\langle \Ee_{n-1,1}^{r-1}, 
\bar{x}_{n,n}\bar{x}_{1,n}\cdots \bar{x}_{1,n} 
\bar{x}_{1,k_{s+1}}\cdots \bar{x}_{1,k_{t}} \rangle +\\
&\quad + \alpha\langle \Ee_{n-1,1}, \bar{x}_{1,n}\rangle 
\langle \Ee_{n-1,1}^{r-1}, \bar{x}_{1,n}
\bar{x}_{n,n}\bar{x}_{1,n}\cdots \bar{x}_{1,n} 
\bar{x}_{1,k_{s+1}}\cdots 
\bar{x}_{1,k_{t}} \rangle + \cdots +\\
&\quad + \alpha^{s-1}
\langle \Ee_{n-1,1}, \bar{x}_{1,n}\rangle 
\langle \Ee_{n-1,1}^{r-1}, 
\bar{x}_{1,n}\bar{x}_{1,n}\cdots \bar{x}_{n,n} 
\bar{x}_{1,k_{s+1}}\cdots \bar{x}_{1,k_{t}} \rangle\\
& = 
\langle \Ee_{n-1,1}, \bar{x}_{1,n}\rangle 
\langle \Ee_{n-1,1}^{r-1}, 
\bar{x}_{n,n}\bar{x}_{1,n}^{s-1}
\bar{x}_{1,k_{s+1}}\cdots \bar{x}_{1,k_{t}} \rangle +\\
&\quad + \alpha^{2}\langle \Ee_{n-1,1}, \bar{x}_{1,n}\rangle 
\langle \Ee_{n-1,1}^{r-1}, \bar{x}_{n,n}\bar{x}_{1,n}^{s-1}
\bar{x}_{1,k_{s+1}}\cdots 
\bar{x}_{1,k_{t}} \rangle + \cdots +\\
&\quad + \alpha^{2(s-1)}
\langle \Ee_{n-1,1}, \bar{x}_{1,n}\rangle 
\langle \Ee_{n-1,1}^{r-1}, 
\bar{x}_{n,n}\bar{x}_{1,n}^{s-1} 
\bar{x}_{1,k_{s+1}}\cdots \bar{x}_{1,k_{t}} \rangle\\
& = \Phi_{s}(\alpha^{2})\langle \Ee_{n-1,1}, \bar{x}_{1,n}\rangle
\langle \Ee_{n-1,1}^{r-1}, 
\bar{x}_{n,n}\bar{x}_{1,n}^{s-1} 
\bar{x}_{1,k_{s+1}}\cdots \bar{x}_{1,k_{t}} \rangle,
\end{align*} 
where the second equality follows by $ (a) $, the third from 
$k_{j}\neq n$ for all $s+1\leq j\leq t$ and the fourth from
the commuting relations of $\bar{x}_{i,j}$.
But 
\begin{align*}
& \langle \Ee_{n-1,1}^{r}, 
\bar{x}_{n,n}\bar{x}_{1,n}^{s} 
\bar{x}_{1,k_{s+1}}\cdots \bar{x}_{1,k_{t}} \rangle = 
\sum_{m_{j}= 1}^{n}\langle \Ee_{n-1,1}, \bar{x}_{n,m_{0}}
\bar{x}_{1,m_{1}}\cdots \bar{x}_{1,m_{t}}\rangle \cdot\\
& \qquad\quad \cdot 
\langle \Ee_{n-1,1}^{r-1}, \bar{x}_{m_{0},n}\bar{x}_{m_{1},n}
\cdots \bar{x}_{m_{s},n}\bar{x}_{m_{s+1},k_{s+1}}\cdots  
\bar{x}_{m_{t},k_{t}} \rangle\\
& \qquad = \sum_{m_{j}= 1}^{n}\langle \Ee_{n-1,1}, \bar{x}_{n,n}
\bar{x}_{1,m_{1}}\cdots \bar{x}_{1,m_{t}}\rangle \cdot\\
& \qquad\quad \cdot 
\langle \Ee_{n-1,1}^{r-1}, \bar{x}_{n,n}\bar{x}_{m_{1},n}
\cdots \bar{x}_{m_{s},n}\bar{x}_{m_{s+1},k_{s+1}}\cdots  
\bar{x}_{m_{t},k_{t}} \rangle\\
& \qquad = \sum_{m_{j}= 1}^{n} \beta \langle \Ee_{n-1,1}, 
\bar{x}_{1,m_{1}}\cdots \bar{x}_{1,m_{t}}\rangle \cdot\\
& \qquad\quad \cdot 
\langle \Ee_{n-1,1}^{r-1}, \bar{x}_{n,n}\bar{x}_{m_{1},n}
\cdots \bar{x}_{m_{s},n}\bar{x}_{m_{s+1},k_{s+1}}\cdots  
\bar{x}_{m_{t},k_{t}} \rangle\\
& \qquad =\beta \langle \Ee_{n-1,1}, 
\bar{x}_{1,n}\rangle \big(
\langle \Ee_{n-1,1}^{r-1}, \bar{x}_{n,n}^{2}\bar{x}_{1,n}^{s-1}
\bar{x}_{1,k_{s+1}}\cdots  
\bar{x}_{1,k_{t}} \rangle +\\
& \qquad\quad + \alpha^{2} 
\langle \Ee_{n-1,1}^{r-1}, \bar{x}_{n,n}^{2}\bar{x}_{1,n}^{s-1}
\bar{x}_{1,k_{s+1}}\cdots  
\bar{x}_{1,k_{t}} \rangle +\cdots +\\
& \qquad\quad + \alpha^{2(s-1)} 
\langle \Ee_{n-1,1}^{r-1}, \bar{x}_{n,n}^{2}\bar{x}_{1,n}^{s-1}
\bar{x}_{1,k_{s+1}}\cdots  
\bar{x}_{1,k_{t}} \rangle\big)\\
 & \qquad =\beta \langle \Ee_{n-1,1}, 
\bar{x}_{1,n}\rangle \Phi_{s}(\alpha^{2})
\langle \Ee_{n-1,1}^{r-1}, \bar{x}_{n,n}^{2}\bar{x}_{1,n}^{s-1}
\bar{x}_{1,k_{s+1}}\cdots  
\bar{x}_{1,k_{t}} \rangle,
\end{align*} 
where the fourth equality follows from $ (a) $ and the 
commuting relations of the elements $\bar{x}_{i,j}$.
Following in this way we get for $m \geq$ min$(r,s) + 1$ that

\begin{align*}
& \langle \Ee_{n-1,1}^{r}, 
\bar{x}_{n,n}\bar{x}_{1,n}^{s} 
\bar{x}_{1,k_{s+1}}\cdots \bar{x}_{1,k_{t}} \rangle =\\
&\beta^{m(m+1)/2} \langle \Ee_{n-1,1}, 
\bar{x}_{1,n}\rangle^{m} \prod_{j=0}^{m-1}\Phi_{s-j}(\alpha^{2})
\langle \Ee_{n-1,1}^{r-m}, \bar{x}_{n,n}^{m+1}\bar{x}_{1,n}^{s-m}
\bar{x}_{1,k_{s+1}}\cdots  
\bar{x}_{1,k_{t}} \rangle.
\end{align*}
If $r < s$, then taking $m = r-1$ we get by $(b)$, $(a)$ and
$k_{j}\neq n$ that  
\begin{align*}
& \langle \Ee_{n-1,1}^{r}, 
\bar{x}_{n,n}\bar{x}_{1,n}^{s} 
\bar{x}_{1,k_{s+1}}\cdots \bar{x}_{1,k_{t}} \rangle =\\
&\ \beta^{(r-1)r/2} \langle \Ee_{n-1,1}, 
\bar{x}_{1,n}\rangle^{r-1} \prod_{j=0}^{r-2}\Phi_{s-j}(\alpha^{2})
\langle \Ee_{n-1,1}, \bar{x}_{n,n}^{r}\bar{x}_{1,n}^{s-r+1}
\bar{x}_{1,k_{s+1}}\cdots  
\bar{x}_{1,k_{t}} \rangle\\
&\ =  \beta^{(r-1)r/2} \langle \Ee_{n-1,1}, 
\bar{x}_{1,n}\rangle^{r-1} \prod_{j=0}^{r-2}\Phi_{s-j}(\alpha^{2})
\beta^{r}\langle \Ee_{n-1,1},\bar{x}_{1,n}^{s-r+1}
\bar{x}_{1,k_{s+1}}\cdots  
\bar{x}_{1,k_{t}} \rangle\\
&\ 	 = 0.
\end{align*}
If $s<r$ then taking $m = s$ we get again 
by $(b)$, $(a)$ and
$k_{j}\neq n$ that 
\begin{align*}
& \langle \Ee_{n-1,1}^{r}, 
\bar{x}_{n,n}\bar{x}_{1,n}^{s} 
\bar{x}_{1,k_{s+1}}\cdots \bar{x}_{1,k_{t}} \rangle =\\
&\qquad \beta^{s(s+1)/2} \langle \Ee_{n-1,1}, 
\bar{x}_{1,n}\rangle^{s} \prod_{j=0}^{s-1}\Phi_{s-j}(\alpha^{2})
\langle \Ee_{n-1,1}^{r-s}, \bar{x}_{n,n}^{s}
\bar{x}_{1,k_{s+1}}\cdots  
\bar{x}_{1,k_{t}} \rangle\\
&= \beta^{s(s+1)/2} \langle \Ee_{n-1,1}, 
\bar{x}_{1,n}\rangle^{s} \prod_{j=0}^{s-1}\Phi_{s-j}(\alpha^{2})
\sum_{m_{j}=1}^{n}
\langle \Ee_{n-1,1}, \bar{x}_{n,n}^{s}
\bar{x}_{1,m_{s+1}}\cdots  
\bar{x}_{1,m_{t}} \rangle \cdot\\
&\quad  \cdot \langle \Ee_{n-1,1}^{r-s-1}, \bar{x}_{n,n}^{s}
\bar{x}_{m_{s+1},k_{s+1}}\cdots  
\bar{x}_{m_{t},k_{t}} \rangle \\
&= \beta^{s(s+1)/2} \langle \Ee_{n-1,1}, 
\bar{x}_{1,n}\rangle^{s} \prod_{j=0}^{s-1}\Phi_{s-j}(\alpha^{2})
\sum_{m_{j}=1}^{n}\beta^{s}
\langle \Ee_{n-1,1}, 
\bar{x}_{1,m_{s+1}}\cdots  
\bar{x}_{1,m_{t}} \rangle \cdot\\
&\quad  \cdot 
\langle \Ee_{n-1,1}^{r-s-1}, \bar{x}_{n,n}^{s}
\bar{x}_{m_{s+1},k_{s+1}}\cdots  
\bar{x}_{m_{t},k_{t}} \rangle = 0.
\end{align*}
Finally, if $r=s$, taking $m = r-1$ we get
by $(b)$, $(a)$ and
$k_{j}\neq n$ that
\begin{align*}
& \langle \Ee_{n-1,1}^{r}, 
\bar{x}_{n,n}\bar{x}_{1,n}^{s} 
\bar{x}_{1,k_{s+1}}\cdots \bar{x}_{1,k_{t}} \rangle =\\
&\qquad \beta^{(r-1)r/2} \langle \Ee_{n-1,1}, 
\bar{x}_{1,n}\rangle^{r-1} \prod_{j=0}^{r-2}\Phi_{r-j}(\alpha^{2})
\langle \Ee_{n-1,1}, \bar{x}_{n,n}^{r}\bar{x}_{1,n}
\bar{x}_{1,k_{s+1}}\cdots  
\bar{x}_{1,k_{t}} \rangle\\
&\qquad  = \beta^{(r-1)r/2} \langle \Ee_{n-1,1}, 
\bar{x}_{1,n}\rangle^{r-1} \prod_{j=0}^{r-2}\Phi_{r-j}(\alpha^{2})
\beta^{r}\langle \Ee_{n-1,1}, \bar{x}_{1,n}
\bar{x}_{1,k_{s+1}}\cdots  
\bar{x}_{1,k_{t}} \rangle\\
&\qquad  = \beta^{(r+1)r/2} \langle \Ee_{n-1,1}, 
\bar{x}_{1,n}\rangle^{r} \prod_{j=0}^{r-2}\Phi_{r-j}(\alpha^{2})
\eps(
\bar{x}_{1,k_{s+1}}\cdots  
\bar{x}_{1,k_{t}}).
\end{align*}
Hence, it follows that 
\begin{align*}
& \langle \Ee_{n-1,1}^{r}, 
\bar{x}_{1,n}^{s}\bar{x}_{1,k_{s+1}}
\cdots \bar{x}_{1,k_{t}} \rangle = \\
& \qquad \Phi_{s}(\alpha^{2})
\langle \Ee_{n-1,1}, \bar{x}_{1,n}\rangle
\langle \Ee_{n-1,1}^{r-1}, 
\bar{x}_{n,n}\bar{x}_{1,n}^{s-1} 
\bar{x}_{1,k_{s+1}}\cdots \bar{x}_{1,k_{t}} \rangle\\
& \qquad = \delta_{r,s}
\beta^{r(r-1)/2} \langle \Ee_{n-1,1}, 
\bar{x}_{1,n}\rangle^{r} \prod_{j=0}^{r-1}\Phi_{r-j}(\alpha^{2})
\eps(
\bar{x}_{1,k_{s+1}}\cdots  
\bar{x}_{1,k_{t}}).
\end{align*} 
Finally, a direct computation shows that 
$$ \langle \Ee_{n-1,1}^{r}, 
\bar{x}_{1,n}^{r}\rangle = \beta^{r(r-1)/2}
\langle \Ee_{n-1,1}, 
\bar{x}_{1,n}\rangle^{r}
\prod_{j=0}^{r-1}\Phi_{r-j}(\alpha^{2}) 
.$$

$ (ii) $ Clearly, it is enough to prove it on the generators.
By Lemma \ref{lema:pair-e-f} $ (i) $ we have 
for all $1< i\leq j \leq n$ and $1\leq k\leq n-1 $ that
$\langle \Ee_{k,1}, \bar{x}_{i,j}\rangle = 0 = \eps(\Ee_{k,1})\eps(\bar{x}_{i,j})$.
The proof of the other equality is completely analogous. 

\par $ (iii) $ We prove it by induction on $n$. By $ (ii) $ 
and the paragraph before the lemma we have
\begin{align*}
\langle \Ee^{M}, \bar{x}^{N} \rangle & =
 \langle \Ee^{M_{n}}\cdots \Ee^{M_{1}}, 
 \bar{x}^{N_{1}}\cdots \bar{x}^{N_{n}} \rangle\\
 &= \langle (\Ee^{M_{n}}\cdots \Ee^{M_{1}})_{(1)}, 
 \bar{x}^{N_{1}}\rangle \langle (\Ee^{M_{n}}\cdots \Ee^{M_{1}})_{(2)},
  \bar{x}^{N_{2}} \cdots \bar{x}^{N_{n}} \rangle\\
& =\langle (\Ee^{M_{n}})_{(1)}\cdots (\Ee^{M_{1}})_{(1)}, 
 \bar{x}^{N_{1}}\rangle \langle (\Ee^{M_{n}})_{(2)}
 \cdots (\Ee^{M_{1}})_{(2)},
  \bar{x}^{N_{2}} \cdots \bar{x}^{N_{n}} \rangle\\
&= \langle (\Ee^{M_{n}})_{(1)}\cdots (\Ee^{M_{1}})_{(1)}, 
 \bar{x}^{N_{1}}\rangle \cdot \\
 &\qquad \cdot
 \langle (\Ee^{M_{n}})_{(2)}\cdots (\Ee^{M_{2}})_{(2)},
  (\bar{x}^{N_{2}} \cdots \bar{x}^{N_{n}})_{(1)} \rangle
  \cdot \\
 &\qquad \cdot     
  \langle (\Ee^{M_{1}})_{(2)},
 ( \bar{x}^{N_{2}} \cdots \bar{x}^{N_{n}})_{(2)} \rangle\\
&= \langle (\Ee^{M_{n}})_{(1)}\cdots (\Ee^{M_{1}})_{(1)}, 
 \bar{x}^{N_{1}}\rangle \langle (\Ee^{M_{n}})_{(2)}
 \cdots (\Ee^{M_{2}})_{(2)},
  \bar{x}^{N_{2}} \cdots \bar{x}^{N_{n}} \rangle \cdot\\
&\qquad \cdot  
  \eps( (\Ee^{M_{1}})_{(2)})\\
&= \langle (\Ee^{M_{n}})_{(1)}\cdots (\Ee^{M_{2}})_{(1)}, 
 (\bar{x}^{N_{1}})_{(1)}\rangle 
 \langle (\Ee^{M_{1}})_{(1)}, 
 (\bar{x}^{N_{1}})_{(2)}\rangle\cdot \\
&\qquad \cdot  
 \langle (\Ee^{M_{n}})_{(2)}\cdots (\Ee^{M_{2}})_{(2)},
  \bar{x}^{N_{2}} \cdots \bar{x}^{N_{n}} \rangle 
  \eps( (\Ee^{M_{1}})_{(2)})\\  
&= \eps ((\Ee^{M_{n}})_{(1)}\cdots (\Ee^{M_{2}})_{(1)}) 
\eps( (\bar{x}^{N_{1}})_{(1)}) 
 \langle \Ee^{M_{1}}, 
 (\bar{x}^{N_{1}})_{(2)}\rangle\cdot \\
&\qquad \cdot  
 \langle (\Ee^{M_{n}})_{(2)}\cdots (\Ee^{M_{2}})_{(2)},
  \bar{x}^{N_{2}} \cdots \bar{x}^{N_{n}} \rangle\\  
&=  \langle \Ee^{M_{1}}, 
 \bar{x}^{N_{1}}\rangle
 \langle \Ee^{M_{n}}\cdots \Ee^{M_{2}},
  \bar{x}^{N_{2}} \cdots \bar{x}^{N_{n}} \rangle.
\end{align*}
Thus, by induction we need only to evaluate the first factor
\begin{equation}\label{eq:first-factor}
\langle \Ee^{M_{1}}, 
 \bar{x}^{N_{1}}\rangle = 
 \langle \Ee_{n-1,1}^{M_{1,n}}\cdots 
 \Ee_{1,1}^{M_{1,2}}, 
 \bar{x}_{11}^{N_{11}}\cdots  \bar{x}_{1n}^{N_{1n}}\rangle.
\end{equation}
But
\begin{align*}
&\langle \Ee_{n-1,1}^{M_{1,n}}\cdots 
 \Ee_{1,1}^{M_{1,2}}, 
 \bar{x}_{11}^{N_{11}}\cdots  \bar{x}_{1n}^{N_{1n}}\rangle  =
\\
& \qquad \alpha^{\sum_{r<s}N_{1r}N_{1s}}
 \langle \Ee_{n-1,1}^{M_{1,n}}\cdots 
 \Ee_{1,1}^{M_{1,2}}, 
 \bar{x}_{1n}^{N_{1n}}\cdots  \bar{x}_{11}^{N_{11}}\rangle\\
\end{align*}
\begin{align*}
& \qquad = \alpha^{\sum_{r<s}N_{1r}N_{1s}} \langle \Ee_{n-1,1}^{M_{1,n}}, 
 (\bar{x}_{1n}^{N_{1n}})_{(1)}\cdots  (\bar{x}_{11}^{N_{11}})_{(1)}\rangle
 \cdot\\
 & \qquad\quad \cdot
 \langle \Ee_{n-2,1}^{M_{1,n-1}} \cdots 
 \Ee_{1,1}^{M_{1,2}}, 
 (\bar{x}_{1n}^{N_{1n}})_{(2)}\cdots  (\bar{x}_{11}^{N_{11}})_{(2)}\rangle.
\end{align*}
Since $\langle \Ee_{k,1}, \bar{x}_{ln}\rangle = 0 = 
\eps(\Ee_{k,1})\eps(\bar{x}_{ln})$ for all
$k<n-1$  and $1\leq l\leq n$ we obtain by $(i)$ that 
\begin{align*}
&\langle \Ee_{n-1,1}^{M_{1,n}}\cdots 
 \Ee_{1,1}^{M_{1,2}}, 
 \bar{x}_{11}^{N_{11}}\cdots  
 \bar{x}_{1n}^{N_{1n}}\rangle  =\\
 & \quad = \alpha^{\sum_{r<s}N_{1r}N_{1s}} 
\langle \Ee_{n-1,1}^{M_{1,n}}, 
 (\bar{x}_{1n}^{N_{1n}}))_{(1)}\cdots  
(\bar{x}_{11}^{N_{11}})_{(1)}\rangle
 \eps ((\bar{x}_{1n}^{N_{1n}})_{(2)})\cdot\\
 & \quad\quad \cdot
 \langle \Ee_{n-2,1}^{M_{1,n-1}} \cdots 
 \Ee_{1,1}^{M_{1,2}}, 
 (\bar{x}_{1n-1}^{N_{1n-1}})_{(2)}\cdots  
 (\bar{x}_{11}^{N_{11}})_{(2)}\rangle\\
 & \quad = \alpha^{\sum_{r<s}N_{1r}N_{1s}} 
\langle \Ee_{n-1,1}^{M_{1,n}}, 
 \bar{x}_{1n}^{N_{1n}}(\bar{x}_{1n-1}^{N_{1n-1}})_{(1)} 
 \cdots  (\bar{x}_{11}^{N_{11}})_{(1)}\rangle
 \cdot\\
 & \quad\quad \cdot
 \langle \Ee_{n-2,1}^{M_{1,n-1}} \cdots 
 \Ee_{1,1}^{M_{1,2}}, 
 (\bar{x}_{1n-1}^{N_{1n-1}})_{(2)}	\cdots  
 (\bar{x}_{11}^{N_{11}})_{(2)}\rangle\\
 & \quad = \alpha^{\sum_{r<s}N_{1r}N_{1s}} 
\langle \Ee_{n-1,1}^{M_{1,n}}, 
 \bar{x}_{1n}^{N_{1n}}\rangle
 \langle \Ee_{n-2,1}^{M_{1,n-1}} \cdots 
 \Ee_{1,1}^{M_{1,2}}, 
 \bar{x}_{1n-1}^{N_{1n-1}}	\cdots  
 \bar{x}_{11}^{N_{11}}\rangle.
\end{align*}
Then the claim follows by the induction hypothesis and 
item $(i)$.
\epf

In view of Lemma \ref{lema:e-M-x-N} $(iii)$, the proof
of the following lemma is completely analogous to
\cite[Thm. 4.3]{Tk}.

\begin{lema}\label{lema:overline-psi-inj}
$\overline{\psi}_{+}$ and 
$\overline{\psi}_{-}$ are injective.
\end{lema}

\pf We prove it only for 
$\overline{\psi}_{+}: K_{+} \to \hat{{\bf u}}_{\alpha, \beta}
(\lieb^{+})^{*}$ since
the proof for 
$\overline{\psi}_{-}$ is analogous.
Let $\sum c_{N}\bar{x}^{N}$ be in the kernel of  
$\overline{\psi}_{+}$. By multiplying by a suitable power
of $\bar{x}_{11}\cdots \bar{x}_{nn}$, we may assume that  $N_{ij} \in \NN_{0}$
for all $1\leq i\leq j\leq n$ with $c_{N}\neq 0$. Now choose an
element $M$ among $N$ such that $c_{M}\neq 0$ and 
$M_{11} + M_{12} + \cdots + M_{nn}$ has the largest value. Then
by Lemma \ref{lema:e-M-x-N} $(iii)$, 
$$0 = \overline{\psi}_{+}\big(\sum c_{N}\bar{x}^{N}\big) (\Ee^{M}) = 
\langle \Ee^{M}, \sum c_{N}\bar{x}^{N}\rangle = c_{M},$$ 
a contradiction. Thus $\overline{\psi}_{+}$ in injective.
\epf 

\noindent \emph{Proof of Prop. \ref{prop:psi-injective}}: 
By the definition of the maps, 
we have the diagram
$$ \xymatrix{
  \overline{H} \ar[r]^(.35){ \overline{\delta}} \ar[d]_{\overline{\psi}} &
K_{+} \ot K_{-} \ar[d]^{\overline{\psi}_{+}\ot \overline{\psi}_{-}}\\
  \hat{{\bf u}}_{\alpha,
\beta}(\liegl_{n})^{*} \ar[r]^(.35){^{t}m} & \hat{{\bf u}}_{\alpha,
\beta}(\lieb^{+})^{*}\ot \hat{{\bf u}}_{\alpha, \beta}(\lieb^{-})^{*},}$$
which is commutative. By Lemma \ref{lema:overline-psi-inj}, 
$\overline{\psi}_{+}$ and $\overline{\psi}_{-}$ 
are injective, then 
$\overline{\psi}_{+}\ot \overline{\psi}_{-}$ is injective.
Since $^{t}m$ and $\overline{\delta}$ are injective 
by \eqref{eq:transp-mult-u-hat}
and \eqref{diag:bar-delta}, it follows that 
$\overline{\psi}$ is also injective. Since
both Hopf algebras have the same dimension by 
Prop. \ref{prop:B-central} $(iii)$ and
Lemma \ref{lema:u-hat-pointed}, it follows that 
$\overline{\psi}$ is an isomorphism.
\qed

\smallbreak 
Using Cor. \ref{cor:PI-hopf-triple-example} we get the 
following corollary.

\begin{cor}\label{cor:PI-hopf-triple-example2}
$(\Oc(GL_{n}),\Oc_{\alpha,\beta}(GL_{n}), \hat{{\bf u}}_{\alpha,
\beta}(\liegl_{n})^{*})$ is a PI-Hopf triple
and one has the central extension of Hopf algebras
\begin{equation}\label{suc:principal}
  1\to \Oc(GL_{n}) \xrightarrow{\iota}
  \Oc_{\alpha,\beta}(GL_{n}) \xrightarrow{\pi}
  \hat{{\bf u}}_{\alpha,
\beta}(\liegl_{n})^{*} \to 1.\qed
\end{equation}
\end{cor}

\subsection{Construction of the quotients}\label{subsec:a-b-root-constr}
In this subsection we perform the construction of the quotients
of $ \Oc_{\alpha,\beta}(GL_{n})$, as in \cite{AG}. 
By Corollary \ref{cor:PI-hopf-triple-example2},
$\Oc_{\alpha,\beta}(GL_{n})$
fits into a central exact sequence 
$$
 1\to \Oc(GL_{n}) \xrightarrow{\iota}
  \Oc_{\alpha,\beta}(GL_{n}) \xrightarrow{\pi}
   \hat{{\bf u}}_{\alpha,\beta}(\liegl_{n})^{*} \to 1.
$$
Let $\varphi:  \Oc_{\alpha,\beta}(GL_{n}) \to A$ 
be a surjective Hopf algebra
map.  The Hopf subalgebra $K= \varphi(\Oc(GL_{n}))$ 
is central in $A$ and
whence $A$ is an $H$-extension of $K$, 
where $H$ is the Hopf algebra
$H = A/ AK^{+}$. Indeed, by Lemma \ref{lema:triple-affine} $(i)$, 
 $\Oc_{\alpha,\beta}(GL_{n})$ is noetherian and 
whence $A$ is also noetherian. 
Then by \cite[Thm.
3.3]{Sch1}, $A$ is faithfully flat over every 
central Hopf subalgebra; in particular
over $K$, and the claim follows directly from \cite[Prop.
3.4.3]{Mo}. Since $K$ is a quotient of $\Oc(GL_{n})$, there is an
algebraic group $\Ga$ and an injective map of algebraic groups
$\sigma: \Ga \to GL_{n}$ such that $K \simeq \Oc(\Ga)$. Moreover, since
$\varphi(\Oc_{\alpha,\beta}(GL_{n})\Oc(GL_{n})^{+}) = AK^{+}$, 
we have $\Oea\Oc(GL_{n})^{+} \subseteq
\Ker \hat{\pi}\varphi$, where $\hat{\pi}: A \to H$ is the canonical
projection. Since  
$\qabhat^{*} \simeq \Oea /[\Oea\Oc(G)^{+}]$
by Prop. \ref{prop:psi-injective}, there
is a surjective map $r:\qabhat^{*} \to H$ 
and the following diagram
\begin{equation}\label{diag:propqsubgr}
 \xymatrix{1
\ar[r]^{} & \Oc(GL_{n}) \ar[r]^{\iota} \ar[d]_{^{t}\sigma}  & \Oea
\ar[r]^{\pi}
\ar[d]^{\varphi} & \hat{{\bf u}}_{\alpha,\beta}(\liegl_{n})^{*} 
\ar[r]^{}\ar[d]^{r} & 1\\
1 \ar[r]^{} &   \Oc(\Ga) \ar[r]^{\hat{\iota}} & A
\ar[r]^{\hat{\pi}} & H \ar[r]^{} & 1,}
\end{equation}
is commutative. Therefore, every quotient of 
$\Oea$ fits into a similar diagram
and can be constructed using central extensions. 
As in \cite{AG} we construct the quotients in three steps. 

\subsubsection{First step}\label{subsub:first-step}
The surjective map 
$r: \qabhat^{*} \to H$ induces an injective
Hopf algebra map $^{t}r: H^{*} \to \qabhat$. 
Denote also by $H^{*}$ the Hopf subalgebra of $\qabhat$ 
given by the image of $^{t}r$. 
Since $\qabhat$ is finite-dimensional and pointed by
Lemma \ref{lema:u-hat-pointed}, 
by \cite[Cor. 1.12]{AG} -- see also \cite{CM, MuI} -- 
the Hopf subalgebras of
$\qabhat$ are parameterized by triples $(\Sigma, I_{+},I_{-})$ 
where $\Sigma$ is
a subgroup of $\Tt:=$ $ G(\qabhat)$ $  \simeq (\ZZ / \ell\ZZ)^{n}$,
$$
I_{+} = \{i\vert\ e_{i} \in H^{*},\  1\leq i< n\}\mbox{ and }
I_{-} = \{i\vert\ f_{i} \in H^{*},\  1\leq i< n\},
$$
such that $w_{i} = 
h_{i}^{n_{\alpha}}h_{i+1}^{n_{\beta}} \in \Sigma$ 
if $i \in I_{+}$ and $w_{j}' = 
h_{j+1}^{n_{\alpha}}h_{j}^{n_{\beta}} \in \Sigma$ if $j \in I_{-}$.

\bigbreak
\subsubsection*{The Hopf subalgebras $\QEabl$ of $\QEabg$ and
$\qablhat$ of $\qabhat$.}

\begin{definition}\label{def:gammaL}
For every triple $(\Sigma, I_+, I_-)$ define $\QEabl$ to be the
subalgebra of $\QEabg$ generated by the elements
$$
\{a_{k}, a_{k}^{-1}, b_{k}, b_{k}^{-1}, e_{i}, f_{j}\vert\ 1\leq k \leq n,
i\in I_{+},\ j\in I_{-} \}.$$
\end{definition}
Note that
$\QEabl$ does not depend on $\Sigma$. With this definition, 
the following proposition
is clear.

\begin{prop}\label{prop:QEabl-hopf-subalgebra}
$\QEabl$ is a Hopf subalgebra of $\QEabg$.\qed
\end{prop}

Let $I_{n}$ be the Hopf ideal of $\QEabg$ given by Thm. 
\ref{thm:elem-centrales-U} $(ii)$ and define $J_{n} = I_{n}\cap \QEabl$.
Clearly, $J_{n}$ is a Hopf ideal of $\QEabl$ and the quotient defines
the Hopf algebra $\qabl = \QEabl / J_{n}$ such that
the diagram 
$$
 \xymatrix{
\QEabl \ar@{^{(}->}[r]^{} \ar@{->>}[d]_{}  & \QEabg
\ar@{->>}[d]^{} \\ 
\qabl \ar@{^{(}->}[r]^{j} & \qab,}$$
commutes, where $j: \qabl \to \qab$ 
is given by $j(a + J_{n}) = a + I_{n}$ for 
all $a \in \QEabl$.
Clearly, $j$ is a well-defined Hopf algebra map 
that makes the diagram commute.
Moreover, it is injective since $j(a + J_{n}) = 0$ 
if and only if $a \in I_{n} \cap \QEabl = J_{n}$.

\par  
Recall that $\qabhat$ is the Hopf algebra given by
$\qab / \mathcal{I}_{\ell}$, where
$\mathcal{I}_{\ell}$ is the ideal 
generated by the central group-likes 
$\{h_{i}^{n_{\alpha}}a_{i}^{-1}-1, h_{i}^{n_{\beta}}
b_{i}^{-1}-1\vert\ 1\leq i\leq n\}$. Then,
$\mathcal{J}_{\ell} := \mathcal{I}_{\ell} \cap \qabl$ 
is a Hopf ideal of $\qabl$. Thus
we define
\begin{equation}\label{eq:def-qablhat}
\qablhat:= \qabl / \mathcal{J}_{\ell},
\end{equation}
to be the finite-dimensional 
Hopf algebra given by the quotient. 

\par Recall that all Hopf subalgebras of $\qabhat$ are determined by 
triples $(\Sigma,J_{+},J_{-})$ where $\Sigma \subseteq \mathbb{T}$
is a subgroup. Since $\QEabl$ is determined by  
$I_{+}$ and  $I_{-}$, then clearly $\qablhat$ is the Hopf subalgebra 
of $\qabhat$ that corresponds to the triple $(\mathbb{T},I_{+},I_{-})$.

\subsubsection*{The quantized coordinate algebra $\Oeal$}
In this subsection we construct the quantum groups $\Oeal$
associated to the triple $(\Tt, I_{+}, I_{-})$. Since the pairing
between $\QEabg$ and $\Oea$ is degenerate, we can not follow
directly the construction made in \cite[2.1.3]{AG}. The advantage here
is that the Hopf algebra $\Oea$ is given by generators and relations. 
This allows us to give an explicit 
construction of $\Oeal$. This construction can be already seen in the work of
M\"uller \cite{Mu, MuI} but without mention of the triple. 
For every triple $(\Tt, I_{+}, I_{-})$ define the sets
\begin{equation}\label{eq:def-de-los-I}
\II_{+} = \left\{(i,j) \vert\ i\leq k < j,\ k\notin I_{+}\right \},
\quad
\II_{-} = \left\{(i,j) \vert\ j\leq k < i,\ k \notin I_{-} \right \},
\end{equation}
and let $\II$ be the two-sided ideal of $\Oea$ generated
by the elements $\left\{x_{i,j} \vert\ (i,j)\in \II_{+}\cup \II_{-}  \right \}$.
Then $\II$ is a Hopf ideal and one has the central sequence of
Hopf algebras
\begin{equation}\label{eq:sucex-J}
\xymatrix{ 1 \ar[r]^{} & \Oc(GL_{n}) / \JJ  
\ar@{^{(}->}[r]^{}   & \Oea /  \II 
\ar@{->>}[r]^{} & \overline{H} / \pi(\II) \ar[r]^{} & 1,} 
\end{equation}
where $\JJ = \II \cap \Oc(GL_{n})$ and $\pi(\II)$ are
Hopf ideals of $\Oc(GL_{n})$ and $\overline{H}\simeq \qablhat^{*}$, 
respectively.

\begin{lema}\label{lema:qablhat}
\begin{enumerate}
 \item[$(a)$] $\JJ$ is the ideal of $\Oc(GL_{n})$ generated
by the elements $\{x_{i,j}^{\ell} \vert\ (i,j)\in \II_{+}\cup \II_{-}\}$.

\item[$(b)$] $\overline{H} / \pi(\II) \simeq \qablhat^{*}$.

\item[$(c)$] The sequence \eqref{eq:sucex-J} is exact and fits into the
commutative diagram 
\begin{equation}\label{diag:sucex-J}
\xymatrix{ 1 \ar[r]^{} & \Oc(GL_{n})  
\ar@{^{(}->}[r]^{} \ar@{->>}[d]^{}  & \Oea  
\ar@{->>}[r]^{\pi}\ar@{->>}[d]^{\Res} & \qabhat^{*} 
\ar[r]^{}\ar@{->>}[d]^{\res}  & 1\\
1 \ar[r]^{} & \Oc(GL_{n}) / \JJ  
\ar@{^{(}->}[r]^{}   & \Oea /  \II 
\ar@{->>}[r]^(.55){\pi_{L}} & \qablhat^{*} \ar[r]^{} & 1,} 
\end{equation}

\item[$(d)$] The Hopf algebra surjection $\pi_{L}$ 
admits a coalgebra section 
$\gamma_{L}$.
\end{enumerate}
\end{lema}

\pf $(a)$ By Prop. \ref{prop:B-central}, we identify
$\Oc(GL_{n})$ with its image under $F^{\#}$; that is, it is 
the Hopf subalgebra generated by the elements 
$\{x_{ij}^{\ell}\vert\ 1\leq i,j\leq n\}$.
Then $\JJ = \II \cap \Oc(GL_{n})$ is contained in the image of
$F^{\#}$, which is an algebra map 
and whence $\JJ$ is generated
by the elements $\{x_{ij}^{\ell} \vert\ (i,j)\in \II_{+}\cup \II_{-}\}$.

$(b)$ By definition, $\pi(\II) $ is the Hopf ideal of $\overline{H}$
generated by the elements 
$\{\bar{x}_{i,j} \vert\ (i,j)\in \II_{+}\cup \II_{-}\}$. Thus
the Hopf algebra $\overline{H}/\pi(\II)$ has a basis given by 
\begin{equation}\label{eq:base-de-ul}
\Big\{\prod_{i, j}^{} \bar{x}_{ij}^{e_{ij}} 
\vert\ (i,j)\notin \II_{+}\cup \II_{-}, 0\leq e_{ij} < \ell \Big\},
\end{equation}
see \cite[Prop. 3.5]{DPW}. In particular, 
$\dim \overline{H}/\pi(\II) = \ell^{n^{2} - |\II_{+}\cup \II_{-}|}$.
Denote by $\res: \qabhat^{*} \to \qablhat^{*}$ the surjective
map induced by the inclusion $j : \qablhat \to \qabhat$, 
which is simply the restriction. Then by Prop. \ref{prop:psi-injective},
we have 
for all $x_{i,j}$ such that $(i,j) \in \II_{+}\cup \II_{-}$ 
that $\res(\bar{x}_{ij}) (h) = \bar{x}_{ij} (j(h)) = 0$ for all
$h \in \qablhat$, since $\qablhat$ is determined 
by the triple $(\Tt, I_{+}, I_{-})$.
Thus there exists a surjective Hopf algebra map 
$p:\overline{H}/\pi(\II) \twoheadrightarrow
\qablhat^{*}$. Since by Thm. \ref{thm:PBW-basis-gl},
$\dim \qablhat = \ell^{n^{2} - |\II_{+}\cup \II_{-}|}$, 
it follows that $p$ is an isomorphism.

\smallbreak $(c)$ Since $\Oc(GL_{n}) / \JJ$ is central in $\Oea /\II$ 
and $\Oea /\II$ is noetherian, by \cite[Thm. 3.3]{Sch1} 
$\Oea/\II$ is faithfully flat over $\Oc(GL_{n}) / \JJ$ 
and by \cite[Prop. 3.4.3]{Mo}
we have the central exact sequence
$$ 
\xymatrix{ 
1 \ar[r]^{} & \Oc(GL_{n}) / \JJ  
\ar@{^{(}->}[r]^{}   & \Oea /  \II 
\ar@{->>}[r]^{} & K \ar[r]^{} & 1,} $$
where $K$ is the Hopf algebra given by 
the quotient 
$$K  = (\Oea /\II) /[(\Oc(GL_{n}) / \JJ)^{+}(\Oea / \II)].$$ 
But
since $(\Oc(GL_{n}) / \JJ)^{+}(\Oea/\II) = (\Oc(GL_{n})^{+}\Oea) /\II$, it 
follows that $ K \simeq \overline{H} /\pi(\II) \simeq \qablhat^{*}$.
Thus we have the central exact sequence
$$ 
\xymatrix{ 
1 \ar[r]^{} & \Oc(GL_{n}) / \JJ  
\ar@{^{(}->}[r]^{}   & \Oea /  \II 
\ar@{->>}[r]^(.55){\pi_{L}} & \qablhat^{*} \ar[r]^{} & 1.} $$
To see that the above exact sequence fits into the commutative
diagram, we need only to show that $\res\pi = \pi_{L}\Res$. But for all
$x \in \Oea$ we have $\pi_{L}\Res(x) = \pi_{L} (x + \II) = 
\pi(x) + \pi(\II) = \res(\pi(x))$.

\smallbreak $(d)$ Using that the set in \eqref{eq:base-de-ul} is 
a basis of $\overline{H} / \pi(\II)$, 
we define as in Corollary \ref{cor:def-gamma}
the linear map $\tilde{\gamma}_{L}: \overline{H} / \pi(\II) \to \Oeal$ by
$\tilde{\gamma}_{L}\big(\prod_{i, j}^{} \bar{x}_{ij}^{e_{ij}}\big)
= \prod_{i, j}^{} x_{ij}^{e_{ij}}$. 
Clearly, $\tilde{\gamma}_{L}$ is a linear section of $\pi_{L}$. Again, 
a direct calculation shows that $\tilde{\gamma}_{L}$ 
is also a coalgebra map. Finally, since $\overline{H}/\pi(\II) \simeq \qablhat^{*}$,
the morphism $\gamma_{L}$ given by the composition of 
$\tilde{\gamma}_{L}$ and this isomorphism gives the 
desired coalgebra section. 
\epf

\begin{obs}\label{rmk:definicion-L}
Since $\Oc(GL_{n}) / \JJ$ is a commutative Hopf algebra which is
a quotient of $\Oc(GL_{n})$, there exists an algebraic
subgroup $L$ of $GL_{n}$ such that $\Oc(GL_{n}) / \JJ \simeq \Oc(L)$. 
Moreover,
by the lemma above, we know that this subgroup is given by the set
of complex matices $M= (m_{ij})$ of $GL_{n}$ such that $m_{ij} = 0$ if 
$(i,j) \in \II_{+}\cup \II_{-}$.
\end{obs}

We set then 
\begin{equation}\label{eq:def-O-L}
\Oeal := \Oea /  \II. 
\end{equation}
Consequently, we can re-write the commutative diagram 
\eqref{diag:sucex-J} as
\begin{equation}\label{diag:sucex-J-2}
\xymatrix{ 1 \ar[r]^{} & \Oc(GL_{n})  
\ar@{^{(}->}[r]^{} \ar@{->>}[d]^{}  & \Oea  
\ar@{->>}[r]^{\pi}\ar@{->>}[d]^{\Res} & \qabhat^{*} 
\ar[r]^{}\ar@{->>}[d]^{\res}  & 1\\
1 \ar[r]^{} & \Oc(L)  
\ar@{^{(}->}[r]^{}   & \Oeal
\ar@{->>}[r]^{\pi_{L}} & \qablhat^{*} \ar[r]^{} & 1.} 
\end{equation}

\subsubsection{Second step}
In this subsubsection we apply 
a previous result of \cite{AG} on quantum subgroups given by a
pushout construction, to perform the second step of the construction.

\par Fix a triple $(\Tt, I_{+}, I_{-})$ and let $L$
be the algebraic group 
associated to this triple as above, see Remark \ref{rmk:definicion-L}. 
Let $\Ga$ be an algebraic group and $\sigma: \Ga \to
GL_{n}$ be an injective homomorphism of algebraic 
groups such that $\sigma
(\Ga) \subseteq L$. Then we have a surjective Hopf algebra map $\
^{t}\sigma: \OO(L) \to \OO(\Ga)$. Applying the pushout construction
given by Prop. \ref{prop:cociente1}, we obtain a Hopf algebra 
$A_{\lgot, \sigma}$ which is part of an exact sequence of Hopf algebras
and fits into the following commutative diagram
\begin{equation}\label{diag:pushout-ol}
\xymatrix{1 \ar[r]^{} & \Oc(G) \ar[r]^{\iota} \ar[d]_{} & \Oea
\ar[r]^{\pi}
\ar[d]^{\Res} & \qabhat^{*} \ar[r]^{}\ar[d]^{\res} & 1\\
1 \ar[r]^{} &   \Oc(L)\ar[d]_{^{t}\sigma}  \ar[r]^{\iota_L} & \Oeal
\ar[r]^{\pi_L}\ar[d]^{\nu} &
\qablhat^* \ar[r]^{}\ar@{=}[d] & 1\\
1 \ar[r]^{} &   \OO(\Ga)  \ar[r]^{j} & A_{\lgot,
\sigma}\ar[r]^{\bar{\pi}} & \qablhat^* \ar[r]^{} & 1.}
\end{equation}

In particular, if $\Ga$ is finite then $\dim A_{\lgot,\sigma}= \vert
\Ga\vert \dim \qablhat = \vert \Ga\vert \ell^{n^{2}-|\II_{+}\cup \II_{-}|}$,
see \cite[Rmk. 2.11]{AG}.

\subsubsection{Third step}
In this subsection we make the third and last step of the
construction. As in \cite{AG}, it consists essentially on taking 
a quotient by a Hopf
ideal generated by differences of central group-like elements of
$A_{\lgot, \sigma}$. 

\par Recall that from the beginning of this section we fixed
a surjective Hopf algebra map $r: \qablhat^{*} \to H$ and $H^{*}$ is
determined by the triple $(\Sigma, I_{+}, I_{-})$. Since the Hopf
subalgebra $\qablhat$ is determined by the triple $(\mathbb{T},I_{+},
I_{-})$ with $\mathbb{T} \supseteq \Sigma$, we have that $H^{*}
\subseteq \qablhat \subseteq \qabhat$. Denote by $v : \qablhat^{*} \to H$ the
surjective Hopf algebra map induced by this inclusion. 

\begin{obs}\label{obs:hat-t} 
By definition, $\mathbb{T}$ is the group generated by
the group-like elements $h_{i},\ 1\leq i\leq n$. Let 
$\delta_{j},\ 1\leq j\leq n$ denote the characters on $\mathbb{T}$
given by $\delta_{j}(h_{i}) = q^{\delta_{i,j}}$, where 
$q= \alpha^{-1}\beta$ is
a primitive $ \ell$-th root of unity which is fixed. 
By Lemma \ref{lema:hopf-pairing-restricted} and Prop.
\ref{prop:psi-injective} we know that
$$ \langle \bar{x}_{jj}, h_{i} \rangle =   
\langle \bar{x}_{jj}, a_{i}^{-1}b_{i} \rangle = 
\langle \bar{x}_{jj}, a_{i}^{-1}
\rangle \langle\bar{x}_{jj}, b_{i} \rangle =
\alpha^{-\delta_{i,j}}\beta^{\delta_{i,j}} = q^{\delta_{i,j}}.$$
Thus, if we restrict 
$\langle \bar{x}_{jj}, - \rangle$ to
$\mathbb{T}$, we may identify $\bar{x}_{jj} = \delta_{j}$ for all
$1\leq i\leq n$.  
\end{obs}

\par Let $\mathbb{T}_{I}$ be the subgroup of
$\mathbb{T}$ generated by the elements 
$\{w_{i}, w'_{j}:\ i\in I_{+},\ j\in I_{-}\}$
and denote by $\rho: \widehat{\mathbb{T}} \to \widehat{\Sigma}$,
$\rho_{I}: \widehat{\mathbb{T}} \to \widehat{\mathbb{T}_{I}}$
the group homomorphisms between the character groups induced
by the inclusions. Set $N = \Ker \rho$ and $M_{I} = \Ker \rho_{I}$.

\par The following lemma is analogous to \cite[Lemma 2.14]{AG}.

\begin{lema}\label{lema:central-grouplikes-1}
\begin{enumerate}
\item[$(a)$] Every $ \chi $ of $M_{I}$ defines an element $\bar{\chi} $ 
of $ G(\qablhat^{*}) $ which is central. In particular, 
since $N \subseteq M_{I}$, $\bar{\chi} \in 
G(\qablhat^{*}) \cap \Z(\qablhat^{*})$ for all $\chi  \in N$. 
\item[$(b)$] $H \simeq \qablhat^{*} / (\bar{\chi} -1|\ \chi \in N)$.
\end{enumerate}
\end{lema}

\pf
$ (a) $ 
Let $\chi \in M_{I}$ and define 
$\bar{\chi} \in G(\qablhat^{*})$
on the generators of $\qablhat$ by 
$$\bar{\chi}(e_{i}) = 0,\qquad\bar{\chi}(f_{j}) = 0
\quad \mbox{ and } \quad \bar{\chi}(h_{k}) = \chi(h_{k})
\quad\mbox{ for all }i\in I_{+},\ j\in I_{-}.$$
To see that is well-defined, it is enough to show that 
$$\bar{\chi}(e_{i}f_{i} - f_{i}e_{i} ) = 
\frac{1}{\alpha - \beta}\bar{\chi}(a_{i}b_{i+1} - a_{i+1}b_{i})
\qquad\mbox{ for all }i\in I_{+}\cap I_{-}.$$ 
But $\bar{\chi}(a_{i}b_{i+1} - a_{i+1}b_{i}) = 
\bar{\chi}(a_{i}b_{i+1}) - \bar{\chi}(a_{i+1}b_{i}) $ 
$=\chi(w_{i}) - \chi(w'_{i}) = 0 $, since $w_{i}$ and
$w'_{i} \in \mathbb{T}_{I}$. Let us see now that $\bar{\chi}$ is central.
Since $\qablhat$ admits a triangular decomposition
which is induced by the triangular decomposition of $\qabhat$,
we may assume that an element of $\qablhat$ is a linear combination
of elements of the form $h m$ with $h \in \mathbb{T}$ and $m$ is a product
of some powers of the elements $e_{i},\ i \in I_{+}$ 
and $f_{j}, \ j \in I_{-}$.
Let $\theta \in \qablhat^{*}$, then 
\begin{align*} 
\bar{\chi}\theta(hm) & = \bar{\chi}(hm_{(1)}) \theta(hm_{(2)}) =
\bar{\chi}(h)\bar{\chi}(m_{(1)}) \theta(hm_{(2)})\\
& =
\chi(h)\epsilon(m_{(1)}) \theta(hm_{(2)})= \chi(h) 
\theta(hm)\mbox{ and }\\
\theta\bar{\chi}(hm) & =  \theta(hm_{(1)})\bar{\chi}(hm_{(2)}) =
\theta(hm_{(1)}) \bar{\chi}(h)\bar{\chi}(m_{(2)})\\
& =
\theta(hm_{(1)}) \chi(h)\epsilon(m_{(2)})= \chi(h) \theta(hm),
\end{align*}
which implies that $\bar{\chi}$ is central.

\par $ (b)$  By $(a)$ we know that $\bar{\chi}$ is a 
central group-like
element of $\qablhat^{*}$ for all $\chi \in N$. 
Hence the quotient
$\qablhat^{*} / (\bar{\chi} - 1\vert \chi\in N)$ 
is a Hopf algebra.

\par On the other hand, 
we know that $H^{*}$ is determined
by the triple $(\Sigma, I_{+},I_{-})$ and consequently 
$H^{*}$ is
included in $\qablhat$. If we denote $v: \qablhat^{*} \to H$ 
the surjective
map induced by this inclusion, we have that $\Ker v = \{f\in
\ql^{*}:\ f(h) = 0,\ \forall\ h\in H^{*}\}$. But $\bar{\chi} -1 \in
\Ker v$ for all $\chi\in N$, by definition. Hence 
there exists a surjective
Hopf algebra map
$$\gamma: \qablhat^{*}/(\bar{\chi}-1\vert\
\chi\in N) \twoheadrightarrow H.$$
But by the proof of Lemma \ref{lema:qablhat} and 
the fact that $\qablhat$ has a PBW-basis we have that
$$
\dim H  = \ell^{\vert I_{+}\vert + \vert I_{-} \vert}\vert 
\widehat{\Sigma}
\vert  = \ell^{\vert I_{+}\vert +
\vert I_{-} \vert}
\frac{\ell^{n}}{\vert N \vert} = 
\dim (\qablhat^{*}/(\bar{\chi}-1\vert\ \chi\in N)),
$$
which implies that $\gamma$ is an isomorphism.
\epf

Before going on with the construction we need the following
technical lemma. Let $\x = \{\bar{\chi}\vert\ \chi\in
M_{I}\}$ be the set of central group-like
elements of $\qablhat^{*}$ given by Lemma 
\ref{lema:central-grouplikes-1}.

\begin{lema}\label{lema:grupoz}
There exists a subgroup $\z:=\{\partial^{m}\vert\ m \in
M_{I}\}$ of $G(A_{\lgot, \sigma})$ isomorphic
to $\x$ consisting of central elements.
\end{lema}

\pf By Lemma \ref{lema:qablhat} $(d)$, 
we know that there exists
a coalgebra section
$\gamma_{L}: \qablhat^{*} \to \Oeal$. 
Then, we obtain the group
of group-like elements given by 
$\y = \{\gamma_{L}(\bar{\chi})\vert\ 
\chi\in
M_{I}\}$ in $\Oeal$. Let $d = \bar{x}_{11}\bar{x}_{22}
\cdots \bar{x}_{nn} \in \Oeal$ and denote $\pi_{L}(d) = D$.
By Remark \ref{obs:hat-t} we know that each 
$\bar{\chi} \in M_{I} \subseteq \mathbb{T} \simeq (\ZZ /\ell \ZZ)^{n}$ 
corresponds to an
element $D^{m}$ with $m \in (\ZZ /\ell \ZZ)^{n}$. 
By abuse of notation
we write $\bar{\chi} = D^{m}$ with $m\in M_{I}$.
Then using the definition
of $\gamma_{L}$ and the elements
$d^{m},\ m \in M_{I}$, it can be seen that 
the elements of $\y$ are central 
and can be described
as $\y = \{d^{m}\vert\ m\in
M_{I}\}$.

\par Since the map $\nu: \Oeal \to A_{\lgot, \sigma}$ given by
the pushout construction is surjective, the image of $\y$ defines a
group of central group-likes in $A_{\lgot, \sigma}$:
$$\z = \{\partial^{m} = \nu(d^{m})\vert\ m\in
M_{I}\}.$$
Besides, $\vert \z \vert = \vert \y \vert = \vert \x
\vert = \vert M_{I} \vert$. Indeed, $\bar{\pi}(\z) = \bar{\pi}\nu(\y) =
\pi_{L}(\y) = \pi_{L}\gamma_{L}(\x)= \x$ since the diagram
\eqref{diag:pushout-ol} is commutative and $\pi_{L}\gamma_{L} = \id$.
Hence $\vert \bar{\pi}(\z) \vert = \vert\x \vert$, from which the
assertion follows. \epf

For the definition of a
subgroup datum $\D= (I_{+}, I_{-}, N, \Ga,
\sigma, \delta)$, see
Def. \ref{def:subgroupdatum}. 
We prove next our first main result, which is needed to 
prove Thm. \ref{thm:main}.
 
\begin{theorem}\label{teo:constrfinal}
Let $\D = (I_{+}, I_{-}, N, \Ga, \sigma, \delta)$ be a subgroup
datum. Then there exists a Hopf algebra $A_{\D}$ which is a
quotient of $\Oea$ and fits into the central exact sequence
$$ 1\to \Oc(\Ga) \xrightarrow{\hat{\iota}} A_{\D}
\xrightarrow{\hat{\pi}} H \to 1. $$
Concretely, $A_{\D}$ is given by the quotient $A_{\lgot,
\sigma}/J_{\delta}$ where $J_{\delta}$ is the two-sided ideal
generated by the set $\{\partial^{z} - \delta(z)\vert\	 z\in N\}$
and the following diagram of exact sequences of Hopf algebras is
commutative
\begin{equation}\label{diag:const}
\xymatrix{1 \ar[r]^{}
& \Oc(GL_{n}) \ar[r]^{\iota} \ar[d]_{\res} & \Oea \ar[r]^{\pi}
\ar[d]^{\Res} & \qabhat^{*} \ar[r]^{}\ar[d]^{p} & 1\\
1 \ar[r]^{} &   \Oc(L)\ar[d]_{^{t}\sigma}  \ar[r]^{\iota_L} & \Oeal
\ar[r]^{\pi_L}\ar[d]^{\nu} &
\qablhat^* \ar[r]^{}\ar@{=}[d] & 1\\
1 \ar[r]^{} &   \Oc(\Ga)  \ar[r]^{j}\ar@{=}[d] & A_{\lgot,
\sigma}\ar[r]^{\bar{\pi}}\ar[d]^{t} & \qablhat^* \ar[r]^{}\ar[d]^{v} & 1\\
1 \ar[r]^{} &   \Oc(\Ga)  \ar[r]^{\hat{\iota}} &
A_{\D}\ar[r]^{\hat{\pi}}& H\ar[r]^{}& 1.}\end{equation}
\end{theorem}

\pf By Remark \ref{obs:hat-t}, $N$ determines a subgroup
$\Sigma$ of $\mathbb{T}$ and the triple $(\Sigma, I_{+}, I_{-})$
gives rise to a surjective Hopf algebra map $r: \qe^{*} \to H$.
Since $\sigma: \Ga \to L \subseteq G$ is injective, by the first
two steps developed before one can construct a Hopf algebra
$A_{\lgot,\sigma}$ which is a quotient of $\Oea$ and an extension
of $\Oc(\Ga)$ by $\qablhat^{*}$, where $\qablhat$ is the Hopf 
subalgebra of
$\qe$ associated to the triple $(\mathbb{T},I_{+},I_{-})$.
Moreover, by Lemma \ref{lema:central-grouplikes-1} $(b)$ and 
the proof of Lemma \ref{lema:grupoz}, $H$ is the
quotient of $\qablhat^{*}$ by the two-sided ideal 
$(D^{m}-1\vert\ m\in
N)$. If $\delta: N \to \widehat{\Ga}$ is a group map, then the
elements $\delta(m)$ are central group-like elements in $A_{\lgot,
\sigma}$ for all $m \in N$, and the two-sided ideal $J_{\delta}$
of $A_{\lgot, \sigma}$ generated by the set $\{\partial^{m} -
\delta(m)\vert m\in N\}$ is a Hopf ideal, because by Lemma
\ref{lema:grupoz} the elements $ \partial^{m} $
are central.
Hence, by \cite[Prop.
3.4 (c)]{Mu} the following sequence is exact
$1\to \Oc(\Ga)/ \J_{\delta} \to A_{\lgot,
\sigma}/J_{\delta} \to \ql^{*} /\bar{\pi}(J_{\delta})\to 1$,
where $\J_{\delta} = J_{\delta} \cap \Oc(\Ga)$. Since
$\bar{\pi}(\partial^{m}) = D^{m}$ and $\bar{\pi}(\delta(m)) = 1$ for
all $m\in N$, we have that $\bar{\pi}(J_{\delta})$ is the two-sided
ideal of $\qablhat^{*}$ given by $(D^{m} -1\vert\ m\in N)$, 
which implies
by Lemma \ref{lema:central-grouplikes-1} $(b)$ that $\qablhat^{*} /
\bar{\pi}(\J_{\delta}) = H$. Hence, if we denote $A_{\D} : =
A_{\lgot, \sigma}/J_{\delta}$, we can re-write the exact sequence of
above as
$$\xymatrix{1\to \Oc(\Ga)/ \J_{\delta} \to
A_{\D} \to H \to 1.}$$
To finish the proof 
it is enough to see that $\J_{\delta} = J_{\delta} \cap
\Oc(\Ga)= 0$, and this can be proved
exactly as it was proved in \cite[Thm. 2.17]{AG}. \epf


\subsection{Characterization of the quotients}
\label{subsec:a-b-root-charac}
Let $\theta: \Oea \to A$ be a surjective Hopf algebra map. In the following
we prove that $A\simeq A_{\D}$ as Hopf algebras for some subgroup
datum $\D$. 

\par By the first 
paragraph of Subsection \ref{subsec:a-b-root-constr}, we know that 
$ A $ fits into a commutative diagram

\begin{equation}\label{diag:propqsubgr-2}
 \xymatrix{1
\ar[r]^{} & \Oc(GL_{n}) \ar[r]^{\iota} \ar[d]_{^{t}\sigma}  & \Oea
\ar[r]^{\pi}
\ar[d]^{\theta} & \qabhat^{*} \ar[r]^{}\ar[d]^{r} & 1\\
1 \ar[r]^{} &   \Oc(\Ga) \ar[r]^{\hat{\iota}} & A
\ar[r]^{\hat{\pi}} & H \ar[r]^{} & 1,}
\end{equation}
where 
$\sigma: \Ga\to GL_{n}$ is an injective map of algebraic groups
and $H^{*}$ is a Hopf subalgebra of $\qabhat$ determined 
by a triple $(\Sigma,
I_+, I_-)$. 
Let $N$ correspond to $\Sigma$ as in the 
paragraph before Lemma \ref{lema:central-grouplikes-1}. 
Our aim is to show that there exists
$\delta$ such that $A \simeq A_{\D}$ for the subgroup datum $\D =
(I_{+}, I_{-}, N, \Ga, \sigma, \delta)$. Recall from 
Subsection \ref{subsub:first-step}
the definition
of the Hopf subalgebra $\qablhat$ of $\qabhat$ 
given by the sets $I_{+}$ and $I_{-}$. In particular,
$H^{*} \subseteq \qablhat \subseteq \qabhat$. Denote by
$v :\qablhat^{*} \to H$  
the surjective Hopf algebra map induced by the inclusion.
The following lemma is crucial for the 
determination of the quantum subgroups; see 
\cite[Lemma 3.1]{AG}.

\begin{lema}\label{lema:fact}
There exist Hopf algebra maps $u,\ w$ such that
the following diagram with exact rows commutes.
In particular, $\sigma(\Ga) \subseteq L \subseteq GL_{n}$. 
$$\xymatrix{1
\ar[r]^{} & \Oc(GL_{n})
\ar[r]^{\iota}\ar@/_2pc/@{.>}[dd]_(.7){^{t}\sigma} \ar[d]_{\res} &
\Oea \ar[r]^{\pi} \ar[d]^{\Res}\ar@/_2pc/@{.>}[dd]_(.7){\theta} 
& \qabhat^{*}
\ar[r]^{}\ar[d]_{p}
\ar@/^2pc/@{.>}[dd]^(.7){r} & 1\\
1 \ar[r]^{} &   \Oc(L)\ar[d]^{u}  \ar[r]^{\iota_L} & \Oeal
\ar[r]^{\pi_L}\ar[d]^{w} &
\qablhat^* \ar[r]^{}\ar[d]_{v} & 1\\
1 \ar[r]^{} &   \Oc(\Ga)  \ar[r]^{\hat{\iota}} &
A\ar[r]^{\hat{\pi}} & H \ar[r]^{} & 1.}$$
\end{lema}

\pf To show the existence of the maps $u$ and $w$ it is enough to
show that $\Ker \Res \subseteq \Ker \theta$, since $u$ is simply $w
\iota_{L}$. This clearly implies that $v\pi_L = \hat\pi w$.

\par By \eqref{eq:def-O-L} we know that $\Ker \Res = \II$, where
$\II$ is the two-sided ideal of $\Oea$ generated
by the elements $\left\{x_{i,j} \vert\ (i,j)\in \II_{+}\cup \II_{-}  \right \}$
and 
$$\II_{+} = \left\{(i,j) \vert\ i\leq k < j,\ k\notin I_{+}\right \}
\mbox{ and }\ 
\II_{-} = \left\{(i,j) \vert\ j\leq k < i,\ k \notin I_{-} \right \}.
$$
Then $\hat{\pi}(\theta(x)) = r\pi(x) = vp\pi(x) = 
v \pi_{L} \Res (x) = 0$ for all
$x \in \II$, and this implies for all $x \in \II$ that 
$\theta (x) \in \Oc(\Ga)^{+}A $ $=
\theta(\Oc(GL_{n})^{+}\Oea)$. Then for all $x \in 
\Ker \Res$, there exist 
$a \in \Oc(GL_{n})^{+}\Oea$ and $c \in \Ker \theta$
such that $x = c+a$. In particular, this holds for 
all $x_{ij}$ with $(i,j) \in \II_{+}\cup \II_{-}$, that is
$x_{ij} = a_{ij} + c_{ij}$, for some $a_{ij} 
\in \Oc(GL_{n})^{+}\Oea $ and $c_{ij} \in \Ker \theta$.
Comparing
degrees in both sides of the equality we have that
$a_{ij} = 0$, which implies that each generator of $\II$
must lie in $\Ker \theta$.
\epf

\begin{obs}\label{rmk:gamma-matrices} 
Any algebraic group $\Ga$ 
appearing in an exact sequence given by a quotient
of $\Oea$ must be composed by block matrices 
$M = (m_{ij})_{1\leq i,j\leq n}$ such
that $m_{ij} = 0$ if $(i,j) \in \II_{+}\cup \II_{-}$.
\end{obs}

The following lemma shows the convenience of characterizing the
quotients $A_{\lgot, \sigma}$ of $\Oea$ as pushouts.

\begin{lema}\label{lema:Acociente}
$A$ is a quotient of
$A_{\lgot, \sigma}$ and the following
diagram commutes
\begin{equation}\label{diag:gordo}
\xymatrix{1 \ar[r]^{} & \Oc(GL_{n}) \ar[r]^{\iota} \ar[d]_{\res} & \Oea
\ar[r]^{\pi}
\ar[d]^{\Res} & \qabhat^{*} \ar[r]^{}\ar[d]^{p} & 1\\
1 \ar[r]^{} &   \Oc(L)\ar[d]_{u}  \ar[r]^{\iota_L} & \Oeal
\ar[r]^{\pi_L}\ar[d]^{\nu} &
\qablhat^* \ar[r]^{}\ar@{=}[d] & 1\\
1 \ar[r]^{} &   \Oc(\Ga)  \ar[r]^{j}\ar@{=}[d] & A_{\lgot,
\sigma}\ar[r]^{\bar{\pi}}\ar[d]^{t} 
& \qablhat^* \ar[r]^{}\ar[d]^{v} & 1\\
1 \ar[r]^{} &   \Oc(\Ga)  \ar[r]^{\hat{\iota}} &
A\ar[r]^{\hat{\pi}}& H\ar[r]^{}& 1.}
\end{equation}
\end{lema}

\pf Using the maps $u,\ w$ given by Lemma \ref{lema:fact}, we have
that $w\iota_{L} = \hat{\iota}u$, that is, the following diagram
commutes

$$\xymatrix{\Oc(L)
\ar[r]^{\iota_{L}} \ar[d]_{u} &
\Ol \ar[d]^{\nu} \ar@/^/[ddr]^{w}\\
\Oc(\Ga) \ar[r]_{j} \ar@/_/[drr]_{\hat{\iota}}& A_{\lgot, \sigma}
\\
& & A.}$$
Since $A_{\lgot, \sigma}$ is a pushout, there exists a
unique Hopf algebra map $t: A_{\lgot, \sigma} \to A$ such that $tv
= w$ and $tj = \hat{\iota}$. 
This implies also that $ v\bar{\pi} = \hat{\pi}t$ and
therefore the diagram \eqref{diag:gordo}  is commutative. \epf

Let $(\Sigma, I_{+}, I_{-})$ be the triple that determines $H$.
Recall that $\mathbb{T}_{I}$ is the subgroup of
$\mathbb{T}$ generated by the elements 
$\{w_{i}, w'_{j}:\ i\in I_{+},\ j\in I_{-}\}$,
$\rho: \widehat{\mathbb{T}} \to \widehat{\Sigma}$ and
$\rho_{I}: \widehat{\mathbb{T}} \to \widehat{\mathbb{T}_{I}}$
denote the group homomorphisms between the character groups induced
by the inclusions and $N = \Ker \rho$, $M_{I} = \Ker \rho$.

\par By Lemmata \ref{lema:central-grouplikes-1} 
and \ref{lema:grupoz},
we know that the Hopf algebra $A_{\lgot, \sigma}$ contains a set of
central group-like elements $\z= \{\partial^{m}\vert\ m\in
M_{I}\}$ such that $\bar{\pi}(\partial^{m}) =
D^{m}$ for all $m\in M_{I}$ and $H =
\qablhat^{*}/(D^{m} -1\vert\ m\in N)$. To see that $A = A_{\D}$ for a
subgroup datum $\D = (I_{+}, I_{-}, N, \Ga, \sigma, \delta)$ it
remains to find a group map $\delta: N \to \widehat{\Ga}$ such that
$A \simeq A_{\lgot, \sigma}/ J_{\delta}$. This is given by the following
lemma which finishes one part of the proof of Thm. \ref{thm:main}.
The proof will be completed by Thm. \ref{teo:order-datum}
in the following subsection.

\begin{lema}\label{lema:thetaiso}
There exists a group homomorphism $\delta: N \to \widehat{\Ga}$ such
that $J_{\delta} = (\partial^{m} - \delta(m)\vert\ m\in N)$ is a
Hopf ideal of $A_{\lgot, \sigma}$ and $A\simeq A_{\D}= A_{\lgot,
\sigma}/ J_{\delta}$.
\end{lema}

\pf Let $\partial^{m} \in \z$. Then $\hat{\pi}t(\partial^{m}) =
v\bar{\pi}(\partial^{m}) = 1$ for all $m \in N$, by Lemma
\ref{lema:central-grouplikes-1} $(b)$. Since $t(\partial^{m})$ is a
group-like element, this implies that $t(\partial^{m}) \in A^{\co
\hat{\pi}} = \Oc(\Ga)$. As $G(\Oc(\Ga)) = \widehat{\Ga}$, we have a
group homomorphism $\delta$ given by
$$\delta: N \to \widehat{\Ga},\qquad \delta(m) =
t(\partial^{m})\quad\forall\ m\in N.$$

Since the elements $ \delta(m)$ and  $
\partial^{m} $ are central in $A_{\lgot,\sigma}$, 
the two-sided ideal given by $J_{\delta} =
(\partial^{m}-\delta(m)\vert\ m\in N)$ is a Hopf ideal and
$t(J_{\delta}) = 0$. Consequently, we have a surjective Hopf
algebra map $ \theta: A_{\D}  \twoheadrightarrow A$, which makes
the following diagram commutative
\begin{equation}\label{diag:episeq}
\xymatrix{1 \ar[r]^{} &   \Oc(\Ga)
\ar[r]^{\tilde{\iota}}\ar@{=}[d] & A_{\D}
\ar[r]^{\tilde{\pi}}\ar@{->>}[d]^{\theta} & H\ar[r]^{}\ar@{=}[d] & 1\\
1 \ar[r]^{} &  \Oc(\Ga)  \ar[r]^{\hat{\iota}} & A
\ar[r]^{\hat{\pi}}& H\ar[r]^{}& 1.}
\end{equation}
Then $\theta$ is an isomorphism by \cite[Cor. 1.15]{AG}. \epf

\subsection{Relations between quantum
subgroups}\label{sec:equivalence}

We study now the category $ \quot(\Oea)$ of quotients
of $\Oea$ when $\alpha$ and $\beta$ satisfy the conditions 
given in \eqref{eq:cond-alpha-beta},
and show that every object is parametrized by a subgroup
data $\D$ of $ \Oea$. 

\par Let $U$ be any Hopf algebra and consider the category $\quot (U)$,
whose objects are surjective Hopf algebra maps $q: U \to A$. If $q:
U \to A$ and $q': U \to A'$ are such maps, then an arrow
$\xymatrix{q\ar[0,1]^{\alpha}& q'}$ in $\quot (U)$ is a Hopf algebra
map $\alpha: A\to A'$ such that $\alpha q = q'$. In this language, a
\emph{quotient} of $U$ is just an isomorphism class of objects in
$\quot (U)$; let $[q]$ denote the class of the map $q$. There is a
partial order in the set of quotients of $U$, given by $[q]\leq
[q']$ iff there exists an arrow $\xymatrix{q\ar[0,1]^{\alpha}& q'}$
in $\quot (U)$. Notice that $[q]\leq [q']$ and $[q']\leq [q]$
implies $[q] = [q']$.

\par We will describe the partial order in the set $[q_{\D}]$,
$\D$ a subgroup datum, of quotients $q_{\D}: \Oea\twoheadrightarrow
A_{\D}$ given by Thm. \ref{teo:constrfinal}, using the
data provided by the set of subgroup data. Eventually, this
will be the partial order in the set of all quotients of $\Oea$. We
begin by the following definition. By an abuse of notation we
write $[A_{\D}] = [q_{\D}]$.

\begin{definition}\label{defi:order-datum} Let
$\D= (I_{+}, I_{-}, N, \Ga, \sigma, \delta)$ and $\D'= (I'_{+},
I'_{-}, N', \Ga', \sigma', \delta')$ be subgroup data of $\Oea$. We say that
$\D \leq \D'$ iff

\begin{itemize}
\item[$\bullet$] $I'_{+} \subseteq I_{+}$ and $I'_{-} \subseteq
I_{-}$.

\smallbreak\item[$\bullet$] $N\subseteq N'$ or equivalently
$\Sigma' \subseteq \Sigma$.

\smallbreak\item[$\bullet$] There exists a morphism of algebraic
groups $\tau:\Ga' \to \Ga$ such that $\sigma \tau = \sigma'$.

\smallbreak\item[$\bullet$] $\delta'\eta = \ ^{t}\tau \delta$,
where $\eta: N \hookrightarrow N'$ denotes the canonical inclusion.
\end{itemize}

If we denote $I = I_{+}\cup I_{-}$ 
and $I' = I'_{+}\cup I'_{-}$, the first condition implies that $I'\subseteq
I$, $\mathbb{T}_{I'} \subseteq \mathbb{T}_{I}$. Denote by
$\rho_{I}: \widehat{\mathbb{T}} \to \widehat{\mathbb{T}_{I}}$
 and $\rho_{I'}: \widehat{\mathbb{T}} \to \widehat{\mathbb{T}_{I'}}$ 
the epimorphisms between the character groups induced
by the inclusions and let $M_{I} = \Ker \rho_{I}$,  
$M_{I'} = \Ker \rho_{I'}$. Then $M_{I} \subseteq M_{I'}$.

\par We say that $\D \simeq \D'$ iff $\D \leq \D'$ and $\D'
\leq \D$. This means 
that $I_{+} = I'_{+}$ and $I_{-} = I'_{-}$, 
$N = N'$,
there exists an isomorphism of
algebraic groups $\tau:\Ga' \to \Ga$ such that $\sigma \tau =
\sigma'$ and $\delta' = \ ^{t}\tau
\delta$.
\end{definition}

The following theorem completes the proof of 
Thm. \ref{thm:main} $ (b) $
and its proof is completely analogous to \cite[Thm. 2.20]{AG}.

\begin{theorem}\label{teo:order-datum} Let
$\D$ and $\D'$ be subgroup data. Then

(a) $[A_{\D}] \leq [A_{\D'}]$ iff $\D\leq\D'$.

(b) $[A_{\D}] = [A_{\D'}]$ iff $\D\simeq\D'$.\qed
\end{theorem}


\subsection{Some properties of the quotients}\label{findimquo}
In this subsection we summarize some properties of the quotients
$A_{\D}$ following the study made
in \cite{AG1}. Let $\D= (I_{+}, I_{-}, N, \Ga, \sigma, \delta)$ be
a subgroup datum of $ \Oea $. By Thm.
\ref{teo:constrfinal}, $A_{\D}$ fits into the
commutative diagram

\begin{equation}\label{diag:facOel}
\xymatrix{1 \ar[r]^{} & \Oc(GL_{n})
\ar[r]^{\iota}\ar@/_2pc/@{.>}[dd]_(.7){^{t}\sigma} \ar[d]_{\res} &
\Oea \ar[r]^{\pi} \ar[d]^{\Res}\ar@/_2pc/@{.>}[dd]_(.7){q_{\D}} &
\qabhat^{*} \ar[r]^{}\ar[d]_{p}
\ar@/_2pc/@{.>}[dd]_(.7){r} & 1\\
1 \ar[r]^{} &   \Oc(L)\ar[d]^{u}  \ar[r]^{\iota_L} & \Oeal
\ar[r]^{\pi_L}\ar[d]^{w} &
\qablhat^* \ar[r]^{}\ar[d]^{v} & 1\\
1 \ar[r]^{} &   \Oc(\Ga)  \ar[r]^{j} &
A_{\D}\ar[r]^{\hat{\pi}} & H \ar[r]^{} & 1.}
\end{equation}
and by Lemma \ref{lema:central-grouplikes-1} $ (b) $
and the proof of Lemma \ref{lema:grupoz},
$H\simeq $ $\qablhat^{*} /
(D^{m} - 1\vert\ m\in N)$. Let $ \Tte $ be the 
diagonal torus of $GL_{n}(\CC)$. The following
lemma shows that it coincides with the group
of characters of $\Oea$.

\begin{lema}\label{lema:dualnonpointed} 
\begin{enumerate}
\item[$(a)$] $\Alg(\Oea, \CC) \simeq \Tte$.
\item[$(b)$] $j$ induces a group map $^{t}j: \Alg(A_{\D}, \CC) \to
\Ga$ and $\Img (\sigma\circ\ ^{t}j) \subseteq \Tte\cap
\sigma(\Ga)$.
\end{enumerate}
\end{lema}

\pf
$ (a) $ Let $ \Lambda \in \Tte $ with main diagonal 
$(\lambda_{1},\ldots, \lambda_{n}) \in (\CC^{\times})^{n}$. 
It defines an element  $\hat{\Lambda}\in \Alg(\Oea, \CC)$ by setting
$\hat{\Lambda} (x_{ij}) =\delta_{i,j} \lambda_{i}$ for all $1\leq i,j\leq n$. 
Hence we may define a group homomorphism $\varphi: \Tte \to \Alg(\Oea, \CC)$
by $\varphi(\Lambda) = \hat{\Lambda}$ for all $\Lambda \in \Tte$. 
This group map is clearly injective, so we need to prove that is also
surjective. Let $\theta \in \Alg(\Oea, \CC)$ and let $x_{ij}$   
be a generator of $ \Oea $ with $1\leq i\neq j\leq n$. Then 
from the defining relations it follows that
$\theta(x_{ij}) = 0$. Since $\theta(g) \neq 0$, we have that
$\theta(x_{kk}) \in \CC^{\times}$ and thus $ \theta = \varphi(\Lambda) $
with $ \lambda_{k} =  \theta(x_{kk}) $ for all $1\leq k \leq n$.

\par $ (b) $ The bottom exact sequence of \eqref{diag:facOel}
induces an exact sequence of groups
\begin{equation}\label{extgrupos} 1 \to G(H^{*}) = \Alg(H, \CC)
\xrightarrow{^{t}\pi} \Alg(A_{\D}, \CC) \xrightarrow{^{t}j}
\Alg(\Oc(\Ga), \CC) = \Ga,
\end{equation}
which fits into the commutative diagram of group maps
$$\xymatrix{1
\ar[r]^{} & G(\qabhat) \ar[r]^(.4){^{t}\pi} & \Alg(\Oea, \CC)
\ar[r]^(.45){^{t}\iota}
& \Alg(\Oc(GL_{n}),\CC) \\
1 \ar[r]^{} &   G(H^{*})  \ar[r]^{^{t}\pi}\ar[u]^{^{t}r} &
\Alg(A_{\D}, \CC) \ar[r]^{^{t}j}\ar[u]_{^{t}q_{\D}} & \Ga
\ar[u]_{\sigma}.}$$
Since $q_{\D}$ is surjective, $^{t}q_{\D}: \Alg(A_{\D}, \CC) \to \Alg(\Oea,
\CC)$ is injective. Thus the subgroup $(\sigma\circ\
^{t}j)(\Alg(A_{\D}, \CC))$ of $\sigma(\Ga)$ must be a subgroup
of $\Img\ ^{t}\iota = \Tte \subseteq GL_{n}=\Alg(\Oc(GL_{n}),\CC)$. 
\epf

We resume some properties of the Hopf algebras
$A_{\D}$ in the following proposition. Recall that
a twist of a Hopf algebra $A$ is an invertible element
$J \in A\otimes A$ such that
$$
 (1\ot J)(\id\ot \com)J  = (J\ot 1)(\com\ot \id)J\quad \mbox{ and }
\quad (\id\ot \epsilon)J  = 1 = (\epsilon \ot 1)J. 
$$
If $A^{J}$ denotes the Hopf algebra with 
the same algebra structure as $A$ but with the
comultiplication given by $\com_{J} = J\com J^{-1}$, then
we say that $A^{J}$ is a twist deformation of $A$.
We say that two Hopf algebras $A$ and $B$ are twist
equivalent if $B\simeq A^{J}$ for some twist $J$. 
The {\it Hopf center} of a Hopf algebra $A$ is the maximal central
Hopf subalgebra $\HZe(A)$ of $A$, see 
\cite[2.2.2]{A}. 

\begin{prop}\label{prop:propiedades}
Let $\D= (I_{+}, I_{-}, N, \Ga, \sigma, \delta)$ be a subgroup
datum.
\begin{enumerate}
  \item[$(a)$] If $A_{\D}$ is pointed, then $I_{+}\cap I_{-} =
  \emptyset$ and $\Ga$ is a subgroup of the group of upper 
  triangular matrices of some size.
  In particular, if $\Ga$ is finite, then it is abelian.
  \item[$(b)$] $A_{\D}$
  is semisimple if and only if $I_{+} \cup I_{-} = \emptyset$ and
  $\Ga$ is finite.
  \item[$(c)$] If $\dim A_{\D} < \infty$ and $A_{\D}^{*}$ is pointed, then
   $\sigma(\Ga)\subseteq \Tte$. 
  \item[$(d)$] If $A_{\D}$ is co-Frobenius then $\Ga$ is reductive.
  \item[$(e)$] If $\HZe(A_{D}) \ncong \HZe(A_{D'})$ then $A_{D}$ and $A_{D'}$
  are not twist equivalent.
\end{enumerate}
\end{prop}

\pf Items $(a), \ldots, (d)$ are a small variation of 
\cite[Prop. 3.8]{AG1}.

\smallbreak $(e)$ 
If $A_{D} \cong A_{D'}^{J}$ for some twist $J \in
A_{D'}\ot A_{D'}$, then $\HZe(A_{D}) \cong \HZe(A_{D'}^{J})$. Since
$\HZe(A_{D'}^{J}) = \HZe(A_{D'})$, the claim follows.
\epf

We end the paper with the following theorem that gives
a new family of Hopf algebras coming from deformations
on two paramaters which can not be obtained as 
quotients of $\Oe$, with $G$ a connected,
simply connected,  simple complex Lie group, $\epsilon$
a primitive $s$-th root of unity, $s$ odd and $3\nmid s$ if 
$G$ is of type $G_{2}$, see \cite{AG}. 
Recall that if $(a_{ij})_{1\leq i,j\leq n}$
is a Cartan matrix, 
a subset $I \subseteq I_{n}= \{1,\ldots , n-1\}$
is called \textit{connected} if for all $i,j \in I$
there exist $i= k_{1}, k_{2},\ldots, k_{r+1}= j$ such
that $a_{k_{t}k_{t+1}}\neq 0$. 

\begin{theorem}\label{thm:pulenta}
Let $\D = (I_{+}, I_{-}, N, \Ga, \sigma, \delta)$ be a finite
subgroup datum such that $I_{+}\cap I_{-} \neq \emptyset$ and
$\sigma(\Ga)\nsubseteq\Tte$. Then $A_{\D}$ is non-semisimple and
non-pointed with non-pointed dual. If moreover $I= I_{+}=I_{-}$ 
and $I$ is connected, then $A_{\D}$ cannot be obtained as a quotient 
of $\Oe$.
\end{theorem}

\pf
The first assertion follows from Prop. \ref{prop:propiedades}. 
Thus we show that $A_{\D}$ can not be obtained as
a quotient of $\Oe$.

\par Suppose that $I= I_{+}=I_{-}$ and let $\qablohat$ 
be the Hopf subalgebra of $\qablhat$ generated by 
$\{e_{i}, f_{i}, w_{i}, w'_{i}:\ i\in I\}$. 
Then $\HZe(\qablohat^{*}) = \CC$. Indeed, suppose on the 
contrary that $C = \HZe(\qablohat^{*}) \neq \CC$, then 
we have a central exact sequence $1\to C \to \qablohat^{*}
\to B \to 1$, where $B = \qablohat^{*}/C^{+}\qablohat^{*}$
is non-trivial.
Taking duals we get an exact sequence  
$1\to B^{*} \to \qablohat
\to C^{*} \to 1$, with $C^{*}$ a group algebra. Since
$\qablohat$ is pointed, by  \cite[Cor. 1.12]{AG},  $B^{*}$ is generated 
by  a subset of $\{e_{i}, f_{i}:\ i\in I\}$ and a subgroup 
$F$ of $\langle w_{i},w'_{i}:\ i\in I\rangle$ such that
$w_{j} \in F$ and $w'_{j} \in F$ if $e_{j} \in B^{*}$
or $f_{j} \in B^{*}$, respectively. But if $e_{j}$ or
$f_{j}$ belongs to $B^{*}$, then $w_{j}\in B^{*}$ or 
$w'_{j} \in B^{*}$ and by \cite[Lemma A.1]{AS-app}
and the fact that $I$ is connected this would imply that
$B^{*}= \qablohat$, a contradiction. Hence 
$B^{*}$ must be generated by group-likes. This 
is also not possible, since this would imply that $\qablohat$
is semisimple. Thus we must have 
$\HZe(\qablohat^{*}) \neq \CC$.

\par By Lemma \ref{lema:central-grouplikes-1} $(i)$
the set of group-likes $\x = \{\bar{\chi}\vert\ \chi\in
M_{I}\}$ of
$\qablhat^{*}$ is central and we have that $\qablhat^{*}/ 
(\bar{\chi}-1\vert\ \chi\in M_{I}) \simeq \qablohat^{*}$.
Since $N\subseteq M_{I}$, $\qablohat$ is a
quotient of $H$ and by Lemma  
\ref{lema:central-grouplikes-1} $(ii)$ we have 
$\qablohat^{*}\simeq H / (\bar{\chi}-1\vert\ \chi\in M_{I}/ N) $.
On the other hand, by Lemma \ref{lema:grupoz}
and Thm. \ref{teo:constrfinal}, $A_{\D}$
contains a group of central group-likes isomorphic to
$\w = \{\bar{\chi}\vert\ \chi\in
M_{I}/N\}$ and the Hopf subalgebra $\Oc(\widetilde{\Ga}):= \Oc(\Ga)\w$
is central in $A_{\D}$. Following \cite[Lemma 3.10]{AG1}, one
can prove that $\HZe(A_{\D}) =  \Oc(\widetilde{\Ga})$, 
$A_{\D}$ is given by a pushout and
$A_{\D}$ fits into the central exact sequence
$$1\to \Oc(\widetilde{\Ga}) \to A_{\D}
\to \qablohat^{*} \to 1. $$

Suppose $A_{\D} \simeq A $ as Hopf algebras, with 
$A$ a quotient of $\Oe$. Then  
by \cite[Lemma 3.10]{AG1}, $\Oe$ fits into a central 
exact sequence 
$$1\to \Oc(\widetilde{\Ga}_{1}) \to A
\to \qeo^{*} \to 1, $$
with $\qeo$ a Hopf subalgebra of $\qe$, $\HZe(A) =  
\Oc(\widetilde{\Ga}_{1})$ and 
$\HZe(\qeo^{*}) =  
\CC$. This implies in particular that $\qeo \simeq \qablohat$. 
Since both pointed Hopf algebras are generated
by group-likes and skew-primitives, by looking 
at the linear spaces generated by the skew-primitives 
we should have that $\alpha = \beta^{-1}$, 
which contradicts our assumption
\eqref{eq:cond-alpha-beta} on the parameters. 
\epf

It remains an open question to determine when 
two quotients given by subgroup data are
isomorphic as Hopf algebras. The problem was solved for a
special case for quotients of $ \Oe $ at \cite{AG1} using 
algebraic geometry and homological tools. In view of 
this, it seems to be difficult to solve the problem with all
generality. This will be left for future research.

\bibliographystyle{amsbeta}

\end{document}